\documentclass[11pt,a4paper]{article}
\usepackage[english]{babel}
\usepackage{url}
\usepackage{enumerate}
\usepackage[a4paper,width=160mm,top=25mm,bottom=25mm]{geometry}
\usepackage{amssymb,amsmath,verbatim,graphicx}
\usepackage{graphicx}
\usepackage{xcolor}
\usepackage{paralist}
\usepackage{tikz,graphics}
\usetikzlibrary{fit,positioning,calc}
\usepackage{tcolorbox}
\usepackage{minitoc}

\usetikzlibrary{positioning}
\usepackage{tkz-euclide}
\parskip\medskipamount
\newtheorem{theorem}{Theorem}[section]

\newtheorem{prop}[theorem]{Proposition}
\newtheorem{lemma}[theorem]{Lemma}
\newtheorem{cor}[theorem]{Corollary}

\newtheorem{example}[theorem]{Example}
\newtheorem{fact}[theorem]{Fact}
\newtheorem{defi}[theorem]{Definition}
\newtheorem{main}[theorem]{Main Result}
\newtheorem{rem}[theorem]{Remark}

\newenvironment{proof}{\par\noindent\textbf{Proof}\hspace{1em}}{\qed}
\newenvironment{proof*}{\par\noindent\textbf{Proof}\hspace{1em}}{}
\def\<{\langle}
\def\>{\rangle}

\newcommand{\cC}{\mathcal{Cy}}

\newcommand{\PG}{\mathsf{PG}}
\newcommand{\cod}{\mathrm{codim}}

\newcommand{\ssH}{\mathbb{H}}

\newcommand{\A}{\mathbb{A}}
\newcommand{\K}{\mathbb{K}}

\renewcommand{\L}{\mathbb{L}}
\newcommand{\F}{\mathbb{F}}

\newcommand{\cP}{\mathcal{P}}

\newcommand{\cV}{\mathcal{V}}
\newcommand{\cL}{\mathcal{L}}

\newcommand{\CD}{\mathsf{CD}}

\newcommand{\ssO}{\mathbb{O}}
\newcommand{\PGL}{\mathsf{PGL}}
\newcommand{\AG}{\mathsf{AG}}

\newcommand{\PSL}{\mathsf{PSL}}

\newcommand{\kar}{\mathrm{char}\,}

\newcommand{\B}{\mathbb{B}}
\newcommand{\D}{\mathbb{D}}

\def\qed{{\hfill\hphantom{.}\nobreak\hfill$\Box$}}

\begin{document}

\author{Anneleen De Schepper\thanks{Supported by the Fund for Scientific Research - Flanders (FWO - Vlaanderen)} \and Hendrik Van Maldeghem\thanks{Partly supported by the Fund for Scientific Research - Flanders (FWO - Vlaanderen)}}
\title{Veronese representation of projective Hjelmslev planes over some quadratic alternative algebras}
\date{\footnotesize  Department of Mathematics: algebra and geometry,\\
Ghent University,\\
Krijgslaan 281-S25,\\
B-9000 Ghent,
BELGIUM\\
 \texttt{Anneleen.DeSchepper@UGent.be}\\  \texttt{Hendrik.VanMaldeghem@UGent.be}\\ $^*$Corresponding author}
\maketitle

\begin{abstract}
We geometrically characterise the Veronese representations of ring projective planes over algebras which are analogues of the dual numbers, giving rise to projective Hjelmslev planes of level 2 coordinatised over quadratic alternative algebras. These planes are related to affine buildings of relative type $\widetilde{A}_2$ and respective absolute type $\widetilde{\mathsf{A}}_2$, $\widetilde{\mathsf{A}}_5$ and $\widetilde{\mathsf{E}}_6$.
\end{abstract}

{\footnotesize
}
\setcounter{tocdepth}{2}

\tableofcontents

\section{Introduction}
In \cite{Sch-Mal:14}, Schillewaert and the second author study ring projective planes over quadratic 2-dimensional algebras over a field $\K$, in casu the split and non-split quadratic \'etale extensions, the inseparable extensions of degree 2 in characteristic 2 and the dual numbers. The authors give a neat and more or less uniform geometric characterisation of the Veronese representations of these planes, each as a point set $X$ in a projective space, equipped with a family of $3$-dimensional subspaces intersecting $X$ in a quadric of a certain type (depending on the algebra) and satisfying some axioms. The obtained geometries are the Segre varieties $\mathcal{S}_{2,2}$, the Hermitian Veronese varieties including those of inseparable type, and the Hjelmslev planes of level 2 over the dual numbers, respectively. The quadrics are hyperboloids, ellipsoids and quadratic cones, respectively. In \cite{Kra-Sch-Mal:15} and~\cite{Sch-Mal:17}, the former two have been extended to planes over \emph{higher-dimensional} non-degenerate quadratic alternative algebras over $\K$, which geometrically corresponds to using higher-dimensional quadrics, namely hyperbolic quadrics (nondegenerate quadrics of maximal Witt index) and elliptic quadrics (nondegenerate quadrics of Witt index 1).

In this paper, we extend and in fact complete their result in the direction of the dual numbers. The dual numbers over $\K$ are set-wise given by $\K[0]=\{k+tk' \mid k,k' \in \K\}$ where $t$ is an indeterminate with $t^2=0$. We extend this notion to ``non-split dual numbers'': quadratic algebras set-wise given by $\B[0]:=\{b+tb' \mid b,b' \in \mathbb{B}\}$, where $\B$ is quadratic associative division algebra over $\K$ (to make sure the resulting algebra is alternative) and $t$ is again an indeterminate with $t^2=0$ satisfying well defined commutation relations to keep the algebra quadratic.  It will turn out that these algebras also result from one application of the Cayley-Dickson process on a quadratic associate division algebra $\mathbb{B}$ over $\K$, using $0$ as a primitive element (an option standardly not considered), whence the notation $\B[0]$.

\subsection{The main result}
The key achievement of this paper is the axiomatic characterisation of the Veronese representations of the ring projective planes over the quadratic dual numbers described above. Roughly speaking, we prove: \begin{quote} \it Any set of points $X$ in projective space (over any field and of arbitrary dimension $N$, possibly infinite), equipped with a family of subspaces (also of arbitrary dimension $d+1$, possibly infinite), each intersecting $X$ in some degenerate quadric (minus the vertex) whose basis is an elliptic quadric, and such that every pair of points of $X$ is contained in such a quadric, and every pair of such quadrics share a point of $X$, essentially is the Veronese representation of a ring projective plane over quadratic dual numbers $\B[0]$, where $\B$ is a quadratic associative division algebra with $2\dim_\K(\B)=d$. \end{quote}

Note that the equality $2 \dim_\K(\B)=d$ yields severe restrictions for the values of $d$ for which such sets $X$ exist (if $\kar(\K) \neq 2$, it implies $d \in \{2,4,8\}$). We also remark that we will more generally work with \emph{ovoids} (see below for a precise definition) instead of elliptic quadrics; our proof will imply that these ovoids will automatically be quadrics.

 In order to give the extended version of the main theorem (cf.\ \textbf{Theorems~\ref{main3} and~\ref{main2}}), we first need to define the ring projective planes over $\B[0]$ and their Veronese representations formally, which is done in Section~\ref{CD}  after discussing the properties of the algebras $\B[0]$. 

\subsection{Motivation}
Our interest in these Veronese varieties originates from their link to certain affine buildings associated to the buildings of the second row of the Freudenthal-Tits magic square. More precisely, as point-line geometries, the obtained structures are related to the below pictured affine Tits indices (below this is also explained in more detail), see \cite{MPW} for the general theory of such Tits indices. Nonetheless, these Veronese varieties are worth studying in their own right as they exhibit geometrically interesting behaviour and substructures. An instance of the latter is given by the notion of a \emph{regular scroll}, which is a generalisation of a normal rational cubic scroll. 
We dedicate an appendix to their properties, as they will be essential in our characterisation.

\begin{table}[h]
\begin{center}
\begin{tabular}{ccc}
 \begin{tikzpicture}[scale=0.75,baseline=-0.5ex]
\node at (0,0.8) {};
\node at (0,-0.8) {};
\node [inner sep=0.8pt,outer sep=0.8pt] at (1,-0.432) (-1) {$\bullet$};
\node [inner sep=0.8pt,outer sep=0.8pt] at (0,-0.432) (0) {$\bullet$};
\node [inner sep=0.8pt,outer sep=0.8pt] at (0.5,0.432) (1) {$\bullet$};
\draw (0,-0.432)--(1,-0.432);
\draw (0.5,0.432)--(1,-0.432);
\draw (0,-0.432)--(0.5,0.432);
\draw [line width=0.5pt,line cap=round,rounded corners] (-1.north west)  rectangle (-1.south east);
\draw [line width=0.5pt,line cap=round,rounded corners] (1.north west)  rectangle (1.south east);
\draw [line width=0.5pt,line cap=round,rounded corners] (0.north west)  rectangle (0.south east);
\end{tikzpicture} 
&
 \begin{tikzpicture}[scale=0.75,baseline=-0.5ex]
\node at (0,0.8) {};
\node at (0,-0.8) {};
\node [inner sep=0.8pt,outer sep=0.8pt] at (0,-0.865) (1) {$\bullet$};
\node [inner sep=0.8pt,outer sep=0.8pt] at (1,-0.865) (2) {$\bullet$};
\node [inner sep=0.8pt,outer sep=0.8pt] at (-0.5,0) (6) {$\bullet$};
\node [inner sep=0.8pt,outer sep=0.8pt] at (1.5,0) (3) {$\bullet$};
\node [inner sep=0.8pt,outer sep=0.8pt] at (0,0.865) (5) {$\bullet$};
\node [inner sep=0.8pt,outer sep=0.8pt] at (1,0.865) (4) {$\bullet$};
\draw (0,-0.865)--(1,-0.865);
\draw (1,-0.865)--(1.5,0);
\draw (1.5,0)--(1,0.865);
\draw (1,0.865)--(0,0.865);
\draw (0,0.865)--(-0.5,0);
\draw (-0.5,0)--(0,-0.865);
\draw [line width=0.5pt,line cap=round,rounded corners] (6.north west) rectangle (3.south east);
\node [inner sep=0.5pt, line width=0.5pt,line cap=round, draw=black, rounded corners, rotate fit=60, fit=(1) (4)] {};
\node [inner sep=0.5pt, line width=0.5pt,line cap=round, draw=black, rounded corners, rotate fit=120, fit=(2) (5)] {};
\end{tikzpicture} 
&
\begin{tikzpicture}[scale=0.75,baseline=-0.5ex]
\node at (0,0.8) {};
\node at (0,-0.8) {};
\node [inner sep=0.8pt,outer sep=0.8pt] at (0,0) (-1) {$\bullet$};
\node [inner sep=0.8pt,outer sep=0.8pt] at (1,0) (0) {$\bullet$};
\node [inner sep=0.8pt,outer sep=0.8pt] at (2,0) (1) {$\bullet$};
\node [inner sep=0.8pt,outer sep=0.8pt] at (2.865,0.5) (2) {$\bullet$};
\node [inner sep=0.8pt,outer sep=0.8pt] at (3.73,1) (3) {$\bullet$};
\node [inner sep=0.8pt,outer sep=0.8pt] at (2.865,-0.5) (4) {$\bullet$};
\node [inner sep=0.8pt,outer sep=0.8pt] at (3.73,-1) (5) {$\bullet$};
\draw (0,0)--(2,0);
\draw (2,0)--(3.73,1);
\draw (2,0)--(3.73,-1);
\draw [line width=0.5pt,line cap=round,rounded corners] (0.north west)  rectangle (0.south east);
\draw [line width=0.5pt,line cap=round,rounded corners] (2.north west)  rectangle (2.south east);
\draw [line width=0.5pt,line cap=round,rounded corners] (4.north west)  rectangle (4.south east);
\end{tikzpicture} \\
$\widetilde{\mathsf{A}}_2$ & $\widetilde{\mathsf{A}}_5$ & $\widetilde{\mathsf{E}}_6$ 
\end{tabular}
\end{center}
\caption{The corresponding Tits indices (affine). \label{dualnonsplit}}
\end{table}


Our geometric characterisation is especially remarkable since the axioms we use are a straightforward extension of the elementary axioms used by Mazzocca and Melone \cite{Maz-Mel:84} to characterise the Veronese representation of all conics of a projective plane over a finite field of odd order (the simplest case in a finite setting), and those axioms also describe and characterise the exceptional Veronese map related to an octonion division algebra, owing its existence to the Tits index $\mathsf{E_{6,2}^{28}}$ of simple algebraic groups. The following two facts demonstrate the strength of our geometric approach:
\begin{itemize}
\item All our results hold over arbitrary fields and in arbitrary dimension (even infinite), except that we exclude the field of order 2 in our main result (this will be explained later on). 
\item Our approach is uniform in the sense that we also capture the \emph{non-degenerate }Veronese varieties, i.e., the Veronese representations of the Moufang projective planes over the quadratic alternative division algebras. 
\end{itemize}
Indeed, as a side result, our axioms provide a new characterisation of the above mentioned non-degenerate varieties,  except that over the field of order $2$ more examples pop up (still classifiable though), and interestingly, one example is strongly related to the large Witt design $S(24,5,8)$ and the sporadic Mathieu group $M_{24}$. These extra examples are so-called \emph{pseudo-embeddings} and our results complement in a surprising way the results of De Bruyn \cite{Bru:12, Bru:13}.

The Veronese variety of the ring projective plane over the dual numbers is related to a (Pappian) Hjelmslev plane of level 2, see \cite{Sch-Mal:14}. As was proved in \cite{Han-Mal:89}, these Hjelmslev planes are the spheres of radius 2 in certain affine buildings of type $\widetilde{\mathsf{A}}_2$ and can be related to the split form. Now, there are other quadratic alternative algebras coordinatising the spheres of radius 2 in certain affine buildings of type $\widetilde{\mathsf{A}}_2$ and one can also define a Veronese representation of them, which is what we do in this paper. The revealed patterns somehow show that the occurring Tits indices are special among their relatives. In a certain sense, they are the simplest given the absolute type. Indeed, the anisotropic kernels are minimal in the $\widetilde{\mathsf{E}}_6$ case, the building at infinity is octonionic and hence exceptional, whereas the residues are quaternion and hence classical.  Algebraically, our results imply that the rings coordinatising the spheres of fixed radius at least 2 are never quadratic, except if the radius is 2 in the cases of above. Note that the spheres of radius 2 play a helpful role in the theory of affine buildings, especially in the $\widetilde{\mathsf{A}}_2$ case, see e.g. \cite{Tit:90}, \cite{Wit:17}, and are therefore interesting structures to investigate.

\subsection{Structure of the paper}
In \textbf{Section~\ref{Overview}}, we formalise the axiomatic setting alluded to above and give a short version of our main result (cf.\ \textbf{Theorem~\ref{maininformal}}). 

In \textbf{Section~\ref{CD}} we introduce the algebras of our interest as ``degenerate quadratic alternative $\K$-algebras $\A$ whose maximal non-degenerate subalgebras are division algebras $\B$ and whose radical is generated by a single element $t$''. We show that these are precisely the non-split dual numbers as described above and that they can equivalently be obtained by one application of the generalised Cayley-Dickson process on a quadratic associative division algebra. For each of these $\K$-algebras $\A$, we then define a ring geometry $\mathsf{G}(2,\A)$ for which we can then consider its Veronese representation $\mathcal{V}_2(\K,\A)$. 

In \textbf{Section~\ref{main}} we can then state our main results formally. Doing so, we also provide a neat geometric description and construction of the Veronese varieties over non-split dual numbers: they fall apart, just like the corresponding algebra, in two isomorphic (in fact, dual) parts, one non-degenerate and one degenerate, in the sense that after projecting from the degenerate part, one obtains a non-degenerate Veronese variety consisting of points and non-degenerate quadrics of Witt index 1, and their exists a duality between these two parts. We also motivate our axioms.

The remaining sections are devoted to proving \textbf{Theorems~\ref{main3} and~\ref{main2}}, see Section~\ref{structuur}.

\par\bigskip
\begin{tcolorbox}
Throughout the entire paper, $\K$ denotes an arbitrary (commutative) field, unless explicitly mentioned otherwise. \end{tcolorbox}

\section{Ordinary and Hjelmslevean Veronese sets}\label{Overview}
 Let $d,v$ be elements of $\mathbb{N} \cup \{-1,\infty\}$ with $d \geq 1$.  
Vere
\subsection{Definitions}
An \emph{ovoid} $O$ in $\PG(d+1,\K)$ is a set of points spanning $\PG(d+1,\K)$ such that for each point $x\in O$ the union of the set of lines intersecting $O$ in $x$ is a hyperplane of $\PG(d+1,\K)$, called the \emph{tangent hyperplane at $x$}, and each other line through $x$ intersects $O$ in exactly one more point (hence $O$ does not contain triples of collinear points). 
For short we call this a $Q^0_d$-quadric, where the $0$ indicates the (projective) dimension of the maximal subspaces lying on $O$ and $d$ the (projective) dimension of its tangent hyperplanes. Examples are given by quadrics of Witt index $1$ (or, equivalently, projective index 0), explaining our notation. Moreover, in the setting we will consider, the ovoids will turn out to be quadrics, see Corollary~\ref{ovoids}. 

\begin{defi}[$(d,v)$-tubes]
\em
In a projective space $\PG(d+v+2,\K)$, we consider a $v$-space $V$ and a $Q^0_d$-quadric in a $(d+1)$-space complementary to $V$. The union of lines joining all points of $V$ with all points of $Q^0_d$ is called a \emph{$(d,v)$-cone} with \emph{base} $Q^0_d$ and \emph{vertex} $V$. The cone without its vertex is called a \emph{$(d,v)$-tube} (with base $Q^0_d$).
\end{defi}

A $(d,v)$-tube with base $Q^0_d$ will be referred to as a \emph{tube} whenever $(d,v)$ and $Q^0_d$ are clear from the context. 

Let $C$ be a $(d,v)$-tube. The unique $(d,v)$-cone containing $C$ is denoted by $\overline{C}$, so $\overline{C}=C \cup V$ for a certain $v$-space $V$. Even though $V$ is not contained in $C$, we also call $V$ the \emph{vertex} of $C$. A \emph{tangent line} to $C$  is a line which has either one or all its points in $\overline{C}$. Let $x$ be any point in  $C$. All lines through $c$ entirely contained in $\overline{C}$ are those contained in $\<x,V\>$. The latter subspace is called a \emph{generator} of $\overline{C}$ and of $C$. The union of the set of tangent lines to $C$ through $x$ is a hyperplane of $\<C\>$, denoted $T_x(C)$, which intersects $\overline{C}$ in the generator through~$x$. 

\subsection{A characterisation of Hjelmslevean and ordinary Veronese sets}

Consider a spanning point set $X$ of $\PG(N,\K)$, $N>d+v+2$,  together with  a collection $\Xi$ of $(d+v+2)$-dimensional projective subspaces of $\PG(N,\K)$, called the \emph{tubic spaces} of $X$, such that, for any $\xi \in \Xi$, the intersection $\xi \cap X$ is a $(d,v)$-tube $X(\xi)$ in $\xi$ with base $Q_d^0$. For $\xi\in\Xi$ and $C=X(\xi)$, we define $\Xi(C)=\xi$.
The union of all vertices of those tubes is denoted by $Y$. The unique $(d,v)$-cone containing $X(\xi)$ is denoted by $\overline{X(\xi)}$ as before.
Note that $X\cap Y = \emptyset$: if a point $x\in X$ belongs to the vertex $V$ of some tubic space $\xi$, then $X \cap \xi$ would strictly contain a $(d,v)$-tube, instead of being one.
We often denote $T_x(X(\xi))$ by $T_x(\xi)$ and we define the \emph{tangent space $T_x$ of $x$} as the subspace spanned by all tangent spaces through $x$ to all tubes through $x$, i.e., $T_x=\<T_x(\xi) \mid x\in \xi \in \Xi\>$.

\begin{defi}\em
The pair $(X,\Xi)$, or simply $X$, is called a \emph{Hjelmslevean Veronese set (of type $(d,v)$)} 
if the following properties hold.

\begin{itemize}
\item[(H1)] Any two distinct points $x_1$ and $x_2$ of $X$ lie in at least one element of $\Xi$,
\item[(H2$^*$)] Any two distinct elements $\xi_1$ and $\xi_2$ of $\Xi$ intersect in points of $X\cup Y$, i.e., $\xi_1\cap\xi_2 = \overline{X(\xi_1)} \cap \overline{X(\xi_2)}$. Moreover,
$\xi_1 \cap \xi_2 \cap X$ is non-empty.
\end{itemize}
\end{defi}

Then informally our main result reads:

\begin{theorem}\label{maininformal}
If $(X,\Xi)$ is a Hjelmslevean Veronese set, then  $X$ is projectively equivalent to a cone over the image of the Veronese map on certain triples of elements of a quadratic alternative algebra whose maximal non-degenerate part is a division algebra and whose radical $R$ is either trivial or a principal ideal.
\end{theorem}

When $\kar(\K)\neq 2$, we essentially (i.e., when only considering the basis of the above cone) obtain seven geometries, four of which have $v=-1$ (these are ordinary Veronese sets) and three of which have $v=d-1 \geq 0$ (actual Hjelmslevean Veronese sets). In the next section we explain this Veronese map and these algebras in more detail, so that in the section thereafter we can state a precise version of the main result, providing in detail which Veronese varieties we obtain.

\section{Non-split dual numbers and associated Veronese representations}\label{CD}

We introduce \emph{(degenerate) quadratic alternative $\K$-algebras $\A$}, then restrict our attention to these algebras whose \emph{radical} is non-trivial but ``small'' (generated, as a ring, by a single element) and show that these can be produced by applying the \emph{Cayley-Dickson doubling process on $\K$}. Each of these notions will be explained too. By abuse of notation we will assume that if a $\K$-algebra is unital, it contains $\K$ (and all our algebras will be unital). Afterwards, we define a ring projective plane over them and consider its Veronese representation.

\subsection{Quadratic alternative algebras over $\K$}
Suppose $\A$ is a unital $\K$-algebra which  is \emph{alternative} (i.e., the associator $[a,b,c]:=(ab)c-a(bc)$ is a $\K$-trilinear alternating map) and  \emph{quadratic} (i.e., each $a\in \A$ satisfies the quadratic equation $x^2-\mathsf{T}(a)x+\mathsf{N}(a)=0$ where the \emph{trace} $\mathsf{T}: \A \rightarrow \K: a \mapsto \mathsf{T}(a)$ is a $\K$-linear map with $\mathsf{T}(1)=2$ and the \emph{norm} $\mathsf{N}: \A \rightarrow \K: a \mapsto \mathsf{N}(a)$ is a $\K$-quadratic map with $\mathsf{N}(1)=1$). The latter property induces an involution $x \mapsto \overline{x}$ on $\A$ fixing $\K$, defined by taking $a\in \A$ to $\mathsf{T}(a)-a$ (to be precise, this is an involutive anti-automorphism, so $\overline{xy}=\overline{y}\,\overline{x}$). Consequently, $\mathsf{N}(a)=a\overline{a}$ for each $a\in \A$. Note that an element $a\in \A$ is invertible if and only if $\mathsf{N}(a)\neq 0$, and if invertible, its inverse is given by $\mathsf{N}(a)^{-1}\overline{a}$. Hence the algebra $\A$ is a \emph{division algebra} if the norm form $\mathsf{N}$ is anisotropic (i.e., if for all $x\in \A$, $\mathsf{N}(x)=0 \Leftrightarrow x =0$).

\subsubsection{The radical of $\A$}
The bilinear form $f$ associated to the quadratic form $\mathsf{N}$ is given by $f(x,y)=\mathsf{N}(x+y)-\mathsf{N}(x)-\mathsf{N}(y)=x\overline{y}+y\overline{x}$. Its \emph{radical} is the set  $\mathsf{rad}(f)=\{x \in \A \mid f(x,y)=0, \, \forall y \in \A\}$. 
The algebra $\A$ is \emph{non-singular} if $f$ has a trivial radical; it is \emph{non-degenerate} if its norm form $\mathsf{N}$ is non-degenerate, which means that $\mathsf{N}$ is anisotropic over $\mathsf{rad}(f)$, i.e., the set $R:=\{r \in \mathsf{rad}(f) \mid \mathsf{N}(r)=0\}=\{0\}$. So $\A$ is non-degenerate if and only if $R$ is trivial. The set $R$ is hence called  the \emph{radical} of $\A$.

Both $\mathsf{rad}(f)$ and $R$ are two-sided ideals of $\A$.  If $\kar{\K} \neq 2$, then for all $r\in \mathsf{rad}(f)$ we have $0=f(r,r)=2\mathsf{N}(r)$, so $R=\mathsf{rad}(f)$. In general, one can show that either $R=\mathsf{rad}(f)$ or $\mathsf{rad}(f)=\A$. Moreover, in the latter case, it follows that $\kar{\K}=2$, $\A$ is a commutative associative $\K$-algebra with $\A^2 \subseteq \K$ and $x \mapsto \overline{x}$ is trivial. 

\begin{rem}
\em One can also define $R$ without referring to the forms $f$ and $\mathsf{N}$ on $\A$: the radical $R$ is precisely the set of elements $r$ such that $ar$ is nilpotent for each $a\in \A$.
\end{rem}

The following fact is well known (see for instance \cite{Kap:53}).
 \begin{fact}\label{qad} Let $\A$ be a non-degenerate quadratic alternative division algebra over a field $\K$ and put $d=\dim_\K(\A)$. Then $\A$ is one of the following.
\begin{compactenum}
\item[$(d=1)$] $\A=\K$;
\item[$(d=2)$] $\A$ is a quadratic Galois extension $\L$ of $\K$;
\item[$(d=4)$] $\A$ is a quaternion division algebra $\ssH$ with center $\K$;
\item[$(d=8)$] $\A$ is a Cayley-Dickson division algebra $\ssO$ with center $\K$;
\item[\emph{(insep)}] $\A/\K$ is a purely inseparable extension with $\A^2 \subseteq \K$ and, if finite, then $d:=[\A:\K]$ can be any power of 2. In this case $\mathrm{char}(\K)=2$.
\end{compactenum}
\end{fact}

Except for the case $(d=2)$ in characteristic 2, these algebras are precisely the ones obtained from subsequent applications of the standard Cayley-Dickson process on $\K$ (see below). 

The following fact has been shown for fields $\K$ with characteristic different from 2 in \cite{Kun-Sch:86}, but one can verify that it also holds for characteristic 2, using the technique of the proof of Fact~\ref{qad} (see also Section 2.6, Theorems 2.6.1 and 2.6.2, on pages 164--166 of
McCrimmon's book \cite{McC:04}).

\begin{fact}\label{direct+} Let $\A$ be a possibly degenerate quadratic alternative unital $\K$-algebra with radical $R$. Then $\A$ contains a unital subalgebra $\B$ maximal with the property of being non-degenerate; and for such $\B$ we have $\B^\perp=R$ and $\A=\B \oplus R$. 
\end{fact}

Our current interest lies in the quadratic alternative algebras $\A$ where  $\B$ is a division algebra and $R$ is generated (as a ring) by a single element $t \in \A\setminus\{0\}$. 

\subsubsection{The extended Cayley-Dickson process}
An extended version of the  \emph{Cayley-Dickson process} produces the above mentioned algebras. We explain this process, briefly but yet in full generality. Let $\A$ be a quadratic alternative $\K$-algebra with associated involution $a\mapsto \overline{a}$ as before, and $\zeta$ some element in $\K$. 

\textbf{The Cayley-Dickson process applied to the pair $(\A,\zeta)$}|We define a $\K$-algebra, which we denote by $\mathsf{CD}(\A,\zeta)$, as follows. We endow the $\K$-vector space $\A \times \A$ with component-wise addition and multiplication given by \[(a,b)\times(c,d)=(ac+ \zeta d\overline{b}, \overline{a}d +cb).\] The resulting $\K$-algebra is quadratic too and hence comes with a standard involution: $\overline{(a,b)}:=(\overline{a},-b)$ and a norm $\mathsf{N}(a,b)=(a,b)\cdot \overline{(a,b)}= (\mathsf{N}(a)-\zeta \mathsf{N}(b),0)$. 
The algebra $\mathsf{CD}(\A,\zeta)$ is a division algebra if and only if $\zeta \notin \mathsf{N}(\A)$ and $\A$ is division. It is non-degenerate if and only if $\zeta \neq0$ and $\A$ is non-degenerate.  

\begin{rem} \em The fact that we allow $\zeta=0$ is the point at which the above process extends the standard one. If $\A$ is non-degenerate, then the radical of $\CD(\A,0)$ is given by $t\A$ and is hence generated by the element $t$.
\end{rem}

\textbf{The Cayley-Dickson process starting from $\K$}|If one starts with $\K$ (note that the associated involution is the identity now) and some element $\zeta \in \K$, then $\mathsf{CD}(\K,\zeta)$ is given by the  quotient $\K[x]/(x^2-\zeta)$. In particular, we obtain a commutative associative $\K$-algebra with involution $a+tb\mapsto a-tb$, where $t^2=\zeta$. The involution is non-trivial precisely if $\kar(\K)\neq 2$. 

\begin{itemize}
\item \textit{Suppose first that $\kar(\K)\neq 2$.} Another application, for some $\zeta'\in \K$, gives us $\mathsf{CD}(\K,\zeta,\zeta')$ which is an associative quaternion algebra with center $\K$ (hence no longer commutative). A third application yields an octonion algebra $\mathsf{CD}(\K,\zeta,\zeta',\zeta'')$ which is no longer associative, but alternative instead. A fourth application would not yield an alternative algebra, so here it stops for us.

\item \textit{Next, suppose $\kar(\K)=2$.} Since the $\K$-algebra obtained after one step has a trivial involution, also after $n$ applications (with $n \in \mathbb{N}$) we get a  commutative associative $\K$-algebra $\mathsf{CD}(\K,\zeta_1,...,\zeta_d)$, which equals $\K[x_1,...,x_d]/(x_i^2=\zeta_i)$, and is briefly denoted by $\A_d$, and $\dim_\K(\A_n)=2^n$.  
\end{itemize}

We can adapt the first step of the Cayley-Dickson doubling process in such a way that we obtain a non-trivial involution, by considering the quotient of $\K[x]$ by the ideal $(x^2+x+\zeta)$ (instead of $(x^2+\zeta)$) for some $\zeta \in \K$. Afterwards we can go on with the ordinary Cayley-Dickson process and reach strictly alternative algebras.


\begin{rem}\label{a+tb} \em We can also view $\mathsf{CD}(\A,\zeta)$ as the extension $\A \oplus t\A:=\{a+tb \mid a,b \in \A\}$ of $\A$ using the indeterminate $t$ with defining multiplication rules (putting $s=1$ precisely if $\kar(\K)=2$ and the involution on $\A$ is trivial and we want to obtain a non-trivial involution; otherwise $s=0$): $t^2 =\zeta+st$ and, for all  $a,b,c,d \in \A$,
 \[\label{mult}\tag{$*$} a(td)=t(\overline{a}d),\qquad (tb)c=t(cb), \qquad (tb)(td)=t^2(d\overline{b}),\] where we use the same notation as above, taking into account the isomorphism $(a,b)\mapsto a+tb$. 
\end{rem}

We now reach the fact that we alluded to, which can also be proven using the technique of the proof of Fact~\ref{qad} and relying on Fact~\ref{direct+}.

\begin{fact}\label{sqad} Let $\A$ be a degenerate quadratic alternative unital $\K$-algebra whose radical $R$ is generated as a ring by a single element $t\in \A\setminus\{0\}$. Let $\B$ be a unital subalgebra of $\A$ maximal with the property of being non-degenerate. Then $\B$ is associative and if $\B$ is moreover a division algebra, we have $\A \cong \mathsf{CD}(\B,0)=\B \oplus t\B$ (with multiplication rules as in $(*)$).
\end{fact}

\subsection{The Veronese representations of planes over $\B$ and $\mathsf{CD}(\B,0)$}

Let $\A$ be either $\mathsf{CD}(\B,0)$ or $\B$ itself, where $\B$ is a quadratic associative division algebra and  $t$ an element generating the ideal $R$ of $\mathsf{CD}(\B,0)$. Note that $t^2=0$. By Remark~\ref{a+tb}, a pair $(a_0,a_1) \in \mathsf{CD}(\B,0)$ can be written uniquely as a sum $a_0 + ta_1$, with $a_0,a_1 \in \B$. Note  that, when $t$ is involved in a multiplication, then it is associative (since $t^2=0$), i.e., $t((xy)z)=t(x(yz))$, which we can then write as $t(xyz)$, for all $x,y,z\in\A$. 

To study $\B$, we pretend $t=0$. For each element $a=a_0+ta_1\in\mathsf{CD}(\B,0)$, with $a_0,a_1\in\B$, we write $\widetilde{a}=a_0$. This extends logically to $\widetilde{a}=a$ when working in $\B$. Like before, put $d=\dim_\K(\A)$. We now use these algebras to define ring geometries coordinatised over them, equipped with a \emph{neighbouring relation}.

\begin{defi} \em  The point-line geometry $\mathsf{G}(2,\A):=(\mathcal{P},\mathcal{L})$ is defined as follows:
\begin{compactenum}[$-$]
\item $\mathcal{P} = \{(x,y,1) \mid x,y \in \A\} \cup \{(1,y,tz_1) \mid z_1 \in \B, y\in \A\} \cup \{(tx_1,1, tz_1) \mid x_1,z_1 \in \B\}$
\item $\mathcal{L} = \{[a,1,c] \mid a,c \in \A\} \cup \{[1,tb_1,c] \mid  b_1 \in \B, c\in \A\} \cup \{[ta_1,tb_1, 1] \mid a_1,b_1 \in \B\}$
\item A point $(x,y,z)$ is \emph{incident} with a line $[a,b,c]$ if and only if $ax+by+cz=0$.
\end{compactenum}
The neighbouring relation $\approx$ on $(\mathcal{P} \cup \mathcal{L}) \times (\mathcal{P} \cup \mathcal{L})$ is defined as follows.  Two points $(x,y,z)$ and $(x',y',z')$ are neighbouring if $(\widetilde{x},\widetilde{y},\widetilde{z})=(\widetilde{x'},\widetilde{y'},\widetilde{z'})$.  Likewise for two lines. A point $(x,y,z)$ and a line $[a,b,c]$ are called neighbouring if $ax+by+cz\in t\B$. 
\end{defi}
 One can verify that a point $P$ and a line $L$ are not neighbouring if and only if $P$ and $Q$ are not neighbouring for each point $Q$ on $L$.

\begin{rem}
\em
If $\A$ is associative, we could also use the homogenous point set $\{(x,y,z)r \mid x,y,z,r\in \A \text{ with } \neg(\mathsf{N}(x)=\mathsf{N}(y)=\mathsf{N}(z)=0) \text{ and } \mathsf{N}(r) \neq 0\}$ and dually, the homogenous line set  $\{s[a,b,c] \mid a,b,c,s\in \A \text{ with } \neg(\mathsf{N}(a)=\mathsf{N}(b)=\mathsf{N}(c)=0) \text{ and } \mathsf{N}(s) \neq 0\}$. The lack of associativity prevents us from doing this, as scalar multiples are not well-defined.\end{rem}

When working with $\B$, i.e., if we let $t$ be 0, the algebra is associative and from the homogenous description above it is then clear that $\mathsf{G}(2,\B)$ is just $\mathsf{PG}(2,\B)$ and the neighbouring relation coincides with equality (between elements of the same type) or incidence (between points and lines). On the other hand, if $\B <\A$, ie., if we are working with $\mathsf{CD}(\B,0)$, then $\mathsf{G}(2,\A)$ is a projective Hjelmslev plane of level 2, with above neighbouring relations.

\begin{defi} \em The Veronese representation $\mathcal{V}_2(\K,\A)$ of $\mathsf{G}(2,\A)$ is the point-subspace structure $(X,\Xi)$ defined by means of the Veronese map
\[ \rho: \mathsf{G}(2,\A) \rightarrow \PG(3d+2,\K): (x,y,z) \mapsto \K(x\overline{x}, y\overline{y}, z\overline{z}; y\overline{z}, z\overline{x}, x\overline{y})\] 
by setting $X=\{\rho(p) \mid p \in \mathcal{P}\}$ and $\Xi=\{\<\rho(L)\>\mid L \in \mathcal{L}\}$, where $\rho(L)$ is defined as $\{\rho(p) \mid p \in L\}$ and incidence is given by containment made symmetric.  
\end{defi}

Note that, for $p \in \mathcal{P}$, $\rho(p) \in \PG(3d+2,\K)$ indeed: the six entries correspond to $d$-tuples over $\K$ and the first three first belong to $\K$, being norms.
In the next section we discuss the geometric structure of a line and the features of this geometry, after studying transitivity properties.

\section{Main results}\label{main}

As mentioned before, our main theorem states that the Hjelmslevean Veronese sets are essentially the above defined Veronese varieties.

\subsection{The precise statements}

\begin{main}\label{main3}
Suppose $(X,\Xi)$ is a Hjelmslevean Veronese set of type $(d,v)$ 
such that $X$ generates $\PG(N,\K)$, where $\K$ is a field with $|\K|>2$.  Then $d$ is a power of $2$, with $d \leq 8$ if $\kar(\K) \neq 2$, and one of the following holds.
\begin{enumerate}[$(i)$]
\item There is only one vertex $V$ and projected from $V$\!, the resulting point-subspace geometry $(X',\Xi')$ is projectively equivalent to $\mathcal{V}_2(\K,\B)$, where $\B$ is a quadratic alternative division algebra over $\K$ and, in  particular, $N=3d+v+1$ and $d=\dim_\K(\B)$;
\item There is a quadratic associative division algebra $\B$ over $\K$ and two complementary subspaces $U$ and $W$ of $\PG(N,\K)$, where $U$ is possibly empty and $\dim W=6d+2$, with $d=v-\dim(U)=2\dim_\K(\B)$, such that the intersection of every pair of distinct vertices is $U$, and the structure of $(X,\Xi)$ induced in $W$ is projectively equivalent to $\mathcal{V}_2(\K,\mathsf{CD}(\B,0))$. 
\end{enumerate}
In particular, for each $\xi \in \Xi$, the basis of the tube $X\cap\xi$ is  always a quadric.
\end{main}

The proof of the theorem will reveal the geometric structure of the Hjelmslevean Veronese sets.

\begin{cor}\label{cor}
Let $\mathcal{V}_2(\K,\mathsf{CD}(\B,0))=(X,\Xi)$ be the Hjelmslevean Veronesean, where $\B$ is a quadratic associative division algebra over $\K$ with $\dim_\K(\B)=d$. Then $X$ spans a projective space $\mathbb{P}= \PG(6d+2,\K)$, each vertex of a quadric in a member of $\Xi$ has dimension $d-1$ and the vertices form a regular spread  $\mathcal{S}$ of a $(3v+2)$-space in $\mathbb{P}$, and there exists a complementary space $F$ of $\mathbb{P}$ such that $(X\cap F,\{\xi\cap F\mid\xi\in\Xi\})=:(X',\Xi')$ is projectively equivalent to $\mathcal{V}_2(\K,\B)$. Moreover, if $\mathcal{B}$ is the set of  $(2v+1)$-spaces each spanned by two distinct vertices, then $(\mathcal{S},\mathcal{B})$, with natural incidence, is a projective plane isomorphic to $\PG(2,\B)$ and there is a linear duality $\chi$ between $(X',\Xi')$ and $(\mathcal{S},\mathcal{B})$ such that $X$ is the union of the subspaces $\<x',\chi(x')\>$, for $x'$ ranging over $X'$.
\end{cor}

A special case of Main Result~\ref{main3} is the case $v=-1$. In this case, we can lift the assumption $|\K|>2$, if we add a few more possibilities. Since this is interesting in its own right, we phrase it explicitly.

\begin{main}\label{main2}
Let $(X,\Xi)$ be a pair where $X$ is a spanning point set of $\PG(N,\K)$ and $\Xi$ a set of at least two $(d+1)$-spaces  intersecting $X$ in  $Q^0_d$-quadrics ($N$ and $d$ are possibly infinite), satisfying axioms \emph{(MM1)} and \emph{(MM2$^*$)}. Then, as a point-line geometry, $(X,\Xi)$ (with natural incidence) is isomorphic to $\PG(2,\A)$ where $\A$ is a quadratic alternative division algebra $\A$ over $\K$ with $\dim_\K(\A)=d$. Moreover, 
\begin{itemize}
\item If $|\K|>2$, $(X,\Xi)$ is projectively equivalent to $\mathcal{V}_2(\K,\A)$, and in particular, $N=3d+2$;
\item If $|\K|=2$, then either $d=1$ or $d=2$. \begin{compactenum}[$-$] \item If $d=1$, then $N \in \{5,6\}$. If $N=5$, there are precisely two projectively non-isomorphic examples, among which $\mathcal{V}_2(\F_2,\F_2)$; if $N=6$, there is a unique possibility. \item If $d=2$, then $N \in \{8,9,10\}$. If $N=10$, then there is precisely one example; in the other two cases there are precisely two projectively unique examples, among which is $\mathcal{V}_2(\F_2,\F_4)$, if $N=8$.
\end{compactenum}
\end{itemize}
\end{main}
 
 For the structure of the additional examples in case $|\K|=2$, and their relation with the Witt design $S(24,5,8)$, we refer to Subsection~\ref{k=2}.

\subsection{A note on the axioms and requirements}

\textbf{The (H2$^*$)-axiom}|In \cite{Sch-Mal:14}, J.\ Schillewaert and the second author showed a similar theorem in the specific case of $(d,v)=(1,0)$, though using slightly different axioms: their first axiom is the same, but their second is weaker and a third axiom was used.
\begin{itemize}
\item[(H2)] Any two distinct elements $\xi_1$ and $\xi_2$ of $\Xi$ intersect in points of $X \cup Y$, i.e., $\xi_1 \cap \xi_2 = \overline{X}(\xi_1) \cap \overline{X}(\xi_2)$, and $\xi_1 \cap \xi_2 \cap Y$ is either empty or a subspace of $\xi_1 \cap \xi_2$ of codimension 1.
\item[(H3)] For each $x \in X$, $\dim(T_x) \leq 4$.
\end{itemize}
The difference in the second axiom lies in our requirement that the intersection of two tubic spaces is never empty. If we now, for the existing example in which $(d,v)=(1,1)$, would require $\dim(T_x) \leq 6$ however, then together with (H1) and (H2) we would obtain the following example, which does not fit in our framework and in which two tubic spaces can be entirely disjoint.

\begin{example} \em Inside $\PG(13,\mathbb{K})$,  take a $3$-space $\Pi_Y$ and a $9$-space $F$ complementary to it. Inside $F$, we consider the Veronese representation $\mathcal{V}_3(\K,\K)$ of the projective space $\PG(3,\K)$ (defined analogously as $\mathcal{V}_2(\K,\K)$). Let $\chi$ be a linear duality between $\mathcal{V}_3(\K,\K)$ and  $\Pi_Y$, i.e., $\chi$ takes points of $\mathcal{V}_3(\K,\K)$ to planes of $\Pi_Y$ and conversely. We define $X$ as the points  on the affine $3$-spaces $\<c,\chi(c)\>\setminus\chi(c)$, where $c$ is a point of $\mathcal{V}_3(\mathbb{K},\K)$. It is clear that $X$ is a spanning point set of $\PG(13,\K)$. 

Each tube $X(\xi)$ with $\xi \in \Xi$ is associated to a unique conic $C$  of $\mathcal{V}_3(\mathbb{K},\mathbb{\K})$ in the following way: the vertex of $X(\xi)$ is $\chi(C)$ and, projected from this vertex, the resulting conic on $X(\xi)$ lies on the regular scroll\footnote{See appendix} determined by $C$, the regular spread corresponding to $\{\chi(c) \mid c \in C\}$ and $\chi$. As can be verified, thus obtained pair $(X, \Xi)$ consists of $(1,1)$-tubes and satisfies axioms (H1), (H2) and (H3) (the latter adapted to $(d,v)=(1,1)$, so $\dim(T_x) \leq 6$).
\end{example}

A similar example exists when $(d,v)=(2,3)$, using $\mathcal{V}_3(\K,\L)$ instead of $\mathcal{V}_3(\K,\K)$, where $\L$ is a quadratic Galois extension of $\K$. There are two options to avoid these examples: either we ask that for each $x\in X$, there exist two tubic spaces $\xi_1$ and $\xi_2$ containing $x$, for which $T_x$ is generated by $T_x(\xi_1)$ and $T_x(\xi_2)$ (which is not the case in the above example), and keep (H1) and (H2) as they are; or we make (H2) stronger by requiring that each two tubic spaces have a point of $X$ in common and we leave out (H3) or any variation of it. We opted for the latter set of axioms. Yet, with hindsight, we do not exclude that also the former would eventually lead to the same characterisation.


\par\bigskip

\textbf{The case $|\K|=2$}|The restriction $|\K|>2$ is not necessary in Case $(i)$ of Main Result~\ref{main3}, since Main Result~\ref{main2} also deals with fields with two elements and no more is needed that depends on $|\K|$. So in this case, the varieties that one obtains after projecting from the single vertex $V$ are in fact as listed in Main Result~\ref{main2}. 

However, when there is more than one vertex, our method breaks down. The reason is essentially that a singular affine line contains just as many points as a secant projective line (and so we cannot distinguish between these). Though we did not succeed in finding counter examples, we do not believe a proof in this case is within reach (noting that, in principle, the vertices could also have infinite dimension). However, equipped with Axiom (H3) and the extra assumption that tubes are convex (which is automatic if $|\K|>2$), the first author could classify the objects for $d=1$. The arguments, however, are not useful for the rest of this paper, and are also too lengthy to include here (the proof will appear in the first author's Ph.D.~thesis; the case $d=2$ is expected to be similar, but even more cumbersome).   

\subsection{Structure of the proof}\label{structuur}

In \textbf{Section~\ref{properties}}, we deduce transitivity properties of Veronese representations $\mathcal{V}_2(\K,\B[0])$, enabling us to show that they are Hjelmslevean Veroneseans indeed. These transitivity properties could also be proved by considering the appropriate affine building, and a vertex-stabiliser. Such an approach would need more elaborate notions and we prefer not to do it this way, but provide an elementary and explicit proof (though omitting tedious computations). The main part of the proof focusses on the reverse direction, showing that there are no other Hjelmslevean Veronese sets than these.

In \textbf{Section~\ref{reduction}},  we start by reducing the situation to two separate cases: the case that there are no degenerate quadrics (in which case the corresponding algebras are division) and the case that all quadrics are degenerate (here the corresponding algebras possess a non-trivial radical). 

In \textbf{Section~\ref{MM}}, we deal with the first of the two above cases and as such we provide an alternative approach to the Veronese representation of projective planes over quadratic alternative division algebras. A large part is devoted to the case of the field of order 2 and, although it is not essential for the rest of the paper and is of a different flavour (being strictly finite), it does reveal a beautiful link with the large Witt design.

Finally, \textbf{Section~\ref{maincase}} treats the proof in the case that all quadrics are degenerate. Basically our approach amounts to a study of the structure induced on the set of vertices (which forms a subspace, say $Y$); the structure of the points of $X$ after projecting from $Y$ and the relation between $X$ and $Y$.  The regular scrolls alluded to before play a crucial role in this, in the sense that they are a restriction of the entire structure to all the quadrics having the same vertex.

\section{Properties of the Veronese variety $\mathcal{V}_2(\K,\A)$}\label{properties}
We start by showing that the Veronese varieties  $\mathcal{V}_2(\K,\A)$ satisfy axioms (H1) and (H2$^*$). We continue with the same notation as in Section~\ref{CD}.

\textbf{The induced action of the collineation group of $\mathcal{V}_2(\K,\A)$}|The geometry $\mathsf{G}(2,\A)$ is a Moufang projective plane if $\A$ is division, and we assert that it is a \emph{Moufang Hjelmslev plane of level }2 if $\A$ is not division. However, most Hjelmslev planes with the Moufang property studied in the literature are \emph{commutative} extensions of Moufang projective planes, i.e., if the underlying projective plane is defined over the not necessarily associative alternative division ring $\D$, then the Hjelmslev plane is defined over the ring $\D[t]/(t^n=0)$, where $t$ is an indeterminate that commutes with each element of $\D$. Hence we provide a full proof of the above stated assertion (suppressing tedious calculations), using standard methods (we need the explicit forms of certain collineations anyway in the proof of Proposition~\ref{actiononveronesean}). The fact that  $\mathsf{G}(2,\A)$ is a Hjelmslev plane of level 2 if $\A$ is not division, is proved in Section~\ref{sectionHP}, where one also can find the precise definition of that notion. We now concentrate on the Moufang property. 

A \emph{collineation} of $\mathsf{G}(2,\A)=(\cP,\cL)$ is a permutation of $\cP\cup\cL$ preserving both $\cP$ and $\cL$ and preserving the incidence relation.  An \emph{elation} of $\mathsf{G}(2,\A)$ is a collineation that fixes all points on a certain line $L$|called the axis| and all lines incident with a certain point $P$|called the center|with $P*L$ (the pair $\{P,L\}$ is called a \emph{flag}). Such an elation is, with this notation, sometimes also called a \emph{$(P,L)$-elation}. The geometry $\mathsf{G}(2,\A)$ is called \emph{$(P,L)$-transitive}, for $P\in\cP$ and $L\in\cL$, with   $P*L$,   if for some line $M*P$, with $M\not\approx L$, the group of $(P,L)$-elations acts transitively on the set of points of $M$ not neighbouring $P$. Then $\mathsf{G}(2,\A)$ has \emph{the Moufang property}, or \emph{$\mathsf{G}(2,\A)$ is a Moufang (projective or Hjelmslev) plane}, if for every point $P$ and every line $L$ incident with $P$ the plane is $(P,L)$-transitive. It is well known and easy to see that this is equivalent with the existence of a triangle $P_0*L_1*P_2*L_0*P_1*L_2*P_0$, with $P_i\not\approx L_i$, $i=0,1,2$, such that $\mathsf{G}(2,\A)$ is $(P_i,L_j)$-transitive for $i\neq j$ and $i,j\in\{0,1,2\}$, because the collineation group generated by the $(P_i,L_j)$-elations, $i\neq j$, $\{i,j\}\subseteq\{0,1,2\}$, acts transitively on the set of flags. 

The collineation group of  $\mathsf{G}(2,\A)$ generated by all elations is called its \emph{little projective group} and shall be denoted by $\mathsf{PSL}_3(\A)$.

\begin{lemma}\label{G2AMoufang} The plane $\mathsf{G}(2,\A)$ is Moufang.
\end{lemma}
\begin{proof}
Indeed, the mappings (using the notation as above, and with $X,Y\in \A$ arbitrarily)

\[\varphi_{23}(Y):(\cP,\cL)\longrightarrow (\cP,\cL):\left\{\begin{array}{rcl}
(x,y,1) & \mapsto & (x,y+Y,1),\\
(1,y,tz_1) & \mapsto & (1,y-t\overline{Y}z_1, tz_1), \\
(tx_1,1,tz_1) & \mapsto & (tx_1,1,tz_1), \\ \hline
\mbox{}[a,1,c] & \mapsto & {[}a,1,c-Y],\\
\mbox{}[1,tb_1,c] & \mapsto & {[}1,tb_1,c-tYb],\\
\mbox{}[ta_1,tb_1,1] & \mapsto & {[}ta_1,tb_1,1]\end{array}\right.$$
and
$$\varphi_{13}(X):(\cP,\cL)\longrightarrow (\cP,\cL):\left\{\begin{array}{rcl}
(x,y,1) & \mapsto & (x+X,y,1),\\
(1,y,tz_1) & \mapsto & (1,y-t\overline{X}\overline{y}z_1, tz_1), \\ 
(tx_1,1,tz_1) & \mapsto & (tx_1+t\overline{X}z_1,1,tz_1), \\ \hline
\mbox{}[a,1,c] & \mapsto & {[}a,1,c-aX],\\
\mbox{}[1,tb_1,c] & \mapsto & {[}1,tb_1,c-X],\\
\mbox{}[ta_1,tb_1,1] & \mapsto & {[}ta_1,tb_1,1]\end{array}\right.\]

are $((0,1,0),[0,0,1])$-elations and $((1,0,0),[0,0,1])$-elations, respectively, and by varying $Y$ and $X$ we obtain $((0,1,0),[0,0,1])$-transitivity and $((1,0,0),[0,0,1])$-transitivity, respectively. Moreover, the triality map 

\[\tau:(\cP,\cL)\longrightarrow (\cP,\cL):\left\{\begin{array}{rcll}
(x,y,1) & \mapsto & (y^{-1},xy^{-1},1), & \mbox{if }y\in\A\setminus t\B,\\
(x,y,1) & \mapsto & (1,x,ty_1), & \mbox{if }t\B\ni y = ty_1, y_1\in\B,\\
(1,y,tz_1) & \mapsto & (t(y^{-1}z_1),y^{-1},1), & \mbox{if }y\in\A\setminus t\B, \\
(1,y,tz_1) & \mapsto & (tz_1,1,ty_1), &  \mbox{if }t\B\ni y = ty_1, y_1\in\B, \\
(tx_1,1,tz_1) & \mapsto & (tz_1,tx_1,1), \\ \hline
\mbox{}[a,1,c] & \mapsto & {[}a^{-1}c,1,a^{-1}] & \mbox{if }a\in\A\setminus t\B,\\
\mbox{}[a,1,c] & \mapsto & {[}1,t(\overline{c}^{-1}a_1),c^{-1}], & \mbox{if } c\in\A\setminus t\B\mbox{ and } t\B\ni a=ta_1,\\
\mbox{}[a,1,c] & \mapsto & {[}tc_1,ta_1,1], & \mbox{if }t\B\ni a=ta_1\mbox{ and }t\B\ni c=tc_1\\
\mbox{}[1,tb_1,c] & \mapsto & {[}c,1,tb_1],\\
\mbox{}[ta_1,tb_1,1] & \mapsto & {[}1,ta_1,tb_1]\end{array}\right.\]

preserves incidence, as one can easily check, and is bijective with inverse $\tau^2$. Conjugating $\varphi_{23}(Y)$ and $\varphi_{13}(X)$ with $\tau$ and $\tau^2$ shows that $\mathsf{G}(2,\A)$ is $((0,0,1),[1,0,0])$-transitive, $((0,1,0),[1,0,0])$-transitive, $((1,0,0),[0,1,0])$-transitive and $((0,0,1),[0,1,0])$-transitive. Hence $\mathsf{G}(2,\A)$ is a Mouf\-ang Hjelmslev plane, as claimed. 
\end{proof}

\begin{prop}\label{actiononveronesean}
The action of the little projective group $\mathsf{PSL}_3(\A)$ of $\mathsf{G}(2,\A)$ on $\cP$ is induced by the action on $X$ of the stabiliser in $\mathsf{PSL}_{3d+3}(\K)$ of the point set $X$ of $\cV_2(\K,\A)$.\end{prop}

\begin{proof}
It suffices to show that the maps $\varphi_{23}(Y)$, $\varphi_{13}(X)$ and $\tau$ are induced in this manner. 
We label a generic point of $\PG(3d+2,\K)$ with $(x,y,z;\xi,\upsilon,\zeta)$, where $x,y,z\in\K$ and $\xi,\upsilon,\zeta\in\A$. Then one calculates that the following $\K$-linear map $\varphi(X,Y)$ induces $\varphi_{23}(Y)$ on $\cP$ if $X=0$ and $\varphi_{13}(X)$ if $Y=0$. 
\[\varphi(X,Y): \PG(3d+2,\K)\longrightarrow \PG(3d+2,\K):(x,y,z;\xi,\upsilon,\zeta)\mapsto(x',y',z';\xi',\upsilon',\zeta'),\] with  
\[\left\{\begin{array}{rcl}
x' & = & x+\overline{\upsilon}\overline{X}+X\upsilon + X\overline{X}z,\\
y' & = & y+\xi\overline{Y} + Y\overline{\xi} + Y\overline{Y}z, \\
z' & = & z, \\ \hline
\xi' & = & \xi+Yz,\\
\upsilon' & = & \upsilon + \overline{X}z,\\
\zeta' & = & \zeta + \overline{\upsilon}\overline{Y} + X \overline{\xi}+ X\overline{Y}z.\end{array}\right.\]

Also, the triality map $\tau$ is induced in $\cP$ by the $\K$-linear map
\[\PG(3d+2,\K)\longrightarrow \PG(3d+2,\K):(x,y,z;\xi,\upsilon,\zeta)\mapsto(z,x,y;\zeta,\xi,\upsilon),\]
which can again be verified by an elementary but tedious calculation. 

Hence, noting that the above maps belong to $\mathsf{PSL}_{3d+3}(\K)$ and stabilise the point set of $\mathcal{V}_2(\K,\A)$, this concludes the proof.
\end{proof}

\begin{cor}\label{trans}
The little projective group is transitive on the set of triangles $P_0*L_1*P_2*L_0*P_1*L_2*P_0$, with $P_i\not\approx L_i$, $i=0,1,2$ and transitive on the set of pairs of points and on the set of pairs of lines which are either neighbouring or not. 
\end{cor}

\begin{proof}
 By Proposition~\ref{G2AMoufang} and the discussion preceding that proposition, $\mathsf{PSL}_3(\A)$ is transitive on the set of flags, and so $\mathsf{G}(2,\A)$ is $(P,L)$-transitive for each point $P$ and each line $L$ incident with $P$. This implies easily that $\mathsf{PSL}_3(\A)$ is transitive on the set of triangles $P_0*L_1*P_2*L_0*P_1*L_2*P_0$, with $P_i\not\approx L_i$, $i=0,1,2$. 
 
Since $\mathsf{G} (2,\A)$ is $(P,L)$-transitive for all flags $(P,L)$, $\mathsf{G}(2,\A)$ is clearly transitive on the pairs of points $(Q,R)$ which are far from each other (note that the neighbouring relation is preserved by all $(P,L)$-elations). By transitivity on points, it suffices to show that two points neighbouring $(1,0,0)$, but different from $(1,0,0)$, can be mapped to each other while fixing $(1,0,0)$. The elations $\varphi_{13}(X)$ and their conjugates under the triality map $\tau$ take care of this. The statement for the lines follows by duality. 
\end{proof}

\textbf{The geometric structure of a line}|By Corollary~\ref{trans}, each line behaves as does the line $[1,0,0]$, whose points $(X,Y,Z)$ satisfy $X=0$, so they are given by $(0,1,z)$ with $z \in \A$ and $(0,ty_1,1)$ with $y_1 \in \B$. Their images under $\rho$ are $(0,1,z\overline{z};\overline{z},0,0)$ and $(0,0,1;ty_1,0,0)$, respectively. These are exactly the points $(K_0,K_1,K_2;A_0,A_1,A_2)$ of $\rho(\mathcal{P})$ satisfying $K_1K_2=\mathsf{N}(A_0)$ and $K_0=A_1=A_2=0$. 
Recalling $A_i=(B_{i0},B_{i1})=(K_{i0},...,K_{id})$, we can write these as equations over $\K$, where $n'$ is the (anisotropic) norm form associated with $\B$.
\[K_1K_2 =n'(B_{00}) \tag{$1$}\]
\[K_0=K_{10}=\cdots = K_{1d}=K_{20}= \cdots =K_{2d}=0 \tag{$2$}\]
We conclude that the corresponding element of $\Xi$, spanned by the points of $\rho([1,0,0])$, is the $(d+1)$-dimensional subspace of $\PG(3d+2,\K)$ satisfying equation ($2$), and the points of $\rho([1,0,0])$ it contains are the ones that additionally satisfy the quadratic equation ($1$). Moreover, $\xi$ contains no other points of $\rho(\mathcal{P})$ than those of $\rho([1,0,0)]$: suppose $(x,y,z) \in \mathcal{P}$ is such that $x \neq 0$ and $x\overline{x}=z\overline{x}=x\overline{y}=0$. Then an easy calculation shows  that $x_0=y_0=z_0=0$ and hence $(x,y,z) \notin \mathcal{P}$, so $x=0$ and hence $(x,y,z)$ belongs to the line $X=0$ indeed.

If $\A=\B$, $X(\xi):=X \cap \xi$ is a quadric of Witt index 1 (i.e., whose maximal isotropic subspaces have projective dimension 0); and if $\A=\mathsf{CD}(\B,0)$, $X(\xi)$ is a cone with base a quadric of Witt index 1 in $\PG(d'+1,\K)$, where $d'=\frac{d}{2}-1$, and a $d'$-dimensional vertex which is omitted.
\par\bigskip

We now show the two specific properties (H1) and (H$2^*$) for $\mathcal{V}_2(\K,\A)$. 

\begin{prop}\label{standardex} The Veronese representation $\mathcal{V}_2(\K,\A)=(X,\Xi)$ of $\mathsf{G}(2,\A)$, where $\A=\mathsf{CD}(\B,0)$ for a quadratic associative division algebra $\B$ over a field $\K$ satisfies the following two properties:
\begin{itemize}
\item[\emph{(H1)}] Any two distinct points $x_1$ and $x_2$ of $X$ lie in at least one element of $\Xi$.
\item[\emph{(H2}$^*)$] Any two distinct elements $\xi_1$ and $\xi_2$ of $\Xi$ intersect in points which belong to the singular quadric uniquely determined by $\xi_1 \cap X$ (i.e., including its vertex).
Moreover, $\xi_1 \cap \xi_2 \cap X$ is non-empty.
\end{itemize}
\end{prop}
\begin{proof}
We first verify property (H1). This follows from the fact that $\mathsf{G} (2,\A)$ is a Hjelmslev plane. An explicit proof goes as follows. By transitivity (cf.\ Corollary~\ref{trans}), we may assume that $x_1$ is $\rho((1,0,0))$ and $x_2$ is either $\rho((0,1,0))$ (if $\rho^{-1}(x_1) \not\approx \rho^{-1}(x_2)$) or $\rho((1,t,0))$ (if $\rho^{-1}(x_1) \approx \rho^{-1}(x_2)$). Either way, $\rho([0,0,1])$ is an element of $\Xi$ containing both $x_1$ and $x_2$. This shows property (H1).

For the second property, transitivity again implies that we may assume that $\xi_1$ corresponds to $\rho([1,0,0])$ and $\xi_2$ to either $\rho([0,1,0])$ (if they come from non-neighbouring lines) or to $\rho([1,t,0])$ (if they come from neighbouring lines). In the first case, $\xi_1$ is as described in the geometric structure of a line above, and $\xi_2$ is completely analogous. It follows that the intersection of $\xi_1$ and $\xi_2$ is the unique point $(0,0,1;0,0,0)$ which is exactly $\rho((0,0,1)) \in X$, which hence belongs to $X \cap \xi_1$ indeed. In the second case, $\xi_2$ is spanned by $\rho((-ty_0,y,1))=(0,\mathsf{N}(y_0),1;y,ty_0,-t\mathsf{N}(y_0))$ for $y=y_0+ty_1\in \A$ and $\rho((-t,1,tz_1))=(0,1,0;-tz_1,0,-t)$ for $z_1 \in \B$, so $\xi_2$ is given by $K_0=B_{10}=B_{20}=B_{00}-B_{11}=B_{21}-K_1=0$. The subspace $\xi_1 \cap \xi_2$ is then given by $K_0=K_1=B_{00}=A_1=A_2=0$, so we get the points $(0,0,k_2;tb_{01},0,0)=\rho((0,tb_{01},1))$, for $k_2 \in \K$ and $b_{01} \in \B$, if $k_2 \neq 0$. If $k_2 =0$, we get precisely the vertex of the quadric determined by $\xi_i \cap X$, $i=1,2$. This shows the claim.
\end{proof}
\par\bigskip  
\section{Basic properties of (Hjelmslevean) Veronese sets}\label{reduction}
We reduce the proof to two essential cases, namely $v=-1$, i.e., vertices are empty (Case $(i)$); or there is more than one vertex and distinct vertices are pairwise disjoint (Case $(ii)$). In Section~\ref{MM}, we deal with Case $(i)$, also allowing $\K=\F_2$, and in fact covering Main Result~\ref{main2} and Main Result~\ref{main3}$(i)$. In Section~\ref{maincase} we then treat Case $(ii)$ above, covering Main Result~\ref{main3}$(ii)$.




We now start with the reduction. 

 \begin{defi}[Singular subspaces]~\ref{defpix2} \em We define a \emph{singular line} $L$ as a line of $\PG(N,\K)$ that has all its points in $X\cup Y$. Two (distinct) points $z,z'$ of $X\cup Y$ are called \emph{collinear} if they are on a singular line.  A subspace $\Pi$ of $\PG(N,\K)$ will be called \emph{singular} if it belongs to $X\cup Y$ and each pair of its points is collinear. If a singular subspace $\Pi$ intersects $Y$ in a hyperplane of $\Pi$, then $\Pi \cap X$ is called a \emph{singular affine subspace}; in particular, if $\Pi$ is a singular line containing a unique point in $Y$, then $\Pi \cap X$ is called a singular affine line.
%
\end{defi}

\begin{lemma}\label{singularlines2} Let $L$ be a line of $\PG(N,\K)$ containing two points $x_1$ and $x_2$ in $X$. Then either $L$ is a singular line having a unique point in $Y$ (hence $L\cap X$ is a singular affine line) or $L \cap (X\cup Y)=\{x_1,x_2\}$. In particular, each singular line contains at least one point of $Y$.
\end{lemma}

\begin{proof}
By (H1), there is a tube $C$ through $x_1$ and $x_2$. Suppose that $L$ contains a third point $z \in X \cup Y$. As $\Xi(C) \cap (X\cup Y) = \overline{C}$, the line $L$ belongs to a generator of $C$ and hence it is a singular line containing a unique point in the vertex of $\overline{C}$, all its other points clearly belonging to $X$.  It follows that no singular line has all its points in $X$. 
\end{proof}

\begin{lemma}\label{L-y-L'} Let $L$ and $L'$ be distinct singular lines containing unique points $y$ and $y'$ in $Y$, respectively, with $y=y'$. Then either $L$ and $L'$ belong to a unique common tube, or the plane $\<L,L'\>$ is a singular plane and $\<L,L'\> \cap X$ is a singular affine plane. 
\end{lemma}

\begin{proof} If $L \cup L'$ belongs to at least two tubes, then (H2) implies that $L \cup L'$ is contained in a generator of those tubes, and then $\<L,L'\>$ is a singular plane with a unique line in $Y$ indeed. So suppose $L \cup L'$ does not belong to any tube. 

Take any point $p$ in the plane $\pi$ spanned by $L$ and $L'$, but not on $L\cup L'$. We consider two lines $M,M'$ in $\pi$ through $p$ not incident with $y$. We consider tubes $C_M$ and $C_{M'}$ through $M$ and $M'$, respectively. If $C_M = C_{M'}$, then $C_M$ contains $L\cup L'$, contradicting our assumption. So $C_M$ and $C_{M'}$ are distinct and hence, $p \in C_M \cap C_{M'} \subseteq X \cup Y$ by (H2$^{*}$). This already shows that $\pi$ is a singular plane, i.e., $\pi \subseteq X \cup Y$.

Now $\pi$ contains a unique line in $Y$ since, by Lemma~\ref{singularlines2} and the fact that $\pi\cap X\neq\emptyset$, $\pi\cap Y$ is a geometric hyperplane of $\pi$. 
 \end{proof}

\begin{cor}\label{uniquetube2} Let $x_1$ and $x_2$ be non-collinear points of $X$. Then there is a unique $v$-space $V$ in $Y$ collinear to both of them. In particular, there is a unique tube through $x_1$ and $x_2$ (denoted $[x_1,x_2]$), and its vertex is $V$\!. 
\end{cor}
\begin{proof} By (H1), there is at least one tube $C$ through $x_1$ and $x_2$ and hence the vertex of $C$ is a $v$-space $V$ collinear to both $x_1$, $x_2$. If there would be a point $y \in Y\setminus V$ collinear with both $x_1$ and $x_2$, then Lemma~\ref{L-y-L'} implies that the lines $x_1y$ and $yx_2$ are in a tube $C'$. Since $y \notin V$, the tubes $C$ and $C'$ are distinct, but then their intersection contains two non-collinear points, contradicting~(H2)$^*$.
\end{proof}

Already at this point, the need for $|\K|>2$ arises. If $|\K|=2$, we would only be able to show that collinearity is an equivalence relation. 

\begin{lemma}\label{L-x-L'2} Two singular affine subspaces $\Pi$ and $\Pi'$ intersecting in at least one point $x\in X$ generate a singular subspace and $\<\Pi,\Pi'\> \cap X$ is a singular affine subspace.

\end{lemma}
\begin{proof}
Let $x\in X$ and suppose that $L$ and $L'$ are distinct singular lines through $x$, having unique points $y$ and $y'$ in $Y$, respectively (note that $y \neq y'$). 
Take a point $p$ on $yy'\setminus\{y,y'\}$. Since $|\mathbb{K}| >2$, there are lines  $M$ and $M'$ through $p$, each of which meets both $L\setminus\{y\}$ and $L'\setminus\{y'\}$ in distinct points of $X$. If $M$ and $M'$ are contained in the same tubic subspace, then, since $X\ni x\notin M\cup M'$, the plane spanned by the lines $M,M'$ is singular, with a unique line in $Y$. Hence we may assume that $M$ and $M'$ are contained in distinct tubic subspaces, and so $p\in X\cup Y$ by (H2$^*$). 
As $p$ was arbitrary on  $yy'\setminus\{y,y'\}$, it follows that the line $yy'$ is contained in $X \cup Y$. So by  Lemma~\ref{singularlines2}  and $|\K|>2$, we may assume $p \in Y$. Applying Lemma~\ref{L-y-L'} on the lines $M$ and $M'$, we again obtain that $\<L,L'\>$ is a singular plane with a unique line (namely $yy'$) in $Y$.

Now let $\Pi$ and $\Pi'$ be general singular affine subspaces with $x\in\Pi\cap\Pi'$. Repeated use of the previous argument for all affine singular lines in $\<\Pi,\Pi'\>$ sharing $x$ shows the lemma.
\end{proof}
\par\bigskip

\begin{cor}\label{Ysub}
The set $Y$ is a subspace.
\end{cor}

\begin{proof}
Let $y_1,y_2\in Y$, $y_1\neq y_2$. Let $y_i$ be contained in a tube $C_i$, $=1,2$.  If $C_1$ equals $C_2$, then the line $y_1y_2$ joining $y_1$ and $y_2$ is contained in its vertex, and in particular in $Y$. If $C_1 \neq C_2$, then their intersection contains a point $x\in X$ by (H2$^*$). Applying Lemma~\ref{L-x-L'2} on the lines $L=xy_1$ and $L'=xy_2$, we obtain that $y_1y_2 \subseteq Y$. 
\end{proof}


\begin{defi}[Maximal singular subspaces]\label{defpix2} \em Let $x$ be any point in $X$. By the previous lemma, we can define $\Pi_x$ as the unique maximal singular affine subspace containing $x$, and we denote its projective completion (i.e., $\<\Pi_x\>$) by $\overline{\Pi}_x$.  Finally, we define $\Pi_x^Y=\overline{\Pi}_x\setminus\Pi_x=\overline{\Pi}_x\cap Y$.\end{defi}

We have the following corollary.

\begin{cor}\label{pix2}   
For $x,x' \in X$, $\Pi_x \cap \Pi_{x'}$ is non-empty if and only if $\Pi_x = \Pi_{x'}$ if and only if  $x$ and $x'$ are collinear. If $x$ and $x'$ are not collinear, then $\overline{\Pi}_x \cap \overline{\Pi}_{x'}$ is the vertex of $[x,x']$.
\end{cor}

\begin{proof}
Suppose $\Pi_x \cap \Pi_{x'}$ contains a point of $X$. It follows from Lemma~\ref{L-x-L'2} that $\Pi_x$ and $\Pi_{x'}$ generate a singular subspace $\Pi$. As both were the maximal ones containing $x$ and $x'$, $\Pi_x=\Pi=\Pi_{x'}$. Clearly, $\Pi_x=\Pi_{x'}$ implies $x' \in \Pi_x$, so $x$ and $x'$ are collinear. For collinear points $x$ and $x'$, we have $x,x' \subseteq \Pi_x \cap \Pi_{x'}$, hence the ``if and only if''-statements follow.

If $x$ and $x'$ are non-collinear points, then Corollary~\ref{uniquetube2} implies that the vertex of $[x,x']$ coincides with $\overline{\Pi}_x \cap \overline{\Pi}_{x'}$.
\end{proof}
\par\bigskip



\begin{lemma}\label{vertex} 
Let $C_1$ and $C_2$ be tubes with respective vertices $V_1$ and $V_2$. Set $V^*=V_1\cap V_2$. 
Then the  vertex $V$ of each tube $C$ contains $V^*$. Hence each point of $X$ is collinear with $V^*$. Moreover, the intersection of any pair of distinct vertices is precisely $V^*$. 

\end{lemma}
\begin{proof}
By (H2$^*$), the tubes $C_1$ and $C_2$ share a point $x\in X$, so $\overline{C_1}\cap \overline{C_2} = \<x,V^*\>$. By the same axiom, $C \cap C_i$ contains a point $z_i \in X$, $i=1,2$. 

We claim that $z_1 \perp z_2$ if and only if $z_1\in\<x,V_1\>$ and $z_2\in\<x,V_2\>$. Indeed, suppose $z_1$ does not belong to $\<x,V_1\>$ (which is equivalent to  $z_1$ not being collinear to $x$) and suppose $z_1 \perp z_2$. The first fact means $C_1=[x,z_1]$, so by Corollary~\ref{pix2} we have $\Pi_x^Y \cap \Pi_{z_1}^Y=V_1$. By the same corollary, the second fact means $V_2 \subseteq \Pi_{z_1}^Y$ and hence $V_2 \subseteq \Pi_x^Y \cap \Pi_{z_1}^Y= V_1$, a contradiction.
The other implication is clear since $\<x,V_1,V_2\>$ belongs to the singular subspace $\Pi_x$.  This shows the claim. 


Now, if $z_1$ is not collinear with $z_2$, then Corollary~\ref{pix2} readily implies $V^* \subseteq V$. 
So suppose $z_1,z_2 \in x^\perp$ (and hence $z_1$ and $z_2$ are equal or collinear).
For $i=1,2$, let $z'_i$ be a point of $C_i$ not collinear to $x$ (and hence neither to $z_i$). Then by the above, the vertex $V'$ of the tube $C'$ through $z_1'$ and $z_2'$ contains 
$V^*$. For $i=1,2$, we now consider $C_i$ and $C'$ instead of $C_1$ and $C_2$. Then again, since $z_i$ and $z'_i$ are not collinear, the previous cases reveal that $V \cap V_i$ contains $V^*$.

As $C$ was arbitrary, we conclude that each tube's vertex $V$ contains $V^*$. 
It immediately follows that each point $x$ is collinear with $V^*$. Assume that there would be two tubes $C'_1$ and $C'_2$ whose respective distinct vertices $V'_1$ and $V'_2$ would intersect in more than $V^*$. Repeating the above argument, we would obtain that $C_1$ and $C_2$ both contain $V'_1 \cap V'_2$, a contradiction.
\end{proof}

\par\bigskip

For an arbitrary subspace $F$ of $\PG(N,\K)$ complementary to $V^*$, we  now consider the map $\rho: X \rightarrow F: x \mapsto \<x,V^*\> \cap F$.  The pair $(\rho(X),\rho(\Xi))$ is well defined then and consists of $(d,v')$-tubes with base $Q^0_d$, where $v'=\cod_V(V^*)$, for any vertex $V$. 

\begin{prop}\label{vertexred}
Let $C_1$ and $C_2$ be tubes with respective vertices $V$ and $V'$ that intersect in a subspace $V^*$ of $V$.
Then $(\rho(X),\rho(\Xi))=(X\cap F,\{\xi\cap F \mid \xi\in\Xi\})$ is a Hjelmslevean Veronese set with $(d,v')$-tubes for $v'=\cod_V(V^*)$. If $v' \geq 0$ then two vertices either coincide or are disjoint and both cases occur. \end{prop}
\begin{proof}
By Lemma~\ref{vertex}, each point of $x \in X$ is collinear with $V^*$  and all elements of $\Xi$ contain $V^*$.  Hence $(\rho(X),\rho(\Xi))=(X\cap F,\{\xi\cap F \mid \xi\in\Xi\})$. Clearly, $\rho(\xi)\cap \rho(X)$ is a $(d,v')$-tube for $v' = \cod_V(V^*)$ for each $\xi\in\Xi$. We show that (H1) and (H2$^*$) are satisfied.

$\quad \bullet$  \textit{Axiom \emph{(H1)}.} Let $x$ and $x'$ be points of $\rho(X)$. Then Axiom (H1) in $(X,\Xi)$ implies that there is a tube $C$ containing $x$ and $x'$. By Lemma~\ref{vertex}, the vertex of $C$ contains $V^*$ and hence $\rho(C)$ is a tube through $x$ and $x'$. 

$ \quad \bullet$ \textit{Axiom \emph{(H2$^*$)}.} Let $\xi$ and $\xi'$ be distinct members of $\rho(\Xi)$. Then $\<\xi,V^*\> \cap \<\xi',V^*\>$ belongs to $X\cup Y$ and contains at least one point $x \in X$ by Axiom (H2$^*$) in $(X,\Xi)$. It is clear that $\xi \cap \xi'$ belongs to $X \cup Y$ and that it contains $\rho(x) \in \rho(X)$.

The rest from the statement follows immediately from Lemma~\ref{vertex}.
\end{proof}

\par\bigskip

\textbf{Vertex-reduction}| Consider the following property.
\begin{itemize}
\item[($\mathsf{V}$)] Two vertices either coincide or have empty intersection, and both cases occur.
\end{itemize}
If ($\mathsf{V}$) is not valid, then either all tubes have the same vertex $V$ or there are tubes $C$ and $C'$ such that their respective vertices $V$ and $V'$ are neither disjoint nor equal. In both cases, we have shown above that the projection from $V$ or from $V \cap V'$, respectively, yields a point-quadric variety with $(d,v')$-tubes with base $Q^0_d$ with $v'=-1$ or with $0 \leq v' \leq v$, respectively, which satisfies (H1), (H2$^*$) and, if $v'\geq 0$,  also ($\mathsf{V}$) applies. 
We deal with those two possibilities separately. 

\section{The case $v=-1$: Ordinary Veronese sets}\label{MM}
In this section we treat the case of Main Result~\ref{main3}$(i)$ by showing Main Result~\ref{main2}, and then use Proposition~\ref{vertexred}.  We will use  a characterisation of the ordinary Veronese representations $\mathcal{V}_2(\K,\B)$ of a projective plane $\PG(2,\B)$ over a quadratic alternative division algebra $\B$ over $\K$ by means of Mazzocca-Melone axioms by O. Krauss, J. Schillewaert and the second author \cite{Kra-Sch-Mal:15}.  

\subsection{The general set-up}

So let $X$ be a spanning point set of $\PG(N,\K)$, $N>d+1$ (possibly infinite), and let $\Xi$ be a collection of $(d+1)$-dimensional projective subspaces of $\PG(N,\K)$  (called the \emph{elliptic spaces}) such that, for any $\xi \in \Xi$, the intersection $\xi \cap X$ is a $Q^0_d$-quadric $X(\xi)$ whose points span $\xi$. The tangent spaces $T_x(X(\xi))$ are also denoted by $T_x(\xi)$ as before and the subspace spanned by all such tangent spaces through $x$ is again called the \emph{tangent space $T_x$ of $x$}, i.e., $T_x=\<T_x(\xi) \mid x\in \xi \in \Xi\>$. For such a quadric $Q$, we denote by $\Xi(Q)$ the unique member of $\Xi$ containing $Q$. We assume that $(X',\Xi)$ satisfies the following two properties:

\begin{enumerate}
\item[{(MM1)}]  each pair of distinct points $x'_1,x'_2 \in X'$ is contained in some element of $\Xi'$, and
\item[{(MM2$^*$)}] the intersection of each pair of distinct elements of $\Xi'$ is precisely a point of $X'$. 
\end{enumerate}

Seeing the different nature of the results depending on $|\K|$, it should not surprise that we divide the proof into two cases.

\subsection{The case $|\K|>2$} \label{k=2}

Our aim is to show that $(X,\Xi)$ satisfies the following properties. 
\begin{itemize}
\item[(MM1)] Any two distinct points $x_1$ and $x_2$ of $X$ lie in a element of $\Xi$;
\item[(MM2)] for any two distinct members $\xi_1$ and $\xi_2$ of $\Xi$, the intersection $\xi_1 \cap \xi_2$ belongs to $X$;
\item[(MM3)] for any $x\in X$ and any three distinct members $\xi_1$, $\xi_2$ and $\xi_3$ of $\Xi$ with $x\in \xi_1\cap \xi_2 \cap \xi_3$ we have $T_x(\xi_3) \subseteq \< T_x(\xi_1), T_x(\xi_2)\>$.
 \end{itemize}
 
 Indeed, if $(X,\Xi)$ satisfies (MM1) up to (MM3), then $(X,\Xi)$ is projectively equivalent to $\mathcal{V}_2(\K,\B)$, with $\B$ a quadratic alternative division algebra $\B$ over $\K$  with $\dim_\K(\B)=d$, as follows from a characterisation of the ordinary Veronese representations $\mathcal{V}_2(\K,\B)$ of a projective plane $\PG(2,\B)$ over such an algebra $\B$, a by means of Mazzocca-Melone axioms by O.\ Krauss, J.\ Schillewaert and H.\ Van Maldeghem in \cite{Kra-Sch-Mal:15}.  Conversely, then it is easily verified that, for each such algebra $\B$, the corresponding Veronese representation $\mathcal{V}_2(\K,\B)$ satisfies our axiom (MM2$^*$). This then shows Main Result~\ref{main2} in case $|\K|>2$ (if $|\K|=2$, (MM3) not necessarily holds). 
 
\par\bigskip
Fix an elliptic space $\xi \in \Xi$ and put $Q=X(\xi)$; let $F$ be a subspace of $\PG(N,\K)$ complementary to $
\xi$. We denote by $\rho: \PG(N,\K) \rightarrow F$ the projection operator that projects from $\xi$ onto $F$.

\begin{lemma}\label{injective}
The projection $\rho$ is injective on $X \setminus Q$. \end{lemma}
\begin{proof}
Take two (distinct) points $p,q \in X\setminus Q$ and suppose for a contradiction that $\rho(p)=\rho(q)$. Then $\xi=\<Q\>$ is a hyperplane of $\<Q,p,q\>$, implying that the line $pq$ intersects $\xi$ in a point $z$. On the other hand, by (MM1$^*$), $p$ and $q$ are contained in a quadric $Q'$ which intersects $\xi$, by (MM2$^*$), in a point of $Q$. Since $z \in \xi \cap \Xi(Q')$, it follows from (MM2$^*$) that $z \in X$. Clearly, the points $p,q,z$ are all distinct, and hence  the line $pq$ in $\Xi(Q')$ contains three points of $X$, a contradiction. The assertion follows.
\end{proof}

\begin{lemma}\label{image}
The image $\rho(X\setminus Q)$ is an affine subspace $A$ whose projective completion equals~$F$. 
\end{lemma}
\begin{proof}
We first prove that the image is an affine space. Let $Q'$ be any quadric distinct from $Q$. By (MM2$^*$),  $Q \cap Q'=\xi \cap \Xi(Q')$ contains exactly one point, say $z$. So $\rho(Q')$ is given by projecting $Q'$ from $z$ and as such it is an affine $d$-space, whose projective $(d-1)$-space at infinity corresponds to $\rho(T_z(Q'))$. 

Consequently, (MM1$^*$) implies that each two points $\rho(p), \rho(q)$, with $p,q   \in X\setminus Q$, are contained in an affine line $L_{pq}$ of $\rho(X([p,q]))$. Let $y$ be the point on $\rho(p)\rho(q)$ not contained in $\rho(X([p,q]))$, i.e., $L_{pq} \cup \{y\} = \rho(p)\rho(q)$. If $y = \rho(r)$ for some $r \in X\setminus Q$, then $\rho(p)\rho(q)=\rho(p)\rho(r)$, and as $|\K| > 2$, this yields at least two points in $L_{pq} \cap L_{pr}$, and by injectivity, their inverses belong to  $X[p,q] \cap X[p,r]$, which only contains one point by (MM1$^*$), a contradiction. This shows $y \notin \rho(X\setminus Q)$. The set $Y:=\{ \rho(p)\rho(q) \setminus L_{pq} \mid p,q \in X\setminus Q, p \neq q\}$ thus belongs to $F \setminus \rho(X\setminus Q)$. 

Now take three points $p,q,r$ in $\rho(X\setminus Q)$ which are not on a line. We claim that $\<p,q,r\> \cap \rho(X\setminus Q)$ is an affine plane (whose projective completion equals $\<p,q,r\>$). Let $y_q$ be the unique point of $Y$ on $pq$ and $y_r$ the unique point of $Y$ on $pr$. By the previous paragraph we already know that a line containing two points of $\rho(X\setminus Q)$ has all but one points in $\rho(X\setminus Q)$, the remaining point being contained in $Y$.  In particular,  the line $y_qy_r$ contains at most one point in $\rho(X\setminus Q)$. Suppose $y_qy_r$ contains a unique point $x \in \rho(X\setminus Q)$. Then $px$ contains a unique point $y \in Y$ through which there is a line $L$ intersecting $pq\setminus\{p,y_q\}$ and $pr\setminus\{p,y_r\}$, since $|\K|>2$. Clearly, $L \neq px$, so $L$ intersects $y_qy_r$ in a point distinct from $x$ and hence $L$ contains at least two points of $\rho(X\setminus Q)$ and two points of $Y$, a contradiction.  Hence all points of $y_qy_r$ belong to $Y$. Now each point $v \in \<p,q,r\>\setminus y_qy_r$ is on a line containing at least two points of $(pq \cup qr \cup rp) \cap X$, implying $v\in\rho(X\setminus Q)$. The claim is proved.

It follows that $\rho(X\setminus Q)$ is an affine subspace of $F$ and, as $X$ spans $\PG(N,\K)$, $\rho(X\setminus Q) \cup Y=F$.
\end{proof}

We keep referring to the projective space at infinity of $\rho(X\setminus Q)$ in $F$ as $Y$. 
\begin{lemma}\label{partition}
Let $x$ be a point of $X\setminus Q$.  For distinct points $p,q$ of $Q$, $\rho(T_p([p,x])) \cap \rho(T_q([q,x]))$ is empty and $\bigcup_{p \in Q} \rho(T_p([p,x]))=Y$.
\end{lemma}
\begin{proof}
Put $Q_p=X([p,x])$ and $F_p:=\rho([p,x])$, i.e., $F_p=\rho(Q_p) \cup \rho(T_p([p,x]))$; likewise $Q_q=X([q,x])$ and $F_q :=\rho([q,x])$. 
As $\rho$ is injective by Lemma~\ref{injective}, (MM2$^*$) implies that $\rho(Q_p) \cap \rho(Q_q)$ is exactly $\rho(x)$. 
Moreover, this also implies that $F_p \cap F_q =\rho(x)$, as otherwise $F_p \cap F_q$ contains a line through $\rho(x)$ and then $\rho(Q_p) \cap\rho(Q_q)$ would be an affine line through $\rho(x)$, a contradiction. It follows that $\rho(T_p(Q_p)) \cap \rho(T_q(Q_q))$ is empty. 

Now take $y \in Y$ arbitrary. Let $r$ be a point of $F \setminus Y$ on the line $\rho(x)y$ and put $r' = \rho^{-1}(r)$ (which is well defined by Lemmas~\ref{injective} and~\ref{image}). Then $[x,r'] \cap Q$ is a point $r''$ and we obtain that $y \in \rho(T_{r''}[r'',x])$. Note that any point $r'$ on $\rho(x)y$ would yield the same point $r''$ by the previous paragraph.
\end{proof}

\begin{lemma}\label{equalT}
Let $p$ be a point in $Q$. Then $\rho(T_p(Q_1))=\rho(T_p(Q_2))$ for all quadrics $Q_1$ and $Q_2$ distinct from $Q$ and with $p \in Q_1 \cap Q_2$. 

\end{lemma}
\begin{proof}
Suppose for a contradiction that $\rho(T_p(Q_1)) \neq \rho(T_p(Q_2))$ for two quadrics $Q_1$ and $Q_2$ distinct from $Q$ with $p \in Q_1 \cap Q_2$. Then there is a point $y \in \rho(T_p(Q_1)) \setminus \rho(T_p(Q_2))$. Let $x_2$ be a point in $Q_2 \setminus \{p\}$. By Lemma~\ref{partition}, $y$ belongs to $\rho(T_{p'}([p',x_2]))$ for some $p' \in Q$ with $p' \neq p$. By (MM2$^*$), $Q_3=X([p',x_2])$ intersects $Q_1$ in a point $x_1$, and $x_1 \neq p$ since $x_2 \notin Q$. Then $Q_1$ and $Q_3$ are two different quadrics through $x_1$, and $y \in \rho(T_p(Q_1)) \cap \rho(T_{p'}([x_2,p']))$, whereas this intersection should be empty  according to Lemma~\ref{partition}. This contradiction shows the lemma. 
\end{proof}

\begin{lemma}\label{mm3}
Let $p$ be any point in $Q$. For any member $\xi' \in \Xi\setminus\{\xi\}$ with $p \in \xi \cap \xi'$,  $T_p = \<T_p(\xi),T_p(\xi')\>$. \end{lemma}

\begin{proof}
By Lemma~\ref{equalT}, $\rho(T_p(Q'))=\rho(T_p(Q''))$ for all quadrics $Q',Q''$ distinct from $Q$. Now fix any quadric $Q' \neq Q$ through $p$. By definition of $T_p$ we have $\rho(T_p)=\rho(T_p(Q'))$. We obtain $\<T_p(Q),T_p(Q')\> \subseteq T_p \subseteq \<Q, T_p(Q'),\>$. Since $\<T_p(Q),T_p(Q')\>$ is a hyperplane of $\<Q, T_p(Q')\>$, we have that either $T_p = \<T_p(Q),T_p(Q')\>$, in which case the lemma is proven, or $T_p = \<Q,T_p(Q')\>$. So suppose we are in the latter case, in which $Q \subseteq T_p$. Then no quadric $Q'' \neq Q$ through $p$ can be contained in $T_p$, for otherwise $\Xi(Q) \cap \Xi(Q'')$ contains at least a line, a contradiction. Switching the roles of $Q$ and $Q'$, we  obtain that $\<T_p(Q),T_p(Q')\> \subseteq T_p \subseteq \<Q', T_p(Q)\>$, and the latter situation cannot occur since $Q\not\subseteq \<Q',T_p(Q)\>$. The lemma is proven.
\end{proof}

Since $Q$ and $p \in Q$ were arbitrary, it follows from Lemma~\ref{mm3} that $(X,\Xi)$ satisfies Axiom (MM3), finishing the proof of Theorem~\ref{main2} in case $|\K| >2$. 

\subsection{The case $|\mathbb{K}|=2$}

When there are only 3 points on a line, the above techniques fail and for a very good reason: We get more examples. As the field is finite, $Q^0_d$-quadrics only exist when $d=1,2$. We deal with those cases separately. 

Since we are working here in projective spaces of order 2, we can add points together: The sum of two points is the third point on the line determined by those two points. This additive structure, with additional neutral element $\emptyset$, where $a+a=\emptyset$, for each point $a$, is an elementary abelian 2-group (the additive group of the underlying vector space).  
 
\subsubsection{$d=1$}

Axioms (MM1) and (MM2$^*$) imply that $(X,\Xi)$, if existing, is as a point-line geometry isomorphic to a projective plane of order $2$ (i.e., $\PG(2,2)$) and hence contains seven points in total. 

\begin{prop}\label{K=2,d=1}
For any pair $(X,\Xi)$ satisfying \emph{(MM1)} and \emph{(MM2$^*$)}, with $X$ a spanning point set of $\PG(N,2)$, $N>2$, and $\Xi$ a set of planes, we have $N \in \{5,6\}$. If $N=5$ there are, up to projectivity, two possibilities---among which $\mathcal{V}_2(\F_2,\F_2)$; if $N=6$ then $X$ is any basis of $\PG(6,\K)$. 
\end{prop}

\begin{proof}
Since there are only seven points, we readily obtain $N \leq 6$. Now by (MM2$^*$), we see that $N\geq 4$ and moreover this axiom implies that each plane of $\PG(N,2)$ contains at most three points of $X$ and each $3$-space of $\PG(N,2)$ at most four points of $X$ (indeed, any set of five points of a projective plane of order 2 forms exactly the set of points on two lines and hence spans a $4$-space of $\PG(N,2)$). 

First suppose $N=4$. We choose five points of $X$, which, by the above, form a basis of $\PG(4,2)$. Now the two remaining points of $X$ are not contained in any 3-space spanned by four points of the basis. But there is only one such point in $\PG(4,2)$. 
This contradiction rules out $N=4$. 

Next, suppose $N=5$. Since no line contains three points of $X$, no plane contains four points of $X$ and no $3$-space contains five points of $X$, there are only two options. Firstly, it could be that no $4$-space contains six points of $X$, in which case we obtain that the seven points of $X$ form a frame. Then $(X,\Xi)$ is projectively equivalent to $\mathcal{V}_2(\F_2,\F_2)$. Secondly, if there is a $4$-space $S$ containing six points of $X$ (seven is impossible by the previous paragraph), then these six points form a frame of $S$ and the seventh point of $X$ is a point outside $S$ forming a basis with any $5$ points of $S \cap X$. One easily checks that such a set satisfies the axioms (MM1) and (MM$^*$), no matter how we choose the elliptic spaces.   

Finally, suppose $N=6$. Then $X$ generates $\PG(N,\K)$ and hence is any basis of it. Also in this case, any choice of the elliptic spaces will do. 
\end{proof}

\subsubsection{$d=2$}
As each ovoid in $\PG(3,2)$ contains five points, it follows as before that the pair $(X,\Xi)$, as a point-line geometry, is a projective plane of order 4, hence containing $21$ points and as such isomorphic to $\PG(2,4)$. Clearly, each set of four points on an ovoid $O$ in $\PG(3,2)$ determines a basis of $\PG(3,2)$. Note that there is a unique frame of $\PG(3,2)$ containing this basis, which then coincides with $O$. More precisely, if we let $e_0$, $e_1$, $e_2$, $e_3$ be any four of its points, then the fifth point is $e_0 + e_1+e_2+e_3$. This will be the key observation to show the following proposition.

\begin{prop}\label{K=2,d=2} For any pair $(X,\Xi)$, where $X$ is a spanning point set of $\PG(N,2)$ with $N >3$ and $\Xi$ a family of $3$-spaces, satisfying \emph{(MM1)} and \emph{(MM2$^*$)}, we have $8 \leq N \leq 10$. If $N=10$ then $(X,\Xi)$ is projectively unique (and denoted by $\mathcal{M}^{10}(\F_2)$); if $N=9$ or $N=8$, then $(X,\Xi)$ results from projecting  $\mathcal{M}^{10}(\K)$ from a suitable point or line, respectively, and there is a unique such line that  gives  $\mathcal{V}_2(\F_2,\F_4)$. In both cases, $(X,\Xi)$ is projectively unique.
\end{prop}

We prove this proposition in a small series of lemmas. In the first lemma (Lemma~\ref{repr}) we consider all representations of $\PG(2,4)$  as point-block geometries in $\PG(N,2)$, such that blocks of $\PG(2,4)$ correspond to ovoids  in $3$-dimensional subspaces  of $\PG(N,2)$. Noting that an ovoid in $\PG(3,2)$ is a \emph{frame} (in general this is a set of $n+2$ points of an $n$-dimensional projective space such that each $n+1$ among them generate the space), the lemma is in fact about {pseudo embeddings} of $\PG(2,4)$. \emph{Pseudo embedding} of point-line geometries have been introduced and studied by De Bruyn \cite{Bru:12,Bru:13}. In Proposition~4.1 of~\cite{Bru:13}, he obtained that the universal pseudo-embeddings of $\PG(2,4)$ lives in $\PG(10,2)$ and an explicit (coordinate) construction has been given by him in Theorem 1.1 of~\cite{Bru:12}. Nevertheless we include our construction, which is in terms of a basis of $\PG(10,2)$  because we will rely on it in the lemmas thereafter to prove results in our more specific setting (in which (MM2) also holds). 

\begin{lemma}\label{repr} Let $(X,\Phi)$ be a pair with $X$ a spanning point set of $\PG(N,2)$, $N >3$ and $\Phi$ a family of ovoids in $3$-spaces, such that, with the natural incidence, $(X,\Phi)$ is a projective plane of order $4$. Then\begin{compactenum}[$(i)$] \item if $N=10$, then $(X,\Phi)$ is projectively unique and denoted by $\mathcal{M}^{10}(\F_2)$; and \item each such structure is the projection of $\mathcal{M}^{10}(\F_2)$.  \end{compactenum} In particular $N\leq 10$. Moreover, the stabiliser of $\mathcal{M}^{10}(\F_2)$ in $\PSL(11,2)$ is group isomorphic to $\mathsf{P\Gamma L}(3,4)$.  
%
\end{lemma}
\begin{proof}
For convenience, we shall call a member of $\Phi$ a \emph{block}. So a block is a line of the projective plane $(X,\Phi)\cong\PG(2,4)$, and at the same an ovoid in some $3$-space of $\PG(N,2)$. The unique block through two distinct points $a,b$ will be denoted by $[a,b]$, since the notation $ab$ will mean something else (namely, $a+b$). 

Let $\circ$ and $*$ be any two (distinct) elements of $X$. Take arbitrarily three blocks $\xi_1^*$, $\xi_2^*$ and $\xi_3^*$ through $*$ and not through $\circ$, and three arbitrary blocks $\xi^\circ_1$, $\xi^\circ_2$ and $\xi^\circ_3$ through $\circ$ but not through $*$, in such a way that the points $\xi_i^*\cap\xi_i^\circ$, $i=1,2,3$, are on a block (this can be achieved by possibly just interchanging $\xi_2^\circ$ and $\xi_3^\circ$). Then we claim that the nine intersection points of these blocks, together with $\circ$ and $*$, fully determine the pair $(X,\Phi)$, as a substructure of $\PG(N,2)$, and $X$ is contained in the span of these eleven points. In particular, $N \leq 10$. 

Let us label the nine intersection points of $\xi^*_i$ and $\xi^\circ_j$, $i,j\in\{1,2,3\}$, by the digits $1$ up to $9$ according to the picture below. Set $I=\{1,2,\ldots,9\}$.

\begin{figure}[ht]
\begin{center}\footnotesize
\begin{tikzpicture}[scale=0.5]
\foreach \i in {-10,-5,0,5,10}{
	 \node [inner sep=0.8pt,outer sep=0.8pt]  at (\i,0) {$\bullet$};
	 \node [inner sep=0.8pt,outer sep=0.8pt]  at (5+\i/2,5-\i/2) {$\bullet$};	
	 \node [inner sep=0.8pt,outer sep=0.8pt]  at (-5-\i/2,5-\i/2) {$\bullet$};	
        \draw[gray,very thin] (\i,0) -- (0,10);
        \draw[gray, very thin]  (5+\i/2,5-\i/2) -- (-10,0);

}

	 \node   at (-5,0) {$\bullet$};
	  \node [left]  at (-5,0) {$\circ123$};
	  
	  \node   at (0,0) {$\bullet$};
	  \node [left]  at (0,0) {$\circ456$};
	  
	  \node   at (5,0) {$\bullet$};
	  \node [left]  at (5,0) {$\circ789$};
	  
	  \node   at (2.5,7.5) {$\bullet$};
	  \node [above right]  at (2.5,7.5) {$147*$};

	  \node   at (5,5) {$\bullet$};
	  \node [above right]  at (5,5) {$258*$};
	  
	  \node   at (7.5,2.5) {$\bullet$};
	  \node [above right]  at (7.5,2.5) {$369*$};

	 	 \node [above left]  at (-2.5,7.5) {$\overline{348}$};	
		 
		  \node [above left]  at (-5,5) {$\overline{159}$};	
		  
		 \node [above left]  at (-7.5,2.5) {$\overline{267}$};
		 
		 \node [above right]  at (10,0) {$\Sigma$};	

	 \node [blue]  at (-10,0) {$\bullet$};
	  \node [left, blue]  at (-10,0) {$*$};
	 
	 \node [blue]  at (0,10) {$\bullet$};
	 \node [above, blue] at (0,10) {$\circ$};

	 \node [blue]  at (-2.85,4.3) {$\bullet$};
	  \node [ left,blue]  at (-2.85,4.3)  {$1$};
	  
	   \node [blue]  at (-4,2) {$\bullet$};
	   \node [ left, blue]  at (-4,2) {$2$};
	   
	   \node [blue] at (-4.6,0.75) {$\bullet$};
	   \node [left, blue]  at (-4.6,0.75) {$3$};
	   
	   \node [blue]  at (0,6) {$\bullet$};
	   \node [above, blue]  at (0,6) {$4$};
		    
	   \node [blue]  at (0,3.35) {$\bullet$};
	   \node [above, blue]  at (0,3.35) {$5$};
	   
	  \node [blue]  at (0,1.45) {$\bullet$};
	   \node [above, blue]  at (0.1,1.45) {$6$};
	   
	   \node [blue]  at (1.55,6.9) {$\bullet$};
	   \node [above, blue]  at (1.55,6.9) {$7$};
	   
	   \node [blue]  at (2.85,4.3) {$\bullet$};
	   \node [above, blue]  at (2.85,4.3) {$8$};
	   
	   \node [blue]  at (4,2) {$\bullet$};
 	   \node [above, blue]  at (4,2) {$9$};
	   
	   \end{tikzpicture}
\end{center}
\caption{The projective plane $(X,\Xi)$ \label{fig1}}
\end{figure}
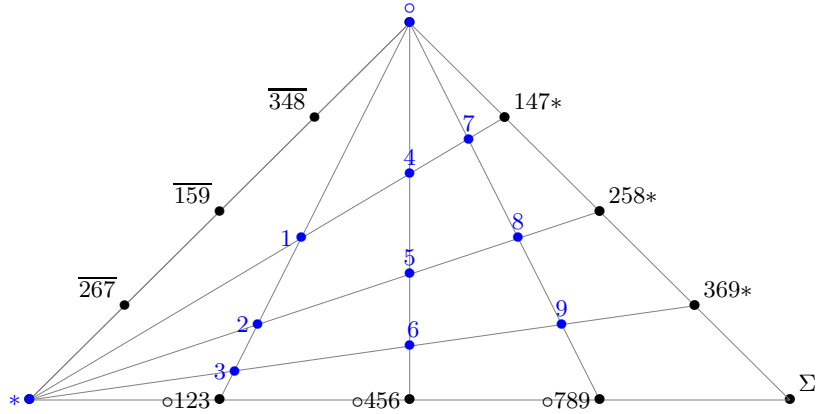

Since each of those six blocks now contains exactly four points of the set $I \cup \{*,\circ\}$, its fifth point is uniquely determined by their sum (and we denote $\circ+1+2+3$ as $\circ123$ and $1+4+7+*$ as $147*$---be aware that we use the elements of $I$ as mere symbols; in general we shall denote the sum of elements of $I\cup\{*\circ\}$ by juxtaposition; we overline a string of elements of $I\cup\{*,\circ\}$ if we mean the sum of the complement of the elements in the string). We obtain six additional points: $147*$, $258*$ and $369*$ (called \emph{$*$-triples}) on the blocks $\xi_1^*$, $\xi^*_2$ and $\xi^*_3$, respectively; and $\circ123$, $\circ456$ and $\circ789$ (called \emph{$\circ$-triples}) on $\xi^\circ_1$, $\xi^\circ_2$ and $\xi^\circ_3$. The $*$-triples are on a block $\xi^\circ_4$ through $\circ$ and the $\circ$-triples are on a block $\xi^*_4$ through $*$. From each of the blocks $\xi^\circ_4$ and $\xi^*_4$, four points are determined, and hence the remaining point is determined as well. For both blocks, this remaining point is $\circ123456789*=:\Sigma$. 

To define the three remaining points of $X$ (those in $[\circ,*]\setminus\{\circ,*\}$), we consider the blocks $[\Sigma,7]$, $[\Sigma,8]$ and $[\Sigma,9]$. By our assumption that $1,5$ and $9$ are on a block, the block $[\Sigma,7]$ contains the points $2,6,7$, the block $[\Sigma, 8]$ contains the points $3,4,8$ and, lastly, the block $[\Sigma,9]$ contains the points $1,5,9$. The fifth points on these blocks are $\overline{267}$, $\overline{348}$ and $\overline{159}$, respectively (these three are called the $\Sigma$-triples; as introduced above, $\overline{abc}$ denotes the sum of the complement of $\{a,b,c\}$ in the set $I \cup \{\circ, *\}$). The points $\overline{267}$, $\overline{159}$ and $\overline{348}$ lie on a block together with $\circ$ and $*$ and they do sum up to zero indeed. 

We need the nine remaining blocks of the projective plane $(X,\Psi)$ to conclude that this is well defined. This could be done by using coordinates, though we prefer to give the remaining blocks by reasoning as follows. 

For each point in $I$, we need two more blocks through it. Taking $1 \in I$ as an example, we note that the blocks through $1$ and $*$, $\circ$ and $\Sigma$, respectively, are given as follows: $[1,*]=\{1,4,7,*,147*\}$, $[1,\circ]=\{\circ,1,2,3,\circ123\}$, $[1,\Sigma]=\{1,5,9,\Sigma, \overline{159}\}$, so for each $\mathsf{x} \in \{*,\circ,\Sigma\}$, the $\mathsf{x}$-triple containing $1$ reveals which points are on the block $[1,\mathsf{x}]$. Since $6,8$ do not occur in any such triple,  the remaining blocks are $[1,6]$ and $[1,8]$, and they need to be distinct (there is no $*$-triple neither containing $1$ nor $6$ nor $8$). Hence the block $[1,6]$ has to contain $\mathsf{x}$-triples not containing $1$ and $6$, but there are exactly three such.  Consequently, there is only one possibility: $[1,6]=\{1,6,258*,\circ789,\overline{348}\}$. Likewise $[1,8]=\{1,8,369*,\circ456,\overline{267}\}$. These indeed have sum zero. In general, let $\{a,b\} \subseteq I$ be any pair that is, just like $\{1,6\}$ and $\{1,8\}$, not contained in any triple. Then for each $\mathsf{x} \in \{*,\circ,\Sigma\}$, there is a unique triple not containing $a$, nor $b$, which we denote by $T^\mathsf{x}(ab)$. For those pairs $\{a,b\}$ (for the record, these are all pairs occurring in $\{1,6,8\}$, in $\{2,4,9\}$ and in $\{3,5,7\}$) we define
\[ [a,b]:= \{a,b,T^*(ab),T^\circ(ab),T^\Sigma(ab)\}.\] A straightforward verification shows that each such block sums up to zero.

We now have 21 blocks, 5 through each point and one through each pair of distinct points, confirming that the above defined set of points and blocks indeed is the projective plane of order 4. Hence the set $I \cup \{*,\circ\}$ defines $(X,\Xi)$ entirely. In particular, $N \leq 10$.

Now let $N=10$. Then we can take for $I\cup\{*,\circ\}$ any basis of $\PG(10,2)$ and we obtain a unique example $\mathcal{M}^{10}(\F_2)$. Now, there are $21\cdot 20\cdot (4\cdot 3\cdot 2)\cdot (4\cdot 3\cdot 2)/2=|\mathsf{P\Gamma L}(3,4)|$ choices for the set $I\cup\{*,\circ\}$ in $X$. All these produce $X$ by the above algorithm in a unique way. Since a base change boils down to an element of $\PGL(11,2)$, this implies that the stabiliser of $X$ in $\PGL(11,2)$ has size at least $|\mathsf{P\Gamma L}(3,4)|$, and since the point-wise stabiliser must be trivial (as $X$ contains the frame $I\cup\{*,\circ,\Sigma\}$), we conclude that this stabiliser is isomorphic to $\mathsf{P\Gamma L}(3,4)$. 

Now define $\Xi$ as the family of $3$-spaces spanned by the members of $\Phi$, and still denote by $\mathcal{M}^{10}(\F_2)$ the pair $(X\Xi)$.  It is easy to verify that $\mathcal{M}^{10}(\F_2)$ satisfies (MM1) and (MM$2^*$): one only needs to verify (MM1) for one particular block, e.g., $\xi_1^*$  and (MM$2^*$) for two particular blocks, e.g., $\xi_1^*$ and $\xi_2^*$. 

Now let $N<10$. Then the 11 points $I\cup\{*,\circ\}$ are not linearly independent, and they are the projection of a base of $\PG(10,2)$ into $\PG(N,2)$, say from the subspace $U$. Since the rest of $X$ is determined uniquely by these eleven points by consecutively summing up sets of four already obtained points, the whole of $X$ is the projection from $U$ of $\mathcal{M}^{10}(\F_2)$. 

This completes the proof of the lemma.
\end{proof}

\par\bigskip
\textbf{The projective plane $(\mathcal{P}, \mathcal{L})\cong\PG(2,4)$}| For future reference, we give a brief description of the projective plane $(\mathcal{P}, \mathcal{L})$ that emerged in the above proof. Put $I=\{1,2,3,4,5,6,7,8,9\}$ and let $T^\circ=\{123,456,789\}$, $T^*=\{147,258,369\}$, $T^\Sigma=\{159,267, 348\}$ and $T=\{168,249,357\}$. For each pair $a,b$ occurring in a triple of $T$, and for each $\mathsf{x}\in\{\circ,*,\Sigma\}$, we let $T^\mathsf{x}(ab)$ be the unique element of $T^\mathsf{x}$ neither containing $a$, nor $b$. Then we have:

\[\mathcal{P} = I \cup \{\circ,*,\Sigma\} \cup \{\circ abc \mid \forall abc \in T^\circ\}  \cup \{abc* \mid \forall abc \in T^*\}  \cup \{\overline{abc} \mid \forall abc \in T^\Sigma\}\]
 \begin{multline*}
\mathcal{L} = \{ \{\circ abc,a,b,c,\circ\} \mid \forall \circ abc \in T^\circ\}  \cup \{ \{abc*,a,b,c,*\} \mid \forall abc* \in T^*\}  \\ \cup \{ \{\overline{abc},a,b,c,\Sigma\} \mid abc \in T^\Sigma\} 
 \cup \{ \{a,b,T^\circ(ab),T^*(ab),T^\Sigma(ab)\} \mid \forall abc \in T\} \}
 \end{multline*}

\par\bigskip

Now that the $N=10$ case is settled, we look at the lower dimensional cases. By the previous lemma these arise as projections of $\mathcal{M}^{10}(\F_2)$. So we search for subspaces of $\PG(10,2)$ from which to project $\mathcal{M}^{10}(\F_2)$.
We call a subspace $S$ \emph{admissible} when $S \cap \<\xi_1,\xi_2\>$ is empty for all blocks $\xi_1,\xi_2 \in \Phi$ of $\mathcal{M}^{10}(\F_2)$. Projecting from an admissible subspace yields a pair $(X_S,\Xi_S)$ (with obvious meaning) which still satisfies Axioms (MM1) and (MM2$^*$). Conversely, if these axioms are still satisfied after projecting from a subspace $S$, it means that $S$ is admissible.

\begin{lemma}\label{M}
Consider $(X,\Xi)=\mathcal{M}^{10}(\F_2)$ in $\PG(10,2)$. Then there is a unique line $M$ in $\PG(10,\K)$ from which the projection $(X_M,\Xi_M)$ of $\mathcal{M}^{10}(\F_2)$ is projectively equivalent to $\mathcal{V}_2(\F_2,\F_4)$. In this case, $M = T_x \cap T_y \cap T_z$ for any three points $x,y,z \in X$ not contained in a common elliptic space.
\end{lemma}

\begin{proof}
By Lemma~\ref{repr} and the existence of $\mathcal{V}_2(\F_2,\F_4)$ in $\PG(8,2)$ (which we view as an $8$-dimensional subspace of $\PG(10,2)$), we know that there is at least one such line $M$. 

Now, for each point $p$ in $\mathcal{V}_2(\F_2,\F_4)$,  the tangent space $T_p$ has dimension 4. 
We claim that, for each $x\in X$, $\dim(T_x)\geq6$. Indeed, since in Lemma~\ref{repr}, the point $\circ$ was arbitrary, it suffices to look at $T_\circ$, where we see that $T_\circ([\circ,1])=\<\circ,12,23\>$, $T_\circ([\circ,4])=\<\circ, 45, 56\>$ and $T_\circ([\circ,7])=\<\circ, 78, 89\>$. Hence $T_\circ$ contains the $6$-space $\<\circ,12,23,45,56,78,89\>$, showing the claim. 
Since the projection from $M$ onto $\PG(8,2)$ maps tangent spaces of $\mathcal{M}^{10}(\F_2)$ to tangent spaces of $\mathcal{V}_2(\F_2,\F_4)$, this implies that $M$ is contained in every tangent spaces of $\mathcal{M}^{10}(\F_2)$, and every such tangent space has dimension 6. 

We now establish uniqueness. It suffices to show the last assertion of the lemma. As above, we deduce that $T_*=\<*,14,47,25,58,36,69\>$ and $T_\Sigma=\<\Sigma, 95,51,62,27,84,43\>$. A straightforward  calculation shows that $\{124689, 135678,234579\}=T_* \cap T_\circ \cap T_\Sigma$. Since any three points $x,y,z\in X$ not contained in an elliptic space can play the role of $\circ$, $*$ and $\Sigma$, the last assertion follows.
\end{proof}

\begin{rem} \em The line $M$ could also be found as the intersection of all tangent hyperplanes: For each $\xi$ in $\Xi$, there is a hyperplane $H_\xi$ of $\PG(10,2)$, called a \emph{tangent hyperplane}, with the property $H_\xi \cap X = \xi \cap X$. \end{rem}

We now determine all admissible subspaces.  First a seemingly unrelated lemma.

\begin{lemma}\label{emptyset}
Let $(X,\Xi)\cong\mathcal{M}^{10}(\F_2)\subseteq\PG(10,2)$. Let $S\subseteq X$ with $1\leq |S|\leq 8$. If the sum of $S$ is $\overline{0}$, then either $S$ is the set of points on a line or $S$ is the symmetric difference of two distinct lines.
\end{lemma}

\begin{proof}
The assumption is equivalent with saying that $S$ is the union of disjoint frames of subspaces. Since no four points of $X$ are contained in a plane, and no three are collinear, a frame inside $X$ has at least five points. Since $|S|\leq 8$, $S$ has to be a frame itself. So $|S|\geq 5$. Suppose $|S|=5$ and assume for a contradiction that $S$ is not a block. If no triple of points of $S$ are contained in a common line, then $S$ is a non-degenerate conic. All such conics are projectively equivalent, and so we may assume $S=\{\circ,*,1,6,8\}$. As this is clearly not a frame, we may assume that three points of $S$ are on a common block $\xi$. But then the elliptic space spanned by the block $\xi'$ defined by the remaining pair $\{a,b\}$ intersects $\<\xi\>$ in a point $c$, with $a,b,c$ collinear. By (MM$2^*$), $c\in X$, a contradiction. Hence $S$ is a block if $|S|=5$. 

Now assume $|S|=6$. If no three points of $S$ are on a common line, then $S$ is a hyperoval, and all such things are projectively equivalent, hence we may take $S=\{\circ,*,\Sigma,1,6,8\}$, which is not a frame. Hence $S$ is not a hyperoval and there exist three points $a,b,c\in S$ on a common block $\xi$. Let $d,e$ be the remaining pair of points on $\xi$ (hence $\xi=\{a,b,c,d,e\}$). Then $a+b+c+d+e=\overline{0}$, and we can replace $\{a,b,c\}$ with $\{d,e\}$ in $S$, cancel double occurrences (which sum up to $\overline{0}$ already) to obtain a set $S'$ of either $5$, $3$ or $1$ point(s) that sum up to $\overline{0}$. From the foregoing, $|S'|=5$ and $S'$ coincides with the block defined by $d,e$; hence $S'=\xi$, contradicting the fact that $a,b,c\notin S'$. Consequently $|S|\neq 6$.

Now assume $|S|=7$. Then $S$ contains three points on a common line $\xi$, say $a,b,c$. We again replace these with the two remaining points   of $\xi$, cancel double occurrences, and obtain a set $S'$ of either $6$, $4$ or $2$ points whose sum is $\overline{0}$. However, such set does not exist by the foregoing. 

Now assume $|S|=8$. The same procedure as in the previous paragraph produces a set $S'$ of either $7$, $5$ or $3$ points whose sum is $\overline{0}$. By the foregoing, $S'$ is a block $\xi$, and we cancelled exactly one double occurrence. This means that $S$ contains exactly four points of a certain block $\xi'$, and also four points of $\xi$. 

The lemma is proved. 
\end{proof}

\begin{lemma}\label{adm}
Let $M$ be the intersection of all $T_x$, for $x\in X$, where $(X,\Xi)\cong\mathcal{M}^{10}(\F_2)\subseteq\PG(10,2)$. Then there are no admissible subspaces of dimension greater than $1$ and all admissible points and lines are contained in $\bigcup_{x \in X} (\<M,x\> \setminus \{x\})$. 
\end{lemma}

\begin{proof}
We determine the admissible points by counting the non-admissible ones. To that end, we introduce \emph{$X$-triangles} and \emph{$X$-quadrangles}: These are sets of three or four points of $X$, respectively, no three of which are contained in a common elliptic space. The \emph{center} of an $X$-triangle or $X$-quadrangle is the sum of its points. 

Note that $X$ does not contain a set of four coplanar points. Indeed, such a set is clearly not contained in a common elliptic space, and intersecting the elliptic spaces determined by two disjoint pairs of points produces a line contained in $X$, a contradiction. 

Now, the projection of an $X$-triangle from its center is a line; the projection of an $X$-quadrangle from its center is a set of four coplanar points. As in the previous paragraph, these sets cannot be contained in a structure that satisfies (MM1) and (MM$2^*$). Hence no center of an $X$-triangle or $X$-quadrangle is admissible. We now show the converse statement.

\textbf{Claim 1}: \emph{Each non-admissible point which is not contained in an elliptical space is either the midpoint of an $X$-triangle or the midpoint of at least two $X$-quadrangles.}\\
Let $p\in \PG(10,\K)$ be non-admissible. Recall that this means $p \in \<\xi_1,\xi_2\>$ for some $\xi_1,\xi_2 \in \Xi$ with $\xi_1 \neq \xi_2$. 
Put $x=\xi_1 \cap \xi_2$. If $p \notin \xi_1 \cup \xi_2$, then there are unique lines $L_1 \subseteq \xi_1$ and $L_2 \subseteq \xi_2$ through $x$  such that $p \in \<L_1,L_2\>$. Let $i\in\{1,2\}$. If $L_i$ is a secant of $X(\xi_i)$, then we denote by $y_i$ the unique point on $L_i \cap X\setminus\{x\}$; if $L_i$ is tangent to $X(\xi_i)$ and then there are two planes, say $Z_i$ and $\overline{Z}_i$ through $L_i$ not tangent to $X(\xi_i)$, and we denote by $z'_i$ and $z''_i$ the points of $Z_i \cap X\setminus\{x\}$ and by $z_i$ the intersection point $L_i \cap \<z'_i,z''_i\>$ (clearly, $z_i \neq x$), likewise for $\overline{Z}_i$ (note that $z_i \neq \overline{z}_i$ since $X$ does not contain a set of four coplanar points). There are four possibilities. 

\begin{compactenum}
\item \textit{Both $L_1$ and $L_2$ are secants and $p \in \<y_1,y_2\>$.} Then $p$ belongs to $[y_1,y_2]$. 
\item \textit{Both $L_1$ and $L_2$ are secants and $p \notin \<y_1,y_2\>$.} In this case, $p$ is the center of the $X$-triangle $\{x,y_1,y_2\}$.
\item \textit{The line $L_1$ is a secant whereas $L_2$ is a tangent (possibly switching $\{1,2\}$).} Without loss, the plane $Z_2$ is such that $z_2 \in \<y_1,p\>$. Then $p$ is the center of  the $X$-triangle $\{y_1,z'_2,z''_2\}$, since $z_1'+z_1''=z_1$ and $z_1+y_1=p$.
\item \textit{Both $L_1$ and $L_2$ are tangents.} Now we have $L_i=\{x,z_i,\overline{z}_i\}$, $i=1,2$ and we can choose notation so that $\{p,z_1,z_2\}$ and $\{p,\overline{z}_1,\overline{z}_2\}$ are lines. Hence $p=z_1+z_2=z_1'+z_1''+z_2'+z_2''$ and $p=\overline{z}_1+\overline{z}_2=\overline{z}_1'+\overline{z}''_1+\overline{z}_2'+\overline{z}''_2$. Hence $p$ is the midpoint of the two $X$-quadrangles $\{z'_1,z''_1,z'_2,z''_2\}$ and $\{\overline{z}'_1,\overline{z}''_1,\overline{z}'_2,\overline{z}''_2\}$ (called \emph{ complementary} $X$-quadrangles). 
\end{compactenum}


Claim 1 is proved. In order to be able to count the number of non-admissible points, it now suffices to determine how many times a point can arise as center of an $X$-triangle or $X$-quadrangle. This is the content of the next two claims. 

\textbf{Claim 2:} \emph{ If $p$ is the center of an $X$-triangle $\{x,y,z\}$, then $p$ cannot be the center of a second $X$-triangle $\{x',y',z'\}$.} \\
Indeed, suppose for a contradiction that $p=x+y+z=x'+y'+z'$, with $x,y,z,x',y',z'\in X$. Then $x+x'+y+x'+z+z'=\overline{0}$ (and $|\{x,y,z,x',y',z'\}|\in\{2,4,6\}$), contradicting Lemma~\ref{emptyset}. 

\textbf{Claim 3:} \emph{ If $p$ is the center of an $X$-triangle $\{x,y,z\}$, then $p$ cannot be the center of an $X$-quadrangle $\{a,b,c,d\}$.}\\
Indeed, as above we obtain $x+y+z+a+b+c+d=\overline{0}$, and hence, by Lemma~\ref{emptyset}, $(\{x,y,z\}\cup\{a,b,c,d\})\setminus(\{x,y,z\}\cap\{a,b,c,d\})$ is a block. So we may assume $d=z$ and $\{a,b,c,x,y\}$ is a block. This contradicts the fact that an $X$-quadrangle does not contain three collinear points by definition.  


\textbf{Claim 4:} \emph{ If $p$ is the center of an $X$-quadrangle $\{a,b,c,d\}$, then $p$ is the center of precisely three other $X$-quadrangles.} \\
Let $\{a',b',c',d'\}$ be a second $X$-quadrangle with center $p$. For $x,y\in\{a,b,c,d\}$ denote by $\xi_{a,b}$ the block containing $a,b$. 
Now $a+b+c+d+a'+b'+c'+d'=\overline{0}$. By Lemma~\ref{emptyset}, $\{a,b,c,d,a',b',c',d'\}$ is the symmetric difference of two distinct lines. Since $\{a,b,c,d\}$ intersects each of these lines in exactly two points, $a',b',c',d'$ are either the points of $\xi_{a,b}$ and $\xi_{c,d}$ distinct from $a,b,c,d$ and $\xi_{a,b}\cap\xi_{c,d}$, or the points of $\xi_{a,c}$ and $\xi_{b,d}$ distinct from $a,b,c,d$ and $\xi_{a,c}\cap\xi_{b,d}$, or the points of $\xi_{a,d}$ and $\xi_{b,c}$ distinct from $a,b,c,d$ and $\xi_{a,d}\cap\xi_{b,c}$. Conversely, all of these possibilities give rise to an $X$-quadrangle with center $p$. The claim follows. 

A straightforward count now reveals that there are $\frac{21\cdot20\cdot16}{3\cdot2\cdot1}=1120$ $X$-triangles and $\frac{1120\cdot9}{4}=630\cdot 4$ $X$-quadrangles. Lastly, there are $21\cdot10=210$ points contained in elliptical spaces but not in $X$ and $21$ points in $X$. This amount to $1981$ non-admissible points, so we miss exactly $66$ of the $2047$ points of $\PG(10,\K)$. This is exactly the number of points contained in the union of $\<M,x\>\setminus\{x\}$ with $x$ varying in $X$. Moreover, all admissible points should be contained in such planes, as they have to ``disappear'' after projecting from $M$, since $\mathcal{V}_2(\F_2,\F_4)$ contains no admissible points. This shows the lemma.
\end{proof}

\par\bigskip

\textbf{Action of $\PGL(3,4)$ and $\PSL(3,4)$ on $\mathcal{M}^{10}(\K)$}---Recall that the line $$M=\{124689, 135678,234579\}$$ of $\PG(10,2)$ is fixed under the action of $\PGL(3,4)$ on $\PG(10,2)$ stabilising $X$. Note that the point sets $\{1,2,4,6,8,9\}$, $\{1,3,5,6,7,8\}$ and $\{2,3,4,5,7,9\}$ are disjoint hyperovals of $\PG(2,4)$. Each of them spans a subspace intersecting $M$ in a different point. Since $\PGL(3,4)$ is transitive on the hyperovals, it is also transitive on the points of $M$. Then the stabiliser of a point of $M$ in $\PGL(3,4)$ is a subgroup of $\PGL(3,4)$ of index $3$ and as such a copy of $\PSL(3,4)$ (which indeed has three orbits on the set of hyperovals). Since $\PSL(3,4)$ has no index 2 subgroups, it also fixes the two other points of $M$. Since each point stabiliser in $\PGL(3,4)$ contains elements of $\PGL(3,4)\setminus\PSL(3,4)$, the group $\PGL(3,4)$ has exactly two orbits on the set of admissible points, namely $M$ and the set of the other 63 admissible points. It is now also easy to see that it has two orbits on set of admissible lines: $\{M\}$, and the set of 63 other lines.

Moreover, if we project $\mathcal{M}^{10}(\K)$ from a point on $M$, then all tangent spaces $T_x$, $x\in X$, get mapped into $5$-dimensional spaces, whereas this is only true for the tangent space $T_y$ if we project from a point of $\<y,M\>\setminus(M\cup\{y\})$, with $y\in X$. Hence these two projections cannot be isomorphic. A similar argument shows that the projection from $M$ is projectively inequivalent to  the projection from any other admissible line. 
  
\textbf{Conclusion}---For any pair $(X,\Xi)$, where $X$ is a spanning point set of $\PG(N,\K)$, $N >3$, satisfying (MM1) and (MM2$^*$), Lemmas~\ref{repr},~\ref{M} and~\ref{adm} and the previous discussion show that $8 \leq N \leq 10$ and, more precisely:
\begin{compactenum}
\item[($N=10$)] $(X,\Xi)\cong\mathcal{M}^{10}(\K)$ if $N=10$;
\item[($N=9$)] $(X,\Xi)$ is the projection of $\mathcal{M}^{10}(\K)$ from one of its 66 admissible points $p$, and there are two non-isomorphic projections, depending on $p\in M$ 
or $p\notin M$%
;
\item[($N=8$)] $(X,\Xi)$ is either the projection 
from $M$ 
and then we obtain $\mathcal{V}_2(\F_2,\F_4)$, or it is the projection of $\mathcal{M}^{10}(\K)$ from one of the 63 admissible lines distinct from $M$.
\end{compactenum}
\par\bigskip
This shows Proposition~\ref{K=2,d=2}.

\begin{rem}\em
In fact, $\mathcal{M}^{10}(\K)$, i.e., the universal pseudo embedding of $\PG(4,2)$, arises from a projection of the universal pseudo embedding of the Witt design $S(5,8,24)$. Indeed, embed $\mathcal{M}^{10}(\K)$ in a hyperplane $\mathcal{H}$ of $\PG(11,\K)$ and take any point $p \in \PG(11,\K)$ in the complement of $\mathcal{H}$. Note that we can uniquely \emph{lift} each point $h$ of $\mathcal{H}$ using $p$, by defining the lift $\ell_p(h)$ of $h$ as the unique point on the line $ph$ distinct from $p$ and $h$. Let  $B$ be any block of $\mathcal{M}^{10}(\K)$. One can verify that the points of $\mathcal{M}^{10}(\K) \setminus B$ together with the lifted points $\ell_p(B \cup M)$, equipped with all the frames of $6$-subspaces formed by 8-subsets, gives the Witt design $S(5,8,24)$. The stabiliser of that 24-set in $\PGL(12,2)$ is exactly the Mathieu group $\mathsf{M}_{24}$.

This set of 24 points can also be obtained as follows: Let $V$ be the 24-dimensional free $\mathbb{Z}_2$ module on the points of the Witt design $S(5,8,24)$. We factor out the submodule generated by all the characteristic vectors of blocks (octads). Since the (extended) binary Golay code had dimension 12, the factor module has dimension $24-12=12$, and the standard basis of $V$ gets mapped onto a set of 24 points on which the group $\mathsf{M}_{24}$ acts naturally. Projectively, we obtain an $11$-dimensional projective space with an action of $\mathsf{M}_{24}$ on a set of 24 points structured as $S(5,8,24)$ where the octads are contained in $6$-spaces. 

Conversely, if one starts with the 24 points of the universal embedding of the Witt design $S(5,8,24)$ in $\PG(11,\K)$, one obtains $\mathcal{M}^{10}(\K)$ by choosing three arbitrary points $p_1,p_2,p_3$ in $S(5,8,24)$ and projecting from the point $p_1+p_2+p_3$. The image of the 21 remaining points of $S(5,8,24)$ then gives $\mathcal{M}^{10}(\K)$, and the projection of the points $p_1,p_2,p_3$ then gives the line $M$. In view of the above, this also means that if we project from the plane $\<p_1,p_2,p_3\>$, then we get the Veronese variety $\mathcal{V}_2(\F_2,\F_4)$. 
\end{rem}

\section{Vertex-reduced Hjelmslevean Veronese sets}\label{maincase} 
Henceforth, we assume that $(X,\Xi)$ in $\PG(N,\mathbb{K})$ has $v \geq 0$ and satisfies property~($\mathsf{V}$). The latter for instance implies that two distinct tubes intersect in either a unique point or a generator:

\begin{lemma}\label{intersectionoftubes2}
Let $C,C'$ be two distinct tubes. Then $C\cap C'$ is either a point of $X$ or a generator of both $C$ and $C'$. \end{lemma}
\begin{proof}
By (H2$^*$), there is a point $x\in X$ contained in both $C\cap C'$. If $C \cap C'=\{x\}$ we are done, so suppose $\overline{C}\cap \overline{C'}$ contains a point $y \in Y$. So $y$ is contained in the respective vertices $V$ and $V'$ of $C$ and $C'$, which means by Property~($\mathsf{V}$) that $V=V'$. Hence $C$ and $C'$ share the generator determined by $x$ and $V$ in this case. Since no point of $C\setminus\<x,V\>$ is collinear with $x$, it is clear by (H2$^*$) that $\Xi(C) \cap \Xi(C')= \<x,V\>$.
\end{proof}

\subsection{Local properties and the structure of $Y$}

We investigate the singular subspaces $\Pi_x$ for $x\in X$ (see Definition~\ref{defpix2}) and the set of tubes $\mathcal{C}_V$ going through a fixed vertex $V$. This brings along the structure of the vertex set $Y$.

\begin{lemma}\label{Pi_x2}
For each $x\in X$, there are tubes $C$ and $C'$ with $C \cap C' =\{x\}$ and, if $V$ and $V'$ are the vertices of $C$ and $C'$ respectively, then $\Pi_x=\<x,V,V'\>$. In particular, $\dim(\Pi_x)=2v+2$. 
\end{lemma}

\begin{proof}
Let $C$ be a tube through $x$ and let $V$ be its vertex. Suppose for a contradiction that no tube through $x$ intersects $C$ in $\{x\}$ only. Then Lemma~\ref{intersectionoftubes2} implies that all tubes through $x$ contain $V$. But then, for each point $x' \in X$ distinct from $x$,  the tube through $x$ and $x'$ also  has $V$ as its vertex, so $x'$ is also collinear with $V$. Consequently,  by Corollary~\ref{pix2}, all tubes have $V$ as their vertex. This contradiction to ($\mathsf{V}$) implies that there is a tube $C'$ with $C \cap C'=\{x\}$. Denote its vertex by~$V'$.

We now show that $\Pi_x=\<x,V,V'\>$. If not, there is a tube $C''$ through $x$ with vertex $V'' \nsubseteq \<V,V'\>$. By ($\mathsf{V}$), $C \cap C' = C \cap C'' = \{x\}$. Take a point $z \in C$ not collinear to $x$ (so not contained in $\Pi_x$) and a point $z'  \in C'\setminus \{x\}$ collinear to $x$ (so contained in $\<x,V'\> \subseteq \Pi_x$), chosen in such a way that $\<z',V\>$ is disjoint from $\<x,V'' \cap \<V,V'\>\>$ (note that $V$ and $\<x,V'\>$ span $\<x,V,V'\>$ whereas $V'' \cap \<V,V'\>$ and $\<x,V'\>$ do not, by assumption on $V''$). By Corollary~\ref{pix2}, $z$ and $z'$ are not collinear and $\overline{\Pi}_z \cap \overline{\Pi}_{z'} = V$ is the vertex of the tube $[z,z']$.  Axiom (H2$^*$) implies that $[z,z']$ intersects $C''$ in a point $z''\in X$. If $z'' \perp z'$, then $z''$ belongs to the generator $\<z',V\>$ of $[z,z']$. But then $z'' \in \<z',V\> \cap C''$, or more precisely, $z''$ belongs to $\<z',V\> \cap \<x,V'' \cap \<V,V'\>\>$, which is empty by the choice of $z'$. Hence $z''$ is not collinear to $z'$ and hence by Corollary~\ref{pix2}, $z''$ is not collinear to $x$, so $C'' = [x,z'']$ and $V''=\overline{\Pi}_{x} \cap \overline{\Pi}_{z''}=V$, a contradiction. The lemma is proven.
\end{proof}

\textbf{Notation}|
Given a tube $C$, we denote  $\Pi_C^Y=\{\Pi_x^Y \mid x \in C\}$ and $\Pi_C=\{\Pi_x \mid x \in C\}$. For a given vertex $V$, the set of all tubes with vertex $V$ is $\mathcal{C}_V$; its structure is described in the following lemma. 

\begin{lemma}\label{Y}
Let $C$ be a tube with vertex $V$. Then  each point of $X$ collinear to $V$ is contained in a unique member of $\Pi_C$  and hence $\<\mathcal{C}_V\>=\<\Pi_C\>=\<C,\Pi_C^Y\>$. This has the following consequences:
\begin{compactenum}[$(i)$]
\item For each tube $C'$ with vertex $V$, containment gives a bijection between the set of generators of $C'$ and the set $\Pi_C$. Consequently, $\Pi_C^Y = \Pi_{C'}^Y$ and collinearity gives a bijection between the generators of $C$ and $C'$. 
\item For each point $z$ not collinear to $V$;  $V$ and $\Pi^Y_z$ are complementary subspaces in $\<\Pi_C^Y\>$; in particular, $\dim(\<\Pi_C^Y\>)=3v+2$;
\item For all non-collinear points $x,x'\in C$, the subspace $\<\Pi_C^Y\>$ is spanned by $\Pi_x^Y$ and $\Pi_{x'}^Y$;
\item Each point $y \in Y$ belongs to $\Pi_c^Y$ for some $c\in C$.
\item $\<\Pi_C^Y\> \cap \<C\>=V$ and the dimension of $\<\mathcal{C}_V\>$ is $3v+d+4$.
\end{compactenum}
\end{lemma}

\begin{proof}
Let $p \in X$ be collinear to $V$ and suppose $p$ is not contained in any member of $\Pi_C$. Take any point $q \in X\setminus (C \cup \Pi_p)$. By (H2$^*$), the tube $[p,q]$ then intersects $C$ in a point $c$, and since $p \notin \Pi_c$ we have $[p,q]= [p,c]$, which implies that $V$ is the vertex of $[p,q]$. In particular, $q$ is collinear with $V$. We obtain that all points of $X$ are collinear to $V$, and hence all tubes have $V$ as their vertex, contradicting ($\mathsf{V}$). We conclude that $p \in \Pi_c$ for some $c\in C$, uniqueness follows from the fact that the sets $\Pi_c$ are mutually disjoint.

Generated by all points of $X$ collinear to $V$,  $\<\mathcal{C}_V\>$ equals $\<\Pi_C\>$. Since for each $c\in C$, $\<\Pi_c\> = \overline{\Pi}_c=\<c,\Pi_c^Y\>$, we obtain $\<\mathcal{C}_V\> = \<C,\Pi_C^Y\>$.  

$\quad (i)$ It immediately follows from the above that each point $c' \in C'$ is contained in a unique member of $\Pi_C$. Observing that collinear points of $C'$ need to be contained in the same such subspace and that points in the same such subspace are necessarily collinear, each generator of $C'$ is contained in a unique member of  $\Pi_C$.

Interchanging the roles of $C$ and $C'$, we see that each generator of $C$ is contained in a member of $\Pi_{C'}$ and so $\Pi_C=\Pi_{C'}$. 
Hence also $\Pi_C^Y=\Pi_{C'}^Y$. 

We have shown that each tube through $V$ has exactly one generator in each member of $\Pi_C$, and all members of $\Pi_C$ contain a generator of that tube. It follows that the map from $C$ to $C'$ taking a generator of $C$ to the unique generator of $C'$ contained in the same member of $\Pi_C$ (i.e., the generator of $C'$ collinear with it) is a bijection.

$\quad (ii)$ Let $z$ be a point not-collinear with $V$ (note that this exists by Property ($\mathsf{V}$)). Such a point is not collinear with any point $c\in C$ and hence, for each $c\in C$, the intersection $\Pi_c^Y \cap \Pi_{z}^Y$ is a $v$-space $V_c$. Again by ($\mathsf{V}$), $V_c \cap V = \emptyset$; so also $\Pi_z^Y \cap V = \emptyset$. Moreover, Lemma~\ref{Pi_x2} implies $\Pi_c^Y=\<V_c,V\>$ for each $c\in C$, and also $\Pi_z^Y=\<V_{c_1},V_{c_2}\>=\<V_c \mid c\in C\>$. So $\<\Pi_C^Y\>=\<\Pi_c^Y \mid c\in C\>=\<V, V_c \mid c\in C\>=\<V,\Pi_{z}^Y\>$. In particular, $\dim(\<\Pi_C^Y\>)=3v+2$.

$\quad (iii)$ This follows immediately from the previous item since $\<V,\Pi_{z}^Y\>=\<V, V_{c_1},V_{c_2}\>=\<\Pi_{c_1}^Y,\Pi_{c_2}^Y\>$.

$\quad (iv)$  For an arbitrary point $y \in Y$, we have $y \in \Pi_z^Y$ for some $z \in X$. If $z$ is collinear with $V$ then $\Pi_z^Y = \Pi_c^Y$ for some $c \in C$. If $z$ is not collinear with $V$, then in the previous paragraph, we showed that $\<\Pi_C^Y\>=\<V,\Pi_{z}^Y\>$.  Hence $y\in \<\Pi_C^Y\>$ and this already implies $Y=\<\Pi_C^Y\>$. Now take a tube $C'$ through $zy$. 
By (H2$^*$), $C \cap C'$ contains a point $c \in X$ and hence $y \in 
\Pi^Y_c$.


$\quad (v)$  As $\<\Pi_C^Y\>$ only contains points of $Y$, the intersection $\<C\>\cap\<\Pi_C^Y\>$ coincides with $V$. Hence $\<\mathcal{C}_V\>$, generated by a quadric $Q$ on $C$ complimentary to $V$ and by $\Pi_C^Y$,  has dimension $3v+d+4$.
\end{proof}

\textbf{Notation}|Point $(i)$ of the previous lemma implies that $\Pi_C^Y$ could, and shall, more accurately be denoted by $\Pi_V^Y$, with $V$ the vertex of $C$, as it does not depend on the element of $\mathcal{C}_V$.

 We consider the following point-line geometry.

\begin{defi}[The point-line geometry $\mathcal{G}_V$]  \em We define  $\mathcal{G}_V$ as a geometry having as point set the set $\mathcal{P}_V$ of singular affine $(v+1)$-spaces having $V$ as their $v$-space at infinity and as line set the set $\mathcal{C}_V$ of tubes with vertex $V$, with containment made symmetric as incidence relation.
\end{defi}

\begin{cor}\label{dualaffineplane}
The point-line geometry $\mathcal{G}_V=(\mathcal{P}_V,\mathcal{C}_V,I)$ is a dual affine plane.
\end{cor}

\begin{proof}
Any two tubes through $V$ intersect in an element of $\mathcal{P}_V$ as they need to share a point of $X$ by (H2$^*$). Let $C \in \mathcal{C}_V$ be arbitrary. Each singular affine $(v+1)$-space $W \in \mathcal{P}_V$, not contained in $C$, is collinear with a unique generator of $C$ by Lemma~\ref{Y}. That generator corresponds to the unique element of $\mathcal{P}_V$ in $C$ that is not contained in a member of $\mathcal{C}_V$ together with $W$. Clearly, no element of $\mathcal{P}_V$ is contained in all tubes through ~$V$. We verified all axioms of a dual affine plane.
\end{proof}

\subsection{Connecting $X$ and $Y$}

Recall that $Y$ is a subspace (cf.\ Corollary~\ref{Ysub}). The connection between the $X$-points and the $Y$-points is a crucial step towards understanding the structure of~$(X,\Xi)$. To this end, we consider the projection $\rho$ of $X$ from $Y$ onto a subspace $F$ of $\PG(N,\K)$ complementary to $Y$, i.e.,
\[\rho: X \rightarrow F : x \mapsto \<Y,x\> \cap F.\]

We show that this projection gives us a well-defined point-quadric set (the quadrics $\rho([x,z])$ and $\rho([x',z'])$ with $\rho(x)=\rho(x')$ and $\rho(z)=\rho(z')$ have to coincide).  For that we need one more general lemma.

\begin{lemma}\label{differentvertex}
Let  $C$ and $C'$ be tubes sharing only one point $x\in X$. Then two points $z \in C$ and $z'\in C'$ are collinear if and only if $\<z,z'\>$ belongs to $\Pi_x$.
\end{lemma}

\begin{proof}
Denote by $V$ and $V'$ the respective vertices of $C$ and $C'$. By Lemma~\ref{Pi_x2}, $\Pi_x=\<x,V,V'\>$. Consider two distinct points $z\in C \setminus\<x,V\>$ and $z' \in C'\setminus\<x,V'\>$. If $z$ and $z'$ would be collinear, then $\Pi_z^Y=\Pi_{z'}^Y$ by Corollary~\ref{pix2}, in particular this means that $V'$ belongs to $\Pi_{z}^Y$. This contradicts the fact that, also by Corollary~\ref{pix2}, $\Pi_z^Y \cap \Pi_x^Y = V$. We conclude that $z$ and $z'$ are not collinear. The converse statement is clear.
\end{proof}

\begin{lemma}\label{rho}
The projection $\rho$ is such that $\rho(x)=\rho(x')$, for $x,x'\in X$, if and only if $x$ and $x'$ are equal or collinear. In particular, for each $x\in X$, we have $\rho^{-1}(\rho(x))=\Pi_x$ and for any tube $C$ with vertex $V$, $\rho(C)$ is a $Q^0_d$-quadric. For any two tubes $C,C'$ we have that $\rho(C)=\rho(C')$ if and only if $C$ and $C'$ have the same vertex; and if the vertices are distinct, then $\rho(C) \cap \rho(C')=\rho(C\cap C')$. In particular, with slight abuse of notation, $\rho^{-1}(\rho(C))=\mathcal{C}_V$. \end{lemma}

\begin{proof}
Let $x$ and $x'$ be two points of $X$. Then $\rho(x)=\rho(x')$ (or equivalently, $\<Y,x\>=\<Y,x'\>$) if and only if $xx'$ contains a point of $Y$, which on its turn is equivalent with $x$ and $x'$ being collinear.  It is then clear that $\rho^{-1}(\rho(x))$ equals the set of points collinear with $x$, so $\Pi_x$. Now let $C$ be a tube with vertex $V$.
Since $\<C\> \cap Y = V$, we obtain that $\rho(C)$ is a quadric of type $Q^0_d$. It follows from Lemma~\ref{Y}$(i)$ that all tubes in $\mathcal{C}_V$ have the same image, as collinear generators are mapped onto the same point. If $C$ and $C'$ have distinct vertices $V$ and $V'$ (hence $V \cap V' = \emptyset$ by $(\mathsf{V})$) then $C \cap C'$ is a unique point $x$ by (H2$^*$) and hence it follows from Lemma~\ref{differentvertex} that $\rho(C) \cap \rho(C')=\rho(x)$.
\end{proof}
\bigskip

We now show that $(\rho(X),\rho(\Xi))$, as a pair of points and $Q^0_d$-quadrics in $F$, satisfies the Axioms~(MM1) and ~(MM2$^*$) introduced in Section~\ref{MM}. 

\begin{prop}\label{veronese}
The pair $(\rho(X),\rho(\Xi))$ satisfies Axioms \emph{(MM1)} and \emph{(MM2$^*$)}. As a point-line geometry, $(\rho(X),\rho(\Xi))$ is hence isomorphic to $\PG(2,\mathbb{B})$, where $\mathbb{B}$ is a quadratic alternative division algebra with $\dim_\K(\mathbb{B})=d$. Consequently, $d$ is a power of $2$, with $d\leq 8$ if $\mathrm{char}(\K) \neq 2$, and $N=6d+2$.
\end{prop}

\begin{proof}
By Lemma~\ref{rho}, $\rho(\Xi)$ is a well-defined family of $(d+1)$-dimensional subspaces in $F$ (called the \emph{elliptic spaces}), such that for each $\xi \in \Xi$, $\rho(\xi) \cap \rho(X)$ contains $\rho(X(\xi))$ (equality will be shown once (MM2$^*$) is established). We prove that the pair $(\rho(X),\rho(\Xi))$ satisfies Axioms (MM1) and (MM2$^*$) and as such is a Veronese variety (cf.~Theorem~\ref{main2}).

$\quad \bullet$ \textit{Axiom} (MM1). Let $z$ and $z'$ be distinct points of $\rho(X)$. Then there are points $x,x' \in X$ with $z=\rho(x)$ and $z'=\rho(x')$. By the above, $x$ and $x'$ are not collinear, so by (H1), they are contained in a unique tubic space $\xi$. Hence $z$ and $z'$ are contained in the elliptic space $\rho(\xi)$ and Axiom (MM1) follows.

$\quad \bullet$ \textit{Axiom} (MM2$^*$). Let $\xi, \xi' \in \Xi$ be distinct tubic spaces and put $\rho(\xi)=\zeta$, $\rho(\xi')=\zeta'$ (note that $\zeta=\zeta'$ is a priori not impossible), and put $C=X(\xi)$ and $C'=X(\xi')$. If the respective vertices $V$ and $V'$ of $C$ and $C'$ coincide, then Lemma~\ref{Y}$(i)$ implies that $\rho(C)=\rho(C')$, and hence there is nothing to show. 
So suppose that $V$ and $V'$ are distinct, and hence disjoint by ($\mathsf{V}$). Axiom (H2$^*$) implies that $C \cap C'$ is a unique point $x\in X$ and by Lemma~\ref{rho} we obtain $\rho(C) \cap \rho(C')=\{\rho(x)\}$. For (MM2$^*$) to hold, we have to show that $\zeta \cap \zeta' = \{\rho(x)\}$ too. So it suffices to show that $\<C,C'\> \cap Y=\<V,V'\>$: in this case, the projection of $\<C,C'\>$ from $Y$ is then isomorphic to the projection of $\<C,C'\>$ from $\<V,V'\>$, and hence (MM2$^*$) follows from (H2$^*$). 

Suppose for a contradiction that $\<C,C'\> \cap Y$ contains a point $y \notin \<V,V'\>$.
By Lemma~\ref{Y}$(iii)$, $y$ is collinear to unique generators $\<z,V\> \subseteq C$ and $\<z',V'\> \subseteq C'$ for some points $z \in C$ and $z'\in C'$. Note that $x$ is not contained in those generators as $y \notin \<V,V'\>=\Pi_x^Y$.  Since $y \notin \<C'\>$, it is clear that $\<C',y\>$ intersects $\<C\>$ in a line $L$ through $x$. Moreover, $L$ is disjoint from the singular line $\<z',y\>$, for no point of $C$ is collinear to $z'$ by Lemma~\ref{differentvertex}. 
So,  $\<L,y,z'\>$ is a $3$-space in $\<C',y\>$, which thus has a plane $\alpha$ in common with $\<C'\>$ (note that $L$ and $\<z',y\>$ do not belong to $\<C'\>$, so neither to $\alpha$).

The plane $\alpha$ contains $x$ and $z'$ and hence $\alpha \cap C'$ is either a conic through $x$ and $z'$, or it is the union of two lines through $x$ and $z'$ respectively, having a point $v'$ of $V'$ in common.
In both cases, there is only one line in $\alpha$ through $x$ which does not contain a unique second point of $C'$ (in the first case, the tangent line through $x$ to $\alpha \cap C'$; in the second case, the line~$\<x,v'\>$).
Take any line $L'$ in $\alpha$ through $x$ having a unique second point $r'$ in common with $C'$. The plane $\<L,L'\>$ intersects the singular line $\<y,z'\>$ in a point $s$. There are at least three valid choices for $L'$ since $|\K|>2$, each yielding another point $s \in \<y,z'\>$. So we can choose $L'$ such that $s \notin \{y,z'\}$. In particular,   $s \neq r'$ since $s \in L'$ only occurs if $s=z'$ (as $\<y,z'\> \cap \alpha=\{z'\}$).

The lines  $\<s,r'\>$ and $L$, contained in the plane $\<L,L'\>$, share a point $r$. As $L$ and $\<y,z'\>$ were disjoint, $r \neq s$, and as $r' \neq x$, we also have $r \neq r'$.  Note that $r$ is contained in the intersection of $\xi$ and any tubic space containing $r'$ and $s$; hence $r\in C\subseteq X$.  So the line $\<s,r'\>$ contains three points in $X$ and is hence singular. But then $\<r,r'\>$ needs to be contained in $\Pi_x$ by Lemma~\ref{differentvertex}, implying that also $s\in \Pi_x$ and hence $y\in \Pi_x^Y$ as well, a contradiction.

$\quad \bullet$ \emph{Claim: The intersection $\rho(\xi) \cap \rho(X)$ equals $\rho(X(\xi))$ for each $\xi \in \Xi$.}\\
Put $C=X(\xi)$. Suppose for a contradiction that $\rho(\xi)$ contains a point $\rho(z)$ with $z\notin \rho^{-1}(\rho(C))$, i.e., $z \notin \mathcal{P}_V$.  Take any point $x\in X$ with $\rho(x) \in \rho(C)$. Then $[z,x]$ is a tube with vertex $V'$ distinct (and hence disjoint) from $V$. So by (H2$^*$), $[z,x] \cap \xi = \{x\}$. By Lemma~\ref{rho} and Axiom (MM2$^*$), $\rho([z,x]) \cap \rho(\xi) = \{\rho(x)\}$; yet this intersection contains the line $\<\rho(z),\rho(x)\>$ by construction. This contradiction shows the claim.

Knowing this, it follows from Theorem~\ref{main2} that, as a point-line geometry, $(\rho(X),\rho(\Xi)$) is isomorphic to a projective plane $\PG(2,\mathbb{B})$, where $\mathbb{B}$ is a quadratic alternative division algebra with $\dim_\K(\mathbb{B})=d$. Consequently, $d$ is a power of 2, smaller or equal than 8 if $\mathrm{char}(\K) \neq 2$. Since $\rho(X)$ spans $F$ (because $X$ spans $\PG(N,\K)=\<F,Y\>)$, this result also implies that $\dim(F)=3d+2$. Together with $\dim(Y)=3v+2=3d-1$, it follows that $N=6d+2$.
\end{proof}

The next corollary now immediately follows from Theorem~\ref{main2}.

\begin{cor}\label{ovoids}
All $Q_d^0$-quadrics are quadrics of Witt index $1$.
\end{cor}

\begin{defi}[The connection map] \em We define $\chi$ as the map from $\rho(X)$ to $Y$, taking a point $z =\rho(x)$ to the subspace $\Pi_x^Y$ at infinity, i.e.,
\[ \chi: \rho(X) \rightarrow Y: \rho(x) \mapsto \Pi^Y_x.\]
\end{defi}
Note that Lemma~\ref{rho} assures that this map is well defined: points with the same image under $\rho$ are collinear and hence determine the same subspace $\Pi_x^Y$.
The following proposition contains an important local property of $\chi$. The proof uses the notion of regular $d$-scrolls, the definition and properties of which we have recorded in Appendix~\ref{scrolls} (see Definition~\ref{regscroll} and Lemma~\ref{regularscroll}), preceded by some auxiliary properties. 

\textbf{Notation} Let $V$ be some fixed vertex and $C$ a fixed tube belonging to $\cC_V$. We now choose the subspace $F$ (complementary to $Y$) such that is contains a $Q_d^0$-quadric $Q$ of $C$ (and so $\<Q,V\>=\<C\>$ and $\rho(C)=Q$).    We define $\rho_V$ as the projection from $V$ onto a complementary subspace $\widetilde{F}$ containing $F$. Then $\rho_V(C)=Q$.  Let $X_V$ be the set of points of $X$ collinear to $V$. By Lemma~\ref{rho}, $\rho(X_V)$ is a $Q_d^0$-quadric, and it obviously coincides with $Q$.  For any point $x\in X_V$, we denote $\Check{x}:=\rho(x)\in Q$ and $\widetilde{x}:=\rho_V(x)\in\widetilde{F}$. 
Also, we denote $\widetilde{Y}:=\rho_V(Y)=Y\cap\widetilde{F}$.

\begin{prop}\label{structureCV}
Let $V$ be a vertex.  Then, firstly, the set $\{\rho_V({\Pi}_{x}^Y) \mid x \in \mathcal{P}_V\}$ induces a regular spread $\mathcal{R}_V$ of $v$-spaces on $\widetilde{Y}$ and the (well-defined) map
\[\chi_V: \rho(X_V) \rightarrow \mathcal{R}_V: \Check{x} \mapsto R_{\Check{x}}:=\rho_V({\Pi}_{x}^Y)\]
takes a conic of $Q$ onto a regulus of $\mathcal{R}_V$  and its restriction to such a conic preserves the cross-ratio (i.e., $\chi_V$ is a projectivity between $Q$ and $\mathcal{R}_V$). 
Secondly, the regular spread $\mathcal{R}_V$, the quadric $Q$ and the map $\chi_V$ determine a regular $d$-scroll $\mathfrak{R}_d(\K)$ in $\widetilde{F}$ and for each tube $C' \in \mathcal{C}_V$, we have that $\rho_V(C')$ is an $\mathfrak{R}_d(\K)$-quadric and vice versa. 
Thirdly, $v=d-1$.
\end{prop}

\begin{proof} By Lemma~\ref{rho}, the map $\chi_V$ is indeed well defined since $\rho^{-1}(\widetilde{x})=\Pi_x$. We proceed in three steps.

\textbf{Part 1:} \emph{ The set $\mathcal{R}_V$ is a spread.}\\
Recall that, by Lemma~\ref{Pi_x2} and Corollary~\ref{pix2}, the set $\Pi_V^Y=\{\Pi_x^Y \mid x \in X_V\}$ is a set of $(2v+1)$-spaces pairwise intersecting each other in $V$, and, by Lemma~\ref{Y}$(iii)$, each point of $Y\setminus V$ is contained in a member of $\Pi_V^Y$. So $\{R_{\Check{x}} \mid x \in X_V\}$  indeed defines a spread $\mathcal{R}_V$  of $v$-spaces on $\widetilde{Y}$. 


For the sequel,  let $C'$ be an arbitrary member of $\mathcal{C}_V$ distinct from $C$. We know that $C$ and $C'$ share a generator $\<x_0,V\>$. So $Q$ and $Q':=\rho_V(C')$ are quadrics sharing the point $\widetilde{x}_0$. Let $x\in V\setminus \<x_0,V\>$. By Lemma~\ref{Y}$(i)$, there is a unique generator, say $\<x',V\>$ of $C'$ collinear to $\<x,V\>$, i.e., $\Check{x}=\Check{x}'$. This implies that the mapping $f:\widetilde{x}\mapsto f(\widetilde{x}):=\widetilde{x}'$ (with $f(\widetilde{x}_0)=\widetilde{x}_0$ by definition) is a projectivity. Note that the points $\widetilde{x}$ and $\widetilde{x}'$ are collinear and the line joining them intersects $R_{\check{x}}$ in some point (since $\<x,x'\>$ intersects $\Pi_x^Y=\Pi_{x'}^Y$).  Note also that, since $\widetilde{Y}$ and $F$ are complementary subspaces of $\widetilde{Y}$, $Q$ is the projection of $Q'$ from $\widetilde{Y}$ onto $F$. 

\textbf{Part 2:} \emph{There is an affine $d$-space $\alpha \subseteq \widetilde{Y}$ intersecting all transversals $\<\widetilde{x},f(\widetilde{x})\>$. The subspace $R_{\Check{x}_0}$ is the $(d-1)$-space at infinity of $\alpha$. Consequently, $v=d-1$.}\\ 
Lemma~\ref{scrolld} yields an affine $d$-space $\alpha$ intersecting each transversal $\<\widetilde{x},f(\widetilde{x})\>$ with $\widetilde{x} \in Q\setminus\{\widetilde{x}_0\}$ in a point $\widetilde{x}^\alpha$. By the same lemma, the induced map $\varphi: Q\setminus\{\widetilde{x}_0\} \rightarrow \alpha: \widetilde{x}\mapsto \widetilde{x}^\alpha$ is such that for any conic $K$ on $Q$ through $\widetilde{x}_0$, $\varphi(K\setminus\{\widetilde{x}_0\})$ is an affine line $L$ and vice versa; moreover, the induced map $\overline{\varphi}_K$ taking $\widetilde{x}\in K\setminus\{\widetilde{x}_0\}$ to $\widetilde{x}^\alpha$ and $\widetilde{x}_0$ to $\<L\>\setminus L$ preserves the cross-ratio.

We now show that $\alpha$ belongs to $\widetilde{Y}$.  Note that each line $\<\widetilde{x},f(\widetilde{x})\>$, with $\widetilde{x}\in Q\setminus\{\widetilde{x}_0\}$, is contained in $\<\widetilde{x},R_{\Check{x}}\>$ (which belongs to $\rho_V(\Pi_x))$ and as such is a singular line having a unique point in~$\widetilde{Y}$. Consequently, $\alpha\subseteq X\cup Y$, and as $|\K|>2$, $\<\alpha\>$ is a  singular $d$-space. Suppose that $\widetilde{x}^\alpha$ and ${\widetilde{z}}^\alpha$ belong to $X$, for two distinct points $\widetilde{x},\widetilde{z} \in Q\setminus\{\widetilde{x}_0\}$. By Corollary~\ref{pix2}, $\widetilde{x}^\alpha$ and ${\widetilde{z}}^\alpha$  are not collinear (since $\Pi_{\widetilde{x}} \neq \Pi_{{\widetilde{z}}}$), contradicting the fact the line $\<\widetilde{x}^\alpha,{\widetilde{z}}^\alpha\>$ is singular, as it lies in $\<\alpha\>$. Again relying on $|\K|>2$, this reveals that each line in the affine space $\alpha$ contains at least two points in $\widetilde{Y}$ and as such, $\alpha \subseteq \widetilde{Y}$. 

As a consequence, $\widetilde{x}^\alpha \in R_{\Check{x}}$. Moreover, the above implies that collinearity is a bijection between $Q\setminus \{\widetilde{x}_0\}$ and $\alpha$ and as such, each member $R_{\Check{x}}$ of $\mathcal{R}_V \setminus \{R_{\Check{x}_0}\}$ intersects $\alpha$ in precisely  $\widetilde{x}^\alpha$.
As $\mathcal{R}_V$ is a spread, also the points of $\<\alpha\>\setminus \alpha$ need to be contained in a member of $\mathcal{R}_V$ too, and the only possibility left is $\<\alpha\>\setminus \alpha \subseteq R_{\check{x}_0}$. We claim that actually $\<\alpha\> \setminus \alpha = R_{\Check{x}_0}$. Indeed, suppose for a contradiction that $\<\alpha\>\setminus \alpha \subsetneq R_{\Check{x}_0}$. Since  $R_{\Check{x}_0} \cap \alpha$ is empty, $R_{\Check{x}_0}$ is a hyperplane of $\<R_{\Check{x}_0},\alpha\>$. Any point $y \in \<R_{\Check{x}_0}, \alpha\> \setminus (R_{\Check{x}_0} \cup \alpha)$ has to be contained in $R_{\Check{x}}$ for some $\widetilde{x} \neq \widetilde{x}_0$. But then the line $\<y,\widetilde{x}^\alpha\>\subseteq R_{\Check{x}}$ has to intersect $R_{\Check{x}_0}$ in a point, whereas $R_{\Check{x}} \cap R_{\Check{x}_0} = \emptyset$. This contradiction shows the claim. As a consequence, since $\alpha$ is an affine $d$-space, $v=d-1$.

\textbf{Part 3:} \emph{$\mathcal{R}_V$ is regular; $\chi_V$ is a projectivity between $Q$ and $\mathcal{R}_V$; $Q$, $\mathcal{R}_V$ and $\chi_V$ define a regular $d$-scroll $\mathfrak{R}_d(\K)$.}\\ 
Consider three distinct members $R_{\Check{x}_1}$, $R_{\Check{x}_2}$ and $R_{\Check{x}_3}$ of $\mathcal{R}_V$. Denote by $K$ the conic $Q \cap \<\Check{x}_1,\Check{x}_2,\Check{x}_3\>$. We claim that the regulus determined by $R_{\Check{x}_1}$, $R_{\Check{x}_2}$ and $R_{\Check{x}_3}$ is $\{R_{\Check{x}} \mid \Check{x} \in K\}$ and as such belongs to $\mathcal{R}_V$, showing that the latter is regular indeed and that a regulus of it corresponds with a conic of $Q$ and vice versa.  

Let $z_1$ be any point in $R_{\Check{x}_1}$. We view $Q$ as $\rho_V(C)$. Choose auxiliary points $\widetilde{x}_0\in K\setminus\{\widetilde{x}_1,\widetilde{x}_2,\widetilde{x}_3\}$ and $\widetilde{x}'_1\in\<\widetilde{x}_1,z_1\>\setminus\{\widetilde{x}_1,z_1\}$, and denote the quadric $\rho_V(X[\widetilde{x}_0,\widetilde{x}'_1])$ by $Q''$ (the tube $X[\widetilde{x}_0,\widetilde{x}'_1]$ indeed has vertex $V$). Like before, there is a projectivity $f$ between $Q$ and $Q''$ which induces a projectivity $\overline{\varphi}_K$ between $K$ and some line $L$ in $\widetilde{Y}$ that takes each $\widetilde{x} \in K\setminus\{\widetilde{x}_0\}$ to $\<\widetilde{x},f(\widetilde{x})\> \cap L$, or equivalently, $\overline{\varphi}_K(\widetilde{x})=R_{\Check{x}} \cap L$, and which maps $\widetilde{x}_0$ to $R_{\Check{x}_0}\cap L$ (this follows from the previous paragraph).   Moreover, as $\overline{\varphi}_K(\widetilde{x}_1)=z_1$ (recall that $z_1$ is on $\<\widetilde{x}_1,\widetilde{x}'_1\>$, and $\widetilde{x}'_1=f(\widetilde{x}_1)$), the line $L$ is the unique line through $z_1$ intersecting $R_{\Check{x}_2}$ and $R_{\Check{x}_3}$. Since $z_1$ in $R_{\Check{x}_1}$ was arbitrary, the claim follows: each transversal of $R_{\Check{x}_1}$, $R_{\Check{x}_2}$ and $R_{\Check{x}_3}$ is intersected by $R_{\Check{x}}$ for each $\widetilde{x} \in K$ and no other member of $\mathcal{R}_V$.

This implies that we indeed have a regular $d$-scroll $\mathfrak{R}_d(\K)$ defined by $Q=\rho(\mathcal{C}_V)$ and $\mathcal{R}_V$. Since this is independent of $C'$, each tube $C''$ of $ \mathcal{C}_V$ is such that the quadric $\rho_V(C'')$ intersects each transversal subspace $\<\widetilde{x},R_{\Check{x}}\>$ with $x \in \mathcal{P}_V$ in a unique point. Moreover, since any two points $\widetilde{x}$ and $\widetilde{z}$ on distinct transversal subspaces determine a unique such tube $X[x,z]$ by (H2$^*$) and each two points of $\mathfrak{R}_d(\K)$ determine a unique $\mathfrak{R}_d(\K)$-quadric (cf.\ Lemma~\ref{regularscroll}), the set $\{\rho_V(C'') \mid C'' \in \mathcal{C}_V\}$ coincides with the set of $\mathfrak{R}_d(\K)$-quadrics. \end{proof}

The combination of Propositions~\ref{structureCV} and~\ref{veronese} now gives us the relation between  the point-quadric variety in $F$ and the set $Y$ of vertices. Our next aim is to show that we can choose $F$ in such a way that $F\cap X=\rho(X)$. But first we deduce something useful from the above proof.

\begin{lemma}\label{TxC}
For each point $x\in X$, $T_x \cap Y = \Pi_x^Y$. If $C^*$ is any tube whose vertex $V^*$ is not collinear to $x$, then $T_x$ and $\<C^*\>$ are complementary subspaces of $\PG(N,\K)$.
\end{lemma}
\begin{proof}
The tangent space $T_x$ is generated by all tangent spaces $T_x(C)$ where $C$ varies over the set of tubes through $x$. The vertices of such tubes are these contained in $\Pi_x$. Take such a vertex $V$. Then each tube $C_x$ through $x$ with vertex $V$ corresponds, projected from $V$,  to a quadric $Q_x$ on the scroll $\mathfrak{R}_d(\K)$. The subspace generated by all tangent spaces through $x$ at these quadrics is precisely $\<T_x(Q_x),R_x\>$ , for some fixed arbitrarily chosen quadric $Q_x$ (using the notation of the above proposition), as follows from the properties of scrolls (cf.\ last assertion of Lemma~\ref{scrolld}). We obtain that the subspace generated by the tangent spaces at $x$ of tubes through $\<x,V\>$ intersects $Y$ precisely in $\Pi_x^Y$.   Since $V$ was an arbitrary vertex collinear to $x$, we conclude that $T_x \cap Y = \Pi_x^Y$ indeed. 

Now consider the tube $C^*$ with vertex $V^*$. Since $V^*$ is not collinear to $x$, $V^*$ and $\Pi_x^Y$ are complementary subspaces of $Y$ by Lemma~\ref{Y}$(ii)$ and~$(iv)$. In the Veronese variety $(\rho(X),\rho(\Xi))$, the point $\rho(x)$ is not contained in $\rho(C^*)$ (since $x$ is not collinear to $V^*$), so  $\rho(T_x)=T_{\rho(x)}$ and $\rho(C^*)$ are also complementary subspaces by the properties of Veronese varieties (this can be verified algebraically but it has also been proven in Proposition 4.5 of \cite{Kra-Sch-Mal:15}). Since $T_x \cap Y$ and $\<C^*\> \cap Y$ are complementary subspaces of $Y$ and since the projections $\rho(T_x)$ and $\rho(\<C^*\>)$ from $Y$ onto $F$ are complementary in $F$, we obtain that $T_x$ and $\<C^*\>$ are complementary in $\<Y,F\>=\PG(N,\K)$.  
\end{proof}

\begin{lemma}\label{Fcut}
There exists a subspace $F^*$ of $\PG(N,\K)$ complementary to $Y$ such that the projection of $X$ from $Y$ onto $F^*$ is precisely the intersection of $F^*$ with $X$.    
\end{lemma}

\begin{proof}
As before, we denote the projection operator from $Y$ onto $F$ by $\rho$ (and $F$ is an arbitrary subspace of $\PG(N,\K)$ complementary to $Y$). Let $C_1,C_2,C_3$ be three tubes of $X$ such that $\rho(C_1)$, $\rho(C_2)$ and $\rho(C_3)$ correspond to the sides of a triangle in the projective plane $(\rho(X),\rho(\Xi))$. Let $x_i$ be the unique intersection point $C_j\cap C_k$, for all $\{i,j,k\}=\{1,2,3\}$ and denote the vertex of $C_i$ by $V_i$, $i=1,2,3$. In $\<C_i\>$, we choose an arbitrary subspace $W_i$ containing $\{x_j,x_k\}$ complementary to $V_i$,  with $\{i,j,k\}=\{1,2,3\}$.

\textit{Claim: $\<W_1,W_2,W_3\>$ and $Y$ are complementary subspaces of $\PG(N,\K)$.}\\
Firstly, $\<W_1,W_2,W_3,Y\>= \PG(N,\K)$, since $\<F,Y\>=\PG(N,\K)$ and $F$ is generated by the projections of $W_1,W_2,W_3$. If $d$ is finite, a dimension argument now readily shows that $\<W_1,W_2,W_3\> \cap Y$ is empty, but since $d = \infty$ is possible, we need a more general argument. First, we show that $\<W_2,W_3\>\cap Y=\emptyset$. Indeed, assume for a contradiction that $p\in \<W_2,W_3\>\cap Y$. Since $p\notin W_2\cup W_3$, this implies that $\<p,W_2\>\cap W_3$ contains a line $L$. Since $p\notin L$, the projection of $L$ under $\rho$ is contained in $\rho(W_2)\cap\rho(W_3)$,  a contradiction to (MM2$^*$) proved in Proposition~\ref{veronese}. 
Next, we show that $\<W_1,W_2,W_3\>\cap Y=\emptyset$.  Assume for a contradiction that $p\in\<W_1,W_2,W_3\>\cap Y$. 
 Since $p\notin W_1 \cup \<W_2,W_3\>$, the subspace $\<p,W_1\>$ intersects $\<W_2,W_3\>$ in at least a plane. But then, as $p \in Y$, the spaces $\rho(W_1)$ and $\<\rho(W_2),\rho(W_3)\>$ also share at least a plane, a contradiction. This shows the claim.
 
Put $F^*:=\<W_1,W_2,W_3\>$ and denote the projection of $X$ from $Y$ onto $F^*$ by $\rho^*$ (this projection makes sense by the above claim). If for each $x\in X$, the intersection $\Pi_x \cap F^*$ is a point of $X$, say $p^*(x)$, then $\rho^*(x)=p^*(x)$ and hence $F^* \cap X$ is isomorphic to $\mathcal{V}_2(\K,\B)$ (by Proposition~\ref{veronese} and with the same notation).  
 
\textit{Claim: We can choose $W_1$, $W_2$ and $W_3$ such that $\Pi_x \cap F^*$ is non-empty for each point $x\in X$; equivalently, $\rho^*(x)\in X$, for all $x\in X$.}\\
We keep the points $x_1,x_2,x_3$ and the subspace $W_2$ as above; and we will determine $W_1$ and $W_3$ in such a way that, for each pair of points $c_1\in (W_1 \cap X)\setminus \{x_3,x_2\}$ and $c_2 \in (W_2 \cap X)\setminus\{x_3,x_1\}$ holds that $[c_1,c_2] \cap C_3 \in W_3$. To that end, take a point $x'_1$ on $C_1 \setminus (\<x_3,V\> \cup \<x_2,V\>)$ and a point $x'_2$ on $W_2\cap X \setminus\{x_1,x_3\}$. We define $W_3$ as $\<[x'_1,c_2] \cap C_3 \mid c_2 \in W_2\cap X\>$ and $W_1$ as $\<[x'_2,c_3] \cap C_1 \mid c_3 \in W_3 \cap X\>$. We first show that $W_3$ is indeed a subspace of $\<C_3\>$ complementary to $V_3$; and in exactly the same way, then also $W_1$ is a subspace of $\<C_1\>$ complementary to $V_1$.

Consider the projection of $X \cup Y$ from $T_{x'_1}$ onto $\<C_3\>$ (by Lemma~\ref{TxC}, these are complementary subspaces). Note that, for each tube $C$ through $x'_1$, $C$ is mapped to the unique point $C\cap C_3$ since $C$ shares the hyperplane $T_{x'_1}(C)$ with $T_{x'_1}$. This means that each point $x'_2$ of $C_2$ is mapped to $[x'_1,x'_2] \cap C_3$. Moreover, the vertex $V_2$ of $C_2$ is mapped to $V_3$ since $T_{x'_1} \cap Y = \Pi_{x'_1}^Y$ and $V_3$ are complementary subspaces of $Y$.  As such, the map $C_2 \rightarrow C_3: x'_2 \mapsto [x'_1,x'_2]\cap C_3$ is the restriction of a projection that takes $W_2$ to $W_3$ (by definition of the latter) and $V_2$ to $V_3$. Since $W_2$ and $V_2$ are complementary in $\<C_2\>$, the same holds for their images $W_3$ and $V_3$ in $\<C_3\>$. Note also that the points $x_1$ and $x_2$ are fixed, so $W_3$ contains these; likewise, $W_1$ contains $x_2$ and $x_3$.  By definition of $W_1$, also  each tube $[c_1,x'_2]$ with $c_1 \in (W_1\cap X)$ intersects $C_3$ in a point of $W_3$. 

Now let $c_1\in (W_1 \cap X)\setminus \{x_3,x_2\}$ and $c_2 \in (W_2 \cap X)\setminus\{x_3,x_1\}$ be arbitrary. If $c_1=x'_1$ or $c_2=x'_2$ then, by definition, $[c_1,c_2] \cap C_3 \in W_3$, so suppose $c_1 \neq x'_1$ and $c_2 \neq x'_2$. Then the four points $x'_1,c_1,x'_2,c_2$ determine a unique $\K$-subplane $\pi$ of $(\rho^*(X),\rho^*(\Xi))$, which on $F^*$ corresponds to a copy $\cV$ of $\cV_2(\K,\K)$ (see Section 5.2 of \cite{Kra-Sch-Mal:15}). Let $c_3$, $c'_3$ and $c''_3$ denote the points of $C_3$ obtained by the intersection with $[x'_1,x'_2]$, $[x'_1,c_2]$ and $[c_1,x'_2]$, respectively. Then these belong to a conic $C_3$ on $W_3$ by the above, and moreover, this conic belongs to $\cV$. In $\cV$, the conic $C$ determined by $c_1$ and $c_2$ (which is part of the tube $[c_1,c_2]$) also intersects $C_3$ in a point. As such, we obtain that $[c_1,c_2] \cap C_3 = C \cap C_3$ belongs to $W_3$ indeed. 

Finally, we show that with these choices of $W_1$, $W_2$ and $W_3$, the claim holds. Take any point $x\in X$. If $\rho^*(x)\in W_i \cap X$ for some $i\in\{1,2,3\}$, then of course $\rho^*(x)\in X$. So assume $\rho^*(x)\notin W_1\cup W_2\cup W_3$. We consider the $\K$-subplane $\pi^*$ of $(\rho^*(X),\rho^*(\Xi))$ determined by the points $\rho^*(x_1),\rho^*(x_2),\rho^*(x_3),\rho^*(x)$ which in $F^*$ gives, as above, a copy $\cV^*$ of $\cV_2(\K,\K)$. Now, inside $\cV^*$, $\rho^*(x)$ lies on some conic $C_x^*$ intersecting $W_1\cap X$, $W_2\cap X$ and $W_3 \cap X$ in three distinct points, say $c_1,c_2,c_3$, respectively. Now, the points $c_1,c_2,c_3$ belong to $X$ and $[c_1,c_2] \cap C_3 = c_3$ by our choice of $W_1$ and $W_3$. As such,  $\rho^*(x)\in C_x^*=\<c_1,c_2,c_3\> \cap X$. The claim follows, ending the proof.  \end{proof}

From now on we assume that $\rho$ has target subspace $F$ such that $F\cap X=\rho(X)$. We now endow $Y$ with the following natural structure and deduce some more properties of it.

\begin{defi}[The point-line geometry $\mathbb{P}_Y$] \emph{Let $\mathcal{P}_Y=\{V \mid V \text{ is the vertex of a tube } C\}$ and $\mathcal{L}_Y=\{\Pi_x^Y \mid x \in X\}$ and let $\mathbb{P}_Y$ denote the point-line geometry $(\mathcal{P}_Y,\mathcal{L}_Y)$ with containment made symmetric as incidence relation. Its dual, $(\mathcal{L}_Y,\mathcal{P}_Y)$ is denoted by $\mathbb{P}^*_Y$.}
\end{defi}

\begin{lemma}\label{projective}
The point-line geometry $\mathbb{P}_Y$ has the following properties:
\begin{compactenum}[$(i)$]
\item For each element of $\mathcal{L}_Y$, all members of $\mathcal{P}_Y$ not disjoint with it, are entirely contained in it and they form a regular spread. In particular, $\mathcal{P}_Y$ is a regular spread of $Y$;
\item the point-line geometries $\mathbb{P}_Y^*$ and $(\rho(X),\rho(\Xi))$ are isomorphic projective planes;
\item the projective plane $\mathbb{P}^*_Y$ is desarguesian. 
\end{compactenum}
Moreover, 
\begin{compactenum}
\item[$(iv)$] the connection map $\chi: (\rho(X),\rho(\Xi)) \rightarrow \mathbb{P}^*_Y: x \mapsto \Pi_x^Y$  is a projectivity;
\item[$(v)$] $X$ is the union over $x\in \rho(X)$ of all subspaces $\<x,\chi(x)\>$ and each member $\xi \in \Xi$ with vertex $V$ is such that $\rho_V(X(\xi))$ is a $\mathfrak{R}_d(\K)$-quadric of the regular $d$-scroll $\mathfrak{R}_d(\K)$ defined by the regular spread $\mathcal{R}_V$, the quadric $\rho(\mathcal{C}_V)$ and the projectivity $\chi_V$, and vice versa.
\item[$(vi)$] all $(d,v)$-tubes entirely contained in $X$ are induced by the members of $\Xi$,
\item[$(vii)$] $(X,\Xi)$ is projectively unique if it exists. 
\end{compactenum}
\end{lemma}

\begin{proof}
$(i)$ Let $\Pi_z^Y$ be an arbitrary member of $\mathcal{L}_Y$ and take a vertex  $V$ not collinear to $z$ (which exists by ($\mathsf{V}$)). Then $\Pi_z^Y$ is complementary to $V$ in $Y$ by Lemma~\ref{Y}$(ii)$ 
and hence we can identify the projection $\widetilde{Y}$ of $Y$ from $V$  with $\Pi_z^Y$, cf.\ Proposition~\ref{structureCV}. This proposition then implies that, for each point $x$ collinear to $V$ (and hence not collinear to $z$), the vertices of the tubes $[x,z]$ (i.e., the $v$-spaces $\Pi_z^Y \cap \Pi_x^Y$) form a  regular spread of $\Pi_z^Y$. Since each pair of vertices is disjoint by ($\mathsf{V}$), all other elements of $\mathcal{P}_Y$ are disjoint from $\Pi_z^Y$. We conclude that the elements of $\mathcal{P}_Y$ having a non-trivial intersection with $\Pi_z^Y$ are contained in it and form a regular spread of it indeed.  In order for $\mathcal{P}_Y$ to be a regular spread of $Y$, we need that each two elements $V_1$ and $V_2$ of $\mathcal{P}_Y$ induce a regular spread on $\<V_1,V_2\>$, and they do:  take two tubes $C_1$ and $C_2$ through $V_1$ and $V_2$, respectively, and let $z$ the unique intersection point of $C_1$ and $C_2$ (which exists by (H2) and is unique by  ($\mathsf{V}$)), then $V_1$ and $V_2$ span the subspace $\Pi_z^Y$ and hence the assertion follows from what we deduced just before.


\qquad $(ii)$ Let $\Pi_x^Y$ be an arbitrary element of $\mathcal{L}_Y$. By Corollary~\ref{pix2}, $
\Pi_x$ is the set of points of $X$ collinear to $\Pi_x^Y$. Also, Lemma~\ref{rho} implies that $\Pi_x$ is the set of points of $X$ mapped by $\rho$ onto $\rho(x)$. Hence $\psi(\Pi_x^Y):= \rho(x)$ defines a bijective correspondence between $\mathcal{L}_Y$ and $\rho(X)$. Now consider the set of elements of $\mathcal{L}_Y$ incident with a fixed element of $\mathcal{P}_Y$, i.e., all subspaces $\Pi_x^Y$ through to a certain vertex $V$, which means all subspaces $\Pi_x^Y$ with $x$ collinear to $V$. Then $\{\rho(x) \mid x \perp V\}=\rho(C)$ for any tube $C$ through $V$ by Proposition~\ref{Y}$(i)$ and Lemma~\ref{rho}; even stronger: each member $\Pi_x^Y$ of $\mathcal{L}_Y$ through $V$ corresponds to a unique point of $\rho(C)$ and vice versa. Hence if we set $\psi(V):=\rho(C)$, then  $\psi:(\mathcal{L}_Y,\mathcal{P}_Y) \rightarrow (\rho(X),\rho(\Xi))$ is a collineation.
As $\psi(\mathbb{P}_Y^*)$ is a projective plane by Proposition~\ref{veronese}, so is~$\mathbb{P}^*_Y$.

\qquad $(iii)$ Since the Desargues theorem is self-dual, it is equivalent to show that $\mathbb{P}_Y$ is desarguesian. Let $\Delta$ be a triangle with vertices $V_1$, $V_2$ and $V_3$ and $\Delta'$ a triangle with vertices $V'_1$, $V'_2$ and $V'_3$. Suppose $\Delta$ and $\Delta'$ are  in central perspective from $V$. We claim that they are in axial perspective too. 

Take any point $v \in V$. Then there are unique lines $L_i$, $i=1,2,3$ through $v$ such that $L_i \cap V_i$ is a point $v_i$ and $L_i \cap V'_i$ is a point $v'_i$. Then the triangles $v_1v_2v_3$ and $v'_1v'_2v'_3$ are centrally in perspective from $v$. Since $Y$ is a subspace of $\PG(N,\K)$, it is desarguesian, so there is an axis $L$, i.e., each intersection point $p_{ij}:=v_iv_j \cap v'_iv'_j$,  with $i, j \in \{1,2,3\}$, $i\neq j$, lies on this line $L$. Let $V_{ij}$ be the unique members of the spread containing the points $p_{ij}$, respectively. The line $L$ is entirely contained in $\<V_{13},V_{23}\>$, and since $V_{12}$ shares a point with $L$, item $(i)$ of this lemma implies that $V_{12} \subseteq \<V_{13},V_{23}\>$. This shows the claim.

In a completely similar fashion, one can show that triangles that are  in axial  perspective, are also in central perspective. This shows that $\mathbb{P}^*_Y$ is desarguesian.

\qquad $(iv)$
Clearly, $\chi$ is the inverse image of the above defined collineation $\psi$, and as such it is a collineation.  We now show its linearity. To that end, let $Q$ be a quadric of $\rho(\Xi)$ and let $C$ be a tube with $\rho(C)=Q$. If $V$ is the vertex of $C$, then the restriction of $\chi$ to the points of $Q$ is given by $\chi_V$, with the notation of Proposition~\ref{structureCV}. According to this proposition, the map $\chi_V$ preserves the cross-ratio and hence so does $\chi$. We conclude that $\chi$ is a linear collineation, i.e., a projectivity.

\qquad $(v)$ For each point $x\in X$, we have that $x$ belongs to $\Pi_x=\<\rho(x),\chi(\rho(x))\>\setminus \chi(\rho(x))$ (recall $\rho(x) \in X$), showing the first part of the assertion. The second part of the assertion follows immediately from Proposition~\ref{structureCV}.

\qquad $(vi)$ Suppose $C$ is a $(d,v)$-tube not contained in a member of $\Xi$. If its vertex $V$ were not contained in $Y$, i.e., if $C$ contains a singular affine line $L$ with $\<L\> \cap Y =\emptyset$, then $\rho(L)$ is a line in $\<\rho(X)\>$ containing at least three points of $\rho(X)$ (since $|\K|>2$), contradicting the properties of ordinary Veronese varieties. Hence $V \subseteq Y$, so $\rho(C)$ is a quadric of $(\rho(X),\rho(\Xi))$. Since for an ordinary Veronese variety with $|\K|>2$, the elliptic spaces are determined by their point set, we obtain that $\rho(C)=\rho(C')$ for some tube $C'$ with $\<C'\>\in \Xi$. Let $V'$ be the vertex of $C'$. Let $x,x'\in C$ such that $\Check{x},\Check{x}'$ are two distinct points of $\rho(C)$, which are automatically non-collinear. Then $x,x'$ are non-collinear and every point of $V$ is collinear to both $x,x'$. By Corollary~\ref{pix2}, $V\subseteq V'$ and so $V=V'$ (because they have the same dimension). 
But then it follows that $\rho_V(C)$ is an $\mathfrak{R}_d(\K)$-quadric on the regular $d$-scroll $\mathfrak{R}_d(\K)$ determined by $\mathcal{R}_V$ and $\rho(C')$, and by the previous item, $\<C\>$ belongs to $\Xi$ after all. 

\qquad $(vii)$ First note that the projective plane $\mathbb{P}^*_Y$ as given above is projectively unique, and so is the Veronese variety $(\rho(X),\rho(\Xi))$. Since all projectivities from $(\rho(X),\rho(\Xi))$ to $\mathbb{P}^*_Y$ are equivalent up to a projectivity of the source geometry $(\rho(X),\rho(\Xi))$, as follows from Main Result~\ref{main2} and Proposition~\ref{actiononveronesean}, we obtain that $(X,\Xi)$ is projectively unique if it exists.
\end{proof}
\bigskip

The above lemma even allows us to exclude one of the possibilities for $d$ if $\kar\K\neq 2$.

\begin{prop}
The variety $(X,\Xi)$  is projectively equivalent to $\mathcal{V}_2(\K,\A)$ where $\A=\mathsf{CD}(\B,0)$ and $\B$ is a quadratic associative division algebra over $\K$, and $\dim_\K(\B)=d$. Hence, if $\kar\K\neq 2$, then $(X,\Xi)$ exists if and only if $d\in \{1,2,4\}$. 
\end{prop}

\begin{proof}
Assume first that $\kar\K\neq 2$. If $d\notin \{1,2,4\}$, the only remaining possibility by Proposition~\ref{veronese} is $d=8$. The same proposition, together with Lemma~\ref{projective}$(ii)$, implies that $\mathbb{P}_Y^*$ is isomorphic to $\PG(2,\mathbb{A})$, where $\mathbb{A}$ is a strictly alternative division algebra over $\K$ with $\dim_\K(\mathbb{A})=8$. But then it is impossible that $\mathbb{P}^*_Y$ is desarguesian (cf.~Lemma~\ref{projective}$(iii)$). Hence $d \neq 8$.

By Proposition~\ref{standardex}, the Veronese representations $\mathcal{V}_2(\K,\A)$ with $\A=\mathsf{CD}(\B,0)$, where $\B$ is a quadratic associative division algebra over $\B$ with $\dim_\K(\B)=d$ ($d$ possibly an infinite cardinal) are Hjelmslevean Veronese sets with $(d,d-1)$-quadrics. Since we have shown above that these are projectively unique, we conclude that $(X,\Xi)$ is projectively equivalent to $\mathcal{V}_2(\K,\A)$.
\end{proof}

This finishes the proof of  Main Theorem~\ref{main3}. 

\begin{rem} \em
Proposition~\ref{projective} shows that one can construct all points and quadrics of $\mathcal{V}_2(\K,\A)$, with $\A=\mathsf{CD}(\B,0)$ where $\B$ is a quadratic associative division algebra over $\B$, by taking a regular $(d-1)$-spread in $(3d-1)$-dimensional projective space over $\K$ together with an ordinary Veronese variety $\mathcal{V}_2(\K,\B)$ and a duality $\chi$ between these. 
\end{rem}

\bigskip
\subsection{Projective Hjelmslev planes of level 2}\label{sectionHP}

To conclude, we say  some more about $(X,\Xi)$ as an abstract point-line geometry. 

\begin{defi}[Projective Hjelmslev plane of level 2] \em An incidence structure $(\mathcal{P},\mathcal{L},I)$ is called a \emph{projective Hjelmslev plane of level $2$} if, for each two points (resp. lines), there is at least one line (resp. point) incident with it, and if there is a canonical epimorphism to a projective plane such that two points (resp. two lines) have the same image if and only if they are not incident with a unique line (resp. point). 
\end{defi}

\begin{prop} \label{chi}
 The pair $(X,\mathcal{C})$ is a \emph{projective Hjelmslev plane of level 2}. More precisely: the map $\overline{\chi} =\chi \circ \rho: X\rightarrow \mathbb{P}^*_Y: x \mapsto \Pi_x^Y$ is an epimorphism satisfying the following properties. 
\begin{compactenum}
\item[\emph{(Hj1)}] Two points $x,x'$ of $X$ are always joined by at least one member of $\mathcal{C}$; this member is unique if and only if $\overline{\chi}(x) \neq \overline{\chi}(x')$;
\item[\emph{(Hj2)}] Two members $C,C'$ of $\mathcal{C}$ always intersect in at least one point; this point is unique if and only if $\overline{\chi}(C) \neq \overline{\chi}(C')$;
\item[\emph{(Hj3)}] The inverse image under $\overline{\chi}$ of a point of $\mathbb{P}^*_Y$, endowed with all intersections with non-disjoint tubes, is an affine plane;
\item[\emph{(Hj4)}] The set of tubes contained in the inverse image under $\overline{\chi}$ of a line of $\mathbb{P}^*_Y$, endowed with all mutual intersections, is an affine plane. 
\end{compactenum}
\end{prop}

\begin{proof}
Clearly, both $\chi$ and $\rho$ are morphisms (they preserve collinearity), hence so is $\overline{\chi}$. Surjectivity follows as each point of $\mathbb{P}^*_Y$ is by definition of the form $\Pi_x^Y$.  
\begin{itemize}
\item[{(Hj1)}] By (H1), each two points $x, x'$ of $X$ are contained in a tube. This tube is unique if and only if $x$ and $x'$ are non-collinear, which is at its turn equivalent with $\Pi_x^Y \neq \Pi_{x'}^Y$ (cf.\ Corollary~\ref{pix2}). 

\item[{(Hj2)}] By Lemma~\ref{intersectionoftubes2}, two tubes $C, C'$ either intersect each other in precisely one point of $X$, or they have a generator (and hence also their vertex) in common. As $\overline{\chi}(C)$ and $\overline{\chi}(C')$ are the respective vertices of $C$ and $C'$, the property holds.

\item[{(Hj3)}] The inverse image under $\overline{\chi}$ of a point of $\mathbb{P}_Y^*$, hence of some $\Pi_x^Y$, is the affine subspace $\Pi_x$. We endow this affine subspace now with the intersections of all tubes  having their vertex in $\Pi^Y_x$, which yields singular affine $(v+1)$-spaces through each element of the spread in $\Pi_x^Y$ and each point of $X$ of $\Pi_x$. Hence we obtain the Brose-Bruck construction of an affine plane.

\item[{(Hj4)}] This follows from Corollary~\ref{dualaffineplane}, by dualising.
\end{itemize}
\end{proof}

\subsection{A link with buildings of relative type $\widetilde{\mathsf{A}}_2$}

Let $\B$ be a quadratic associative division algebra over $\K$ with standard involution $b\mapsto\overline{b}$, $b\in\B$. Let $\B((t))$ be the set of formal power series over $\B$, i.e., $$\B((t))=\left\{\sum_{i=z}^{\infty} t^ib_i\mid z\in\mathbb{Z}, b_i\in\B\right\},$$ with standard addition, and multiplication determined by  $$(t^ib_i)(t^jb_j)=\left\{\begin{array}{ll} t^{i+j}(b_ib_j),& \quad\mbox{if }i,j\in2\mathbb{Z},\\t^{i+j}(b_jb_i), &\quad\mbox{if }i\in2\mathbb{Z}+1, j\in2\mathbb{Z},\\ t^{i+j}(\overline{b}_ib_j), &\quad\mbox{if }i\in2\mathbb{Z}, j\in2\mathbb{Z}+1,\\ t^{i+j}(b_j\overline{b}_i), &\quad\mbox{if }i,j\in2\mathbb{Z}+1.\end{array}\right.$$

One easily checks that this defines an alternative multiplication that renders $\B((t))$ an alternative division ring. Note that
\begin{itemize}
\item If $\B$ is commutative and the involution trivial, then $\B((t))$ is also commutative (and associative);
\item if $\B$ is commutative and the involution is nontrivial, then $\B((t))$ is associative but not commutative;
\item if $\B$ is not commutative, then $\B((t))$ is not associative.
\end{itemize} Defining $\nu(\sum_{i=z}^{\infty} t^ib_i)=z$, if $b_z\neq 0$, we obtain a division ring with valuation $\nu$. Its valuation ring is $\B[[t]]=\left\{\sum_{i=0}^{\infty} t^ib_i\mid b_i\in\B\right\}$. The residue field $\B((t))/(t)$ is precisely $\B$. Now note that $\B[[t]]/(t^2)=\mathsf{CD}(\B,0)$. It follows from \cite{VM1,VM2} that $\B((t))$ defines an affine building $\Delta$ of type $\widetilde{\mathsf{A}}_2$ whose spherical building $\Delta_\infty$ at infinity corresponds to the projective plane over $\B((t))$. Each vertex residue is isomorphic to the building of type $\mathsf{A_2}$ associated with the projective plane $\PG(2,\B)$. More tedious to check is that for each vertex of $\Delta$, the corresponding Hjelmslev plane of level 2 (see \cite{Han-Mal:89}) is precisely $\mathsf{G}(2,\mathsf{CD}(\B,0))$. 

In each of the three cases above, $\Delta$ is the building of a simple algebraic group, and the absolute type is determined by $\Delta_\infty$. 

\begin{itemize}
\item If $\B$ is commutative and the involution trivial, then $\Delta_\infty$ is a split building, hence the absolute type of $\Delta$ is $\widetilde{\mathsf{A}}_2$;
\item if $\B$ is commutative and the involution is nontrivial, then $\B((t))$ is a quaternion algebra, $\Delta_\infty$ has absolute type $\mathsf{A}_5$ and so $\Delta$ has absolute type $\widetilde{\mathsf{A}}_5$;
\item if $\B$ is quaternion, then $\B((t))$ is a Cayley-Dickson algebra, $\Delta_\infty$ has absolute type $\mathsf{E_6}$ and hence $\Delta$ has absolute type $\widetilde{\mathsf{E}}_6$.
\end{itemize}
To obtain the Tits index, the vertex residues of $\Delta$ must coincide with the residues corresponding  to the encircled nodes on the absolute diagram. In the first case, this is just an $\mathsf{A}_2$ diagram; in the second case the residue is over a quadratic extension of the base field $\K$ and we have a quasi-split $\mathsf{A}_2\times\mathsf{A}_2$ diagram; in the last case the residue is a plane over a quaternion division algebra and hence we obtain an $\mathsf{A_5}$ diagram. This explains Table~\ref{dualnonsplit}.

\appendix

\section{Scrolls}\label{scrolls}
Recall that a \emph{normal rational curve in $\PG(m,\K)$} is given by $\{(x_0^m,x_0^{m-1}x_1,...,x_0x_1^{m-1},x_1^m) \, \mid \, (x_0,x_1) \in (\K \times \K)\setminus (0,0)\}$.

\begin{defi}\label{scroll} \em Let $\Pi_k$ and $\Pi_\ell$ be complementary subspaces of a projective space $\PG(k+\ell+1,\K)$ of respective dimensions $k$ and $\ell$. In $\Pi_k$ and $\Pi_\ell$, respectively, we consider normal rational curves $C_k$ and $C_\ell$, between which we have a bijection $\varphi$ preserving the cross-ratio (i.e., a projectivity). The union of all \emph{transversal lines} $\<p,\varphi(p)\>$ with $p \in C_k$ is called  a \emph{normal rational scroll} and is denoted by~$\mathfrak{S}_{k,\ell}(\K)$.
\end{defi}

We are particularly interested in $\mathfrak{S}_{1,2}(\K)$, which consists of lines between a normal rational curve in dimension 1 (a line) and one in dimension 2 (a conic) and as such is contained in $\PG(4,\K)$. This object is also more specifically called a \emph{normal rational cubic scroll}. Each conic on $\mathfrak{S}_{1,2}(\K)$ that intersects all transversals of $\mathfrak{S}_{1,2}$ will be called an $\mathfrak{S}_{1,2}(\K)$-conic. Some of the following property are folklore. We omit the proofs or just give a clue, since it is not the essential part of the paper. 

\begin{lemma}\label{conicscroll} Let $\mathfrak{S}=\mathfrak{S}_{1,2}(\K)$ be a normal rational cubic scroll in $\PG(4,\K)$, $|\K|>2$,  defined by the line $L$, the conic $C$ and a projectivity $\varphi: C \rightarrow L$. Firstly, given two points $p$ and $q$ on distinct transversals of $\mathfrak{S}$ and with $p, q \notin L$, there is a unique $\mathfrak{S}$-conic through $p$ and $q$. Secondly, each two $\mathfrak{S}$-conics intersect in a point of $\mathfrak{S}$. Thirdly, if $\mathfrak{S}_c$ is the set of all $\mathfrak{S}$-conics through a  point $c \in C$, then all tangent spaces through $c$ to these conics are in the plane spanned by the point $\varphi(c)$ and the tangent line through $c$ at $C$.
\end{lemma}

\begin{proof}
This is a straightforward proof if one chooses coordinates for $\PG(4,\K)$ such that the points of $C$ are given by $(x_0^2,x_0x_1,x_1^2,0,0)$, those of $L$ by $(0,0,0,x_0,x_1)$ and such that  $\varphi$ maps $(x_0^2,x_0x_1,x_1^2,0,0)$ to $(0,0,0,x_0,x_1)$
\end{proof}

Sort of conversely, two conics intersecting in a point and between which there is a projectivity (a linear collinearity, i.e., one that preserves the cross-ratio), determine a unique normal rational cubic scroll. We at once phrase this more generally for quadrics of Witt index 1 in $\PG(d+1,\K)$ with $d \geq 1$.

\begin{lemma}\label{scrolld}\label{conicscroll2}
Let $Q_1$ and $Q_2$ be two quadrics of Witt index $1$ in $\PG(d+1,\K)$ for $d\geq 1$, intersecting each other in a point $c$ and spanning $\PG(2d+2,\K)$ (i.e., $\<Q_1\> \cap \<Q_2\> = \{c\}$ too), between which there is a projectivity $\varphi: Q_1 \rightarrow Q_2$ fixing $c$. Then there is an affine $d$-space $\alpha$ intersecting all transversals $\<x,\varphi(x)\>$ for $x\in Q_1\setminus\{c\}$ and all points of $\alpha$  are on such a transversal. The induced mapping $\overline{\varphi}$ between $Q_1$ and $\<\alpha\>$ taking a point $x\in Q_1\setminus\{c\}$ to $\<x,\varphi(x)\> \cap L$ and $c$ to $\<\alpha\>\setminus \alpha$ then takes a conic of $Q_1$ to an affine line of $\alpha$ and preserves the cross-ratio. Moreover, $T_c(Q_2)$ is contained in $\<T_c(Q_1),\overline{\varphi}(c)\>$. 
\end{lemma}
\begin{proof}
Start with the case $d=1$, which can be done by choosing coordinates in $\PG(4,\K)$ in such a way that each point of $C_1$ can be written as $(x_0^2,x_0x_2,x_2^2,0,0)$; likewise each point of $C_2$ can be written as $(x_0^2,0,0,x_0x_4,x_4^2)$ and such that $\varphi((x_0^2,x_0x_2,x_2^2,0,0))=(x_0^2,0,0,x_0x_2,x_2^2)$.
This can be extended to the case where $d >1$ to obtain the affine subspace $\alpha$. The remaining part of the statement then follows by taking a conic $C_1$ on $Q_1$ through $c$ and letting $C_2$ be its image on $Q_2$, and then one can apply the result for $d=1$ on this.
\end{proof}

If $d=1$, the above implies that $Q_2$ is on the normal rational cubic scroll determined by $Q_1$, the line $\<\alpha\>$, and the induced projectivity $\overline{\varphi}$. If $d>1$, a similar statement holds for a different, generalised, type of scroll which we are about to introduce. Before doing so, we define a projectivity  between a regular $(d-1)$-spread $\mathcal{R}$ in $\PG(2d-1,\K)$ and a quadric $Q$ in $\PG(d+1,\K)$ of Witt index 1, as a bijection $\varphi$ between the elements of $\mathcal{R}$ and the points of $Q$ such that, restricted to a regulus $\mathcal{G}$ of $\mathcal{R}$, $\varphi(\mathcal{G})$ is a conic on $Q$ and $\varphi_{\mid \mathcal{G}}$ preserves the cross-ratio.

\begin{defi}\label{regscroll} \em Let $Q$ be a quadric of Witt index 1 in $\Pi \cong \PG(d+1,\K)$   and $\mathcal{R}$ a regular $(d-1)$-spread in $\Pi' \cong \PG(2d-1,\K)$, where $\Pi$ and $\Pi'$ are complementary subspaces of a projective space $\PG(3d+1,\K)$. Suppose we have a bijection $\varphi$ between $Q$ and $\mathcal{R}$ preserving the cross-ratio. Then we call the union of all transversal subspaces $\<p,\varphi(p)\>$ with $p \in Q$   a \emph{regular $d$-scroll} and denote it by~$\mathfrak{R}_d(\K)$. Each quadric of Witt index 1 intersecting each transversal subspace $\<p,\varphi(p)\>$ in point not in $\varphi(p)$ is called a \emph{$\mathfrak{R}_d(\K)$-quadric}.
\end{defi}

This regular $d$-scroll also exhibits the properties of a normal rational cubic scroll mentioned in Lemma~\ref{conicscroll}. We give a sketch of the proof.

\begin{lemma}\label{regularscroll} Let $\mathfrak{R}:=\mathfrak{R}_d(\K)$ be a regular $d$-scroll in  $\PG(3d+1,\K)$, $|\K|>2$  defined by a quadric $Q$ of Witt index 1 in a complementary $(d+1)$-space $\Pi$ of  $\PG(3d+1,\K)$, a regular  $(d-1)$-spread $\mathcal{R}$ in a $(2d-1)$-space $\Pi'$ of $\PG(3d+1,\K)$ and some projectivity $\varphi$ between $Q$ and $\mathcal{R}$. Then, given two points $p$ and $q$ on distinct transversal subspaces of $\mathfrak{R}$ and with $p, q \notin \Pi'$, there is a unique $\mathfrak{R}$-quadric through $p$ and $q$. Furthermore, each two $\mathfrak{R}$-quadrics of $\mathfrak{R}$ intersect in a unique point of $\mathfrak{R}$ not on $\mathcal{R}$.
\end{lemma}

\begin{proof}
First note that for each $\mathfrak{R}$-quadric $Q'$ through a point $c$ of $Q$, $\varphi$ induces a projectivity $\psi$ between the quadrics $Q$ and $Q'$ fixing $c$ and hence, by Lemma~\ref{scrolld}, $\psi$ extends to a projectivity $Q \rightarrow \<\alpha_c\>$ where $\alpha_c=\overline{\psi}(Q)$ is an affine $d$-space with $\overline{\psi}(c)=\varphi(c)$ its $(d-1)$-space at infinity. One can then see that $Q'$ plays the same role as $Q$ w.r.t.\ $\mathcal{R}$.

Let $c$ and $p$ be distinct points of $Q$ and take an arbitrary point $p' \in \<p,\varphi(p)\>\setminus (p \cup \varphi(p))$. Then the unique $\mathfrak{R}$-quadric through $c$ and $p'$ can be determined as the unique quadric in the inverse image of $\overline{\psi}^{-1}(\<\alpha_c\>)$ containing $p'$. 


Secondly, each two $\mathfrak{R}$-quadrics intersect non-trivially. To that end, we consider the (injective) projection $\rho$ of $\mathfrak{R}$ from $\<p,\varphi(p)\>$ onto a complimentary subspace $F$ of $\PG(3d+1,\K)$, taking $\mathfrak{R}$-quadrics to affine $d$-spaces, which intersect each other in at most one point by the above. Moreover, since, for each point $p' \in \<p,\varphi(p)\>$, its tangent space is mapped to a projective $(d-1)$-space by Lemma~\ref{scrolld}, the projections of the $\mathfrak{R}$-quadrics through some fixed point $p'$ share their $(d-1)$-space at infinity. In $F$, it readily follows that each pair of affine $d$-spaces with distinct (and hence disjoint) $(d-1)$-spaces at infinity, intersect in precisely one point. The assertion follows.
\end{proof}

\end{document}